\documentclass[10pt]{article}

\usepackage[utf8]{inputenc}
\usepackage[T1]{fontenc}
\usepackage[ngerman,english]{babel}
\usepackage{lmodern}
\usepackage{bm} 

\usepackage{url}

\usepackage{todonotes}

\usepackage[]{geometry}
\usepackage{xparse} 

\usepackage{amsmath}
\usepackage{amsthm}
\usepackage{amsfonts}
\usepackage{amssymb}
\usepackage{pict2e}

\usepackage{tikz-cd}

\usepackage{enumitem}
\usepackage{graphicx}
\usepackage{subcaption}
\usepackage{float}
\usepackage{color}
\usepackage{xcolor}

\usepackage{mathabx}
    \DeclareMathSymbol{\partial}{\mathord}{letters}{"40} 

\usepackage{mathrsfs}
\usepackage{mathtools}  
    \mathtoolsset{showonlyrefs} 

\newcommand{\customyinyang}[1][1]{%
    \begin{tikzpicture}[scale=#1*0.07]
      \draw[line width = #1*0.05mm,transform canvas={yshift=0.02cm}] (0,0) circle (1cm);
      \path[fill=black,transform canvas={yshift=0.02cm}] (90:1cm) arc (90:-90:0.5cm)
                        (0,0)    arc (90:270:0.5cm)
                        (0,-1cm) arc (-90:-270:1cm);

    \end{tikzpicture}}

\newtheorem{theorem}{Theorem}[section]
\newtheorem{prop}[theorem]{Proposition}
\newtheorem{corollary}[theorem]{Corollary}
\newtheorem{lemma}[theorem]{Lemma}
\newtheorem{definition}[theorem]{Definition}
\newtheorem{remark}[theorem]{Remark}

\DeclareMathOperator{\Span}{span}

\newcommand{\supp}{\operatorname{supp}}
\newcommand{\eps}{\varepsilon}

\newcommand\dint{\mathrm{d}}
\newcommand{\dd}{\,\mathrm{d}} 

\newcommand{\ee}{\mathrm{e}} 
\newcommand{\ii}{\mathrm{i}} 

\newcommand{\cC}{\ensuremath{\mathcal C}}

\newcommand{\N}{\mathbb{N}}
\newcommand{\Z}{\mathbb{Z}}

\newcommand{\R}{\mathbb{R}}
\newcommand{\C}{\mathbb{C}}

\newcommand{\T}{\mathbb{T}}

\newcommand{\funpool}{\mathscr{B}}
\newcommand{\testpool}{\tilde{\mathscr{B}}}

\newcommand{\n}{\ensuremath{{\N}_0}}
\newcommand{\nd}{\ensuremath{\n^d}}

\def \bn{\bar{n}}
\def \bj{\bar{j}}
\def \bh{\bar{h}}

\def \bm{\bar{m}}
\def \bk{\bar{k}}
\def \bl{\bar{l}}
\def \bb{\bar{b}}
\def \bt{\bar{t}}
\def \balpha{\bar \alpha}
\def \beps{\bar{\varepsilon}}
\def \bgamma{\bar{\gamma}}

\newcommand{\tq}{\tilde{q}}

\DeclareMathOperator{\dist}{dist}

\newcommand{\fj}{f^{\mathrm{hpc}}_{\bj}}

\NewDocumentCommand{\norm}{m}{\ensuremath{\left\lVert #1 \right\rVert}} 
\NewDocumentCommand{\abs}{m}{\ensuremath{\left\lvert #1 \right\rvert}} 
\NewDocumentCommand{\br}{m}{\ensuremath{\left\langle #1 \right\rangle}} 
\NewDocumentCommand{\dual}{m}{\ensuremath{\langle #1 \rangle_{\raisebox{-0.6ex}{\customyinyang[1]}}}}

\NewDocumentCommand{\hpc}{O{r} O{p} O{q} O{[0,1]^{d}}}{S^{#1}_{#2,#3}B_{\mathrm{hpc}}\ensuremath{\left(#4\right)}}

\title{Besov regularity of multivariate non-periodic functions \\ in terms of half-period cosine coefficients and consequences \\ for recovery and numerical integration}

\author{Martin Sch\"afer\footnote{Corresponding author. Email: martin.schaefer@mathematik.tu-chemnitz.de}
\and Tino Ullrich}
\date{Chemnitz Technical University, 09111 Chemnitz, Germany \\[2ex]
April~4, 2025}

\begin{document}

\maketitle

\vspace*{-3ex}

\begin{abstract}
\noindent
In the setting of $d$-variate periodic functions, often modelled as functions on the torus $\T^d\cong[0,1]^d$, the classical tensorized Fourier system is the
system of choice for many applications. Turning to non-periodic functions
on the unit cube $[0,1]^d$ the Fourier system is not as well-suited as exemplified by the Gibbs phenomenon at the 
boundary. 
Other systems have therefore been considered for such functions. One example is 
the half-period cosine system, which occurs naturally as the eigenfunctions of the Laplace operator under homogeneous Neumann
boundary conditions. 
We introduce and analyze
associated function spaces, $\hpc$, of dominating mixed Besov-type generalizing earlier concepts in this direction, in particular the half-period cosine spaces 
from~\cite{Dick2013,goda2019lattice}.
As a main result, we show that there is a 
natural parameter range,
where $\hpc$ coincides with the classical Besov space of dominating mixed smoothness $S^{r}_{p,q}B([0,1]^d)$. 
This finding has direct implications for different functional analytic tasks in $S^{r}_{p,q}B([0,1]^d)$. It allows to systematically transfer methods, originally taylored to the periodic domain, to the non-periodic setup. To illustrate this, we investigate half-period cosine approximation, sampling reconstruction, and tent-transformed cubature. 
Concerning cubature, for instance, we are able to reproduce the optimal convergence rate $n^{-r}(\log n)^{(d-1)(1-1/q)}$ 
for tent-transformed digital nets in the range $1\le p,q\le\infty$, $\tfrac{1}{p}<r<2$, where $n$ is the number of samples.
In our main proof we 
rely on Chui-Wang discretization of the dominating mixed
Besov space $S^{r}_{p,q}B(\R^d)$. This paper therefore also includes an extension of~\cite[Thm.~5.1]{DerUll19} to the multivariate domain.
\end{abstract}

\noindent
{\bf Keywords :} Besov space, Chui-Wang wavelets, Half-period cosine basis, Tent transform, Cubature, Recovery.\\
{\bf 2010 Mathematics Subject Classification :} 41A10, 41A25, 42C10, 42C40, 42B05, 42B35, 46E35, 65D32.

%
%
%
%
%
%


\section{Introduction}

Fourier series expansions of the form
\begin{align}\label{Fourier_expansion_1D}
f = \sum_{k\in\Z} \hat{f}(k) \exp_{k} =  \hat{f}(0) + \sum_{k\in\N}  \big(\hat{f}_{c}(k)  \cos_k +  \hat{f}_{s}(k)\sin_k\big)
\end{align}
have been a well-established and prevalent tool in mathematics
since the beginnings of Fourier analysis in the 19th century, see e.g.~\cite{Four1822}.
Their elementary building blocks, the exponentials 
and trigonometric functions, 
have proven to be well-suited for
the approximation of periodic functions $f:\R\to \C$, often modelled as functions $f:\T\to \C$ on the torus $\T$, if those are sufficiently smooth. Furthermore, the
smoothness of $f$ is handily encoded in the decay of the Fourier coefficients, 
a fact that is the foundation of the Fourier analytic characterization of a great variety of function spaces.
The tremendous success 
of the expansions~\eqref{Fourier_expansion_1D}
as a numerical tool, as pointed out in~\cite{IserNor2008}, 
can be traced back to 
mainly two features:
\begin{enumerate}
\item[(F1)]
The fast decay of the Fourier coefficients for smooth $f:\T\to \C$ leading to a fast convergence of~\eqref{Fourier_expansion_1D} to $f$.
\item[(F2)]
The stability of the expansion~\eqref{Fourier_expansion_1D} 
with respect to perturbations of the Fourier coefficients and, again for smooth $f$, the good numerical approximability of the latter.
\end{enumerate}

\noindent
As for~(F1), the Fourier coefficients decay of order $\mathcal{O}(k^{-r-1})$ 
if $f\in C^r(\T)$ and the $r$th derivative $f^{(r)}$ is of bounded variation. Under the assumption of analyticity the decay is even exponential.
As for~(F2), good stability properties of the expansion~\eqref{Fourier_expansion_1D} are, for example, given with respect to uniform convergence. Here the approximation is stable with respect to  perturbations of the coefficients in the $\ell_1$-norm. The latter can be numerically approximated via cubature with the discrete Fourier transform (DFT). If $f$ is analytic the 
deviation of the approximate values to the true coefficients decays exponentially in the number of samples.
The DFT, finally, can be efficiently computed by the fast Fourier transform (FFT). 
To recall such Fourier analytic results we refer to the standard literature (e.g.~\cite{Henrici1993,PlPoStTa2018}).

The short recapitulation above illustrates that
the numerical utility of the expansions~\eqref{Fourier_expansion_1D} rests on a number of factors,
which unfortunately disappear
when turning to non-periodic functions $f:I\to\C$
on a bounded interval $I\subset\R$. Considering for example  
$I=[0,1]$, 
the decay rate of
the Fourier coefficients breaks down to $\mathcal{O}(k^{-1})$ if $f(0)\neq f(1)$, even for
highly smooth functions such as $f\in C^\infty([0,1])$. Likewise, the approximability of the coefficients via the DFT drops sharply to a rate of only $\mathcal{O}(n^{-2})$ in the number of samples $n$.
The reason for this breakdown is the discrepancy of the function values at the endpoints of $I$, which correspond to jump discontinuities of the periodized function
\begin{align}\label{simple_periodization}
    f_{\rm per}:\R \to \C \;,\quad f_{\rm per}(x):=f(x\bmod 1) \,.
\end{align}
Gibbs' phenomenon, materializing at these jumps,
prohibits a fast convergence of~\eqref{Fourier_expansion_1D}. The slow decay of the Fourier coefficients further disrupts the correlation between the coefficients' decay rate and the smoothness of $f$, 
preventing straight-forward Fourier analytic characterization of smoothness.

Due to these deficiencies,
other approaches are necessary when dealing with non-periodic functions on compact domains. On the one hand, there are a variety of periodization strategies which aim to map functions to smooth periodic counterparts so that classical Fourier analysis can be applied efficiently. In this category fall domain transformation techniques, as analyzed e.g.\ in~\cite{PoNas20,PoNas21,PoNas22,NUUChangeVariable2015}, as well as extension methods enlarging the domain and extending the functions smoothly to the boundary, see e.g.~\cite{BOYD2010,Huy2010,AdHuy2020}. A combination of both techniques is used in~\cite{PoWei2024}. Another classical, very elegant approach is the polynomial subtraction method~\cite{kantorovich1958approximate,ROACHE1978204}, where a Hermite interpolant is subtracted from $f$ to enhance the approximation (cf.~Lanczos representation~\cite{Lyn1974}).

On the other hand, a more direct solution is to look for suitable substitutes of the Fourier system on the given domain. Examples here are the Chebyshev polynomials on $[-1,1]$ or systems of localized functions such as e.g.\ wavelet systems, see e.g.~\cite{Jia2009,DahKunUr1999}. Often such systems, like for instance the Chebyshev polynomials, are the dual counterparts of corresponding periodization procedures. The same holds true for
the half-period cosine system, which is a simple modification of the Fourier basis on $[0,1]$, see the definition in Subsection~\ref{ssec:HPC} and Section~\ref{sec:hpc_system} for the  associated periodization, and the basis for our considerations in this paper. Whereas it may not be as powerful for approximation purposes as other more sophisticated function systems, 
its conceptual proximity to the original Fourier system has many advantages. For instance, it inherits many features, 
among them the validity of Fej\'{e}r and de la Vall\'{e}e-Poussin type theorems ensuring pointwise convergence of associated expansions in a broad spectrum of situations.

\subsection{Modified Fourier systems}

The half-period cosine system
and other similarly modified Fourier systems 
have a long history in the theory of differential equations.
They naturally appear in the solution of Sturm-Liouville eigenfunction problems. The half-period cosine system, for instance, explicitly appeared as early as 1910 in Haar's dissertation~\cite{Haar1910}, which is nowadays more famous for the introduction of the Haar wavelet, as a specific solution to the Sturm-Liouville problem with homogeneous Neumann boundary conditions.
In the numerical approximation community these systems gained attention much later, notably in 2008 after Iserles and N\o{}rsett showed good numerical properties of a modified Fourier basis in $L_2([-1,1])$.
In~\cite{IserNor2008} the named authors considered the problem of approximating analytic functions $f$ on the domain $[-1,1]$, where the Fourier system
\begin{align}\label{def:Fourier_IserNor}
\mathcal{F}([-1,1]) := \big\{\tfrac{1}{\sqrt{2}}  \big\}  \cup \big\{ \cos(\pi k\cdot) \big\}_{k \in \N} \cup \:\, \big\{ \sin(\pi k\cdot ) \big\}_{k \in \N} 
\end{align}
is an orthonormal basis. They
noticed that the decay rate of the cosine coefficients $\hat{f}_{c}(k)$ and that of the sine coefficients $\hat{f}_{s}(k)$ differ by an order of $1$. Indeed, a calculation of these coefficients via partial integration reveals that the coefficients $\hat{f}_{s}(k)$ decay by an order of $1$ slower than the coefficients $\hat{f}_{c}(k)$ if $f(-1)\neq f(1)$:
\begin{align*}
			|\hat{f}_{c}(k)| &= \bigg| \frac{(-1)^k}{(\pi k)^2}\big[f^\prime(x)\big]_{x=-1}^{x=1} \:-\: \frac{1}{(\pi k)^2} \int\limits_{-1}^{1} f^{\prime\prime}(x)\cos(\pi k x) \,\dint x  \bigg| \:\:\:\in\:\: \mathcal{O}(k^{-2}) \,, 	\\
			|\hat{f}_{s}(k)| &= \bigg| \frac{(-1)^{k+1}}{\pi k}\big[f(x)\big]_{x=-1}^{x=1} \:-\: \frac{1}{\pi k} \int\limits_{-1}^{1} f^{\prime}(x)\cos(\pi k x) \,\dint x  \bigg| \:\:\:\in\:\: \mathcal{O}(k^{-1}) \,.
\end{align*}
To improve the overall rate,
Iserles and N\o{}rsett therefore suggested to replace the
sine functions in~\eqref{def:Fourier_IserNor} and came up with the modified system
\begin{align*}
  \mathcal{M}([-1,1]) := \big\{ \tfrac{1}{\sqrt{2}} \big\} \cup \big\{\cos(\pi k\cdot )\, \big\}_{k \in \N} \,\cup \:\, \big\{ \sin(\pi(k-\tfrac{1}{2})\cdot)\, \big\}_{k \in \N} \,.
\end{align*} 
By their modification, the orthonormality of the system is not disturbed. 
Moreover, the decay rate of the modified sine coefficients $\hat{f}_{s}^{\rm mod}(k)$
now matches that of the cosine coefficients as the following calculation shows,
\begin{align*}
 |\hat{f}_{s}^{\rm mod}(k)|
			&= \bigg| \frac{(-1)^{k+1}}{(\pi (k-\tfrac{1}{2}))^2}\big[ f^\prime(x)\big]_{x=-1}^{x=1}  \:-\: 
			 \frac{1}{(\pi (k-\tfrac{1}{2}))^2} \int\limits_{-1}^{1} f^{\prime\prime}(x)\sin(\pi(k-\tfrac{1}{2})x) \,\dint x \bigg| 
    \:\:\:\in\:\: \mathcal{O}(k^{-2}) \,. 
\end{align*}
Altogether, one thus gains $1$ order in the decay rate of the coefficients resulting in an order $1$ improvement in the approximation performance. 
As Iserles and N\o{}rsett themselves note, 
the system $\mathcal{M}([-1,1])$ has appeared earlier in the literature, e.g.\ in 1935
in a publication by Krein~\cite{Krein35} or in Young's proof of the Kadec $\tfrac{1}{4}$-theorem~\cite{Young80} from 1980. 
However, a systematic and comprehensive analysis was lacking,
specifically with respect to features (F1) and (F2). 
This motivated~\cite{IserNor2008}, which also served as a starting point for a series of follow-up papers~\cite{AdcockIserNor2010,IserNor2009,HuyIserNor2010a,HuyIserNor2010b}.

The main accomplishments in~\cite{IserNor2008} are the analysis of
pointwise convergence of expansions in $\mathcal{M}([-1,1])$, see also Olver~\cite{Olver2009}.
It is further shown that instead of approximating the coefficients by the DFT one can efficiently compute them via asymptotic expansions or quadrature techniques for highly oscillatory integrals, developed e.g.\ in~\cite{Iser2004,Iser2005,HuyVan2006}.
This addresses feature~(F2), where one should add that the stability properties of expansions in $\mathcal{M}([-1,1])$ obviously coincide with those of ordinary Fourier expansions.
Concerning~(F1), the coefficient decay is still rather unsatisfactory. While clearly an improvement over $\mathcal{F}([-1,1])$, it is limited to $\mathcal{O}(k^{-2})$ even for analytic $f$ if $f(-1)\neq f(1)$. 
To obtain higher rates, higher order analogues of $\mathcal{M}([-1,1])$ can be used
based on 
Birkhoff expansions. Such expansions were first considered in 1908 by Birkhoff~\cite{Birk1908} in the theory of differential equations, see also~\cite{Stone26,Naimark1968}.   
For a suitable choice of differential operator and boundary conditions the resulting system of eigenfunctions allows for orthogonal expansions with a rapid decay of the expansion coefficients. A possible choice, as e.g.\ used in~\cite{IserNor2009,Adcock2011}, are the polyharmonic operators and homogeneous Neumann boundary conditions.

\subsection{The half-period cosine system}
\label{ssec:HPC}

The half-period cosine system considered in this paper is closely related to $\mathcal{M}([-1,1])$. In the $1$-dimensional case, it consists of the functions $c_k:[0,1]\to\R$ defined by
\begin{align}\label{def:hpc_functions_1D}
c_0(\cdot):=1 \quad\text{and}\quad c_k(\cdot) := \sqrt{2}\cos(\pi k\cdot) \quad\text{for }k\in\N\,.
\end{align} 
It is a substitute for the orthonormal Fourier basis 
\begin{align*}
	\mathcal{F}([0,1]) := \big\{ 1 \big\}  \cup \big\{ \sqrt{2}\cos(2\pi k\cdot) \big\}_{k \in \N} \cup \:\, \big\{ \sqrt{2}\sin(2\pi k\cdot ) \big\}_{k \in \N} 
\end{align*}
on the interval $[0,1]$ and subsequently denoted by $\mathcal{C}([0,1])$. It differs from the Fourier basis in that respect that in $\mathcal{C}([0,1])$ the sine functions of $\mathcal{F}([0,1])$ are replaced by the half-period cosine functions $\sqrt{2}\cos(\pi k\cdot)$, $k\in\{1,3,\ldots\}$.

The relation between the two systems can be described by the transformation $f \mapsto  \sqrt{\tau_1^\prime(\cdot)}f(\tau_1(\cdot))$, which maps functions $f$ on $[-1,1]$ to functions on $[0,1]$, where $\tau_1$ is the
linear assignment
		    \begin{align}\label{eqdef:transmap}
			\tau_1: [0,1]\to[-1,1] \;,\quad\tau_1(x):=2x-1 \,.
			\end{align}
Applied to $\mathcal{M}([-1,1])$, this transformation leads to the system transformation 
		\begin{align*}
			\tfrac{1}{\sqrt{2}} \,\mapsto\, 1 \;,\quad
		\cos(\pi k\cdot) \,\mapsto\, (-1)^k \sqrt{2}  \cos(2\pi k\cdot)  \;,\quad 
		\sin(\pi(k-\tfrac{1}{2})\cdot) \,\mapsto\, (-1)^{k+1} \sqrt{2} \cos(\pi(2k-1)\cdot) \,. 
		\end{align*}
One can see that, up to the switching signs, $\mathcal{M}([-1,1])$  is transformed to $\mathcal{C}([0,1])$. The latter
can hence be considered the corresponding system on the domain $[0,1]$.
It features the same advantageous properties as $\mathcal{M}([-1,1])$.
On the downside, its approximation performance is subject to the same limitations.

Since we work on a $d$-dimensional cube $[0,1]^d$, with $d\in\N$, we need a $d$-variate version of this system. It is given by 
\begin{align}\label{def:hpc_system}
    \mathcal{C}_d:= \mathcal{C}([0,1]^d):=\big\{ c_{\bk}(\cdot) \big\}_{\bk\in\N^d_{0}} 
\end{align}
and has as elements the tensor products $c_{\bk}:[0,1]^d \to\R$
of the functions $c_k$ from~\eqref{def:hpc_functions_1D}. Indexed by multi-indices $\bk=(k_1,\ldots,k_d)\in\N^d_{0}$, those are defined as
\begin{align*}
c_{\bk}(x):= c_{k_1}(x_1)\cdots c_{k_d}(x_d) \quad\text{for } 
x=(x_1,\ldots,x_d)\in\R^d \,.
\end{align*}
More details 
of the system $\mathcal{C}_d$ are recalled and proved in Section~\ref{sec:hpc_system}. Note that in~\cite{IserNor2009}
the modified Fourier system $\mathcal{M}([-1,1])$ has also been extended to the multivariate cube $[-1,1]^d$. A short summary on this multivariate system $\mathcal{M}([-1,1]^d)$  is given by Adcock and Huybrechs in~\cite{AdcockHuy2011}. Further properties are proved by Adcock in~\cite{Adcock2010}.

\subsection{Periodization and the tent transform}
\label{ssec:per_and_tent}

An important tool in the analysis of the half-period cosine system is the so-called tent transform.
In dimension $d=1$ this is the function
\begin{align*}
\phi_1: [0,1] \to [0,1] \;,\quad \phi_1(t):= 1 - |2t -1| \,,
\end{align*}
whose graph looks like a tent. In some papers it is called baker's transform (e.g.\ in~\cite{Hick2002}) due to its stretching and folding of the interval $[0,1]$ reminiscent of the way a baker handles the bread dough. On the hypercube $[0,1]^d$ the tent transform $\phi:[0,1]^d \to [0,1]^d$ is the function that applies $\phi_1$ to each coordinate.

Its relevance in the analysis of the system~\eqref{def:hpc_system} comes from the fact that 
it connects the half-period cosine functions with the 
corresponding full-period functions. The precise relation is
\begin{align}\label{rel:tent_hpc}
c_{\bk} \circ \phi = 2^{|\bk|_0/2} \cos(2\pi k_1 \cdot)\cdots \cos(2\pi k_d \cdot)  \;,\quad \bk=(k_1,\ldots,k_d)\in\N_0^d \,,
\end{align}
where $|\bk|_0$ denotes the number of non-zero entries of $\bk$.
The transformation $f\circ \phi$ can be viewed as a periodization procedure 
that takes $f$ from the non-periodic domain $[0,1]^d$ to the periodic domain $\T^d\cong[0,1]^d$.
Comparing it to the standard periodization, given by~\eqref{simple_periodization} in the univariate case,
more regularity is preserved than simple repetitions of the function would.
The transform $\phi$ is a central element of the periodization principle concretized in Section~\ref{sec:hpc_system}, which is a formal connection between half-period cosine analysis of non-periodic functions on the cube $[0,1]^d$ and classical Fourier analysis of periodic functions on $\T^d$.

Important applications of $\phi$ have particularly been found in sampling theory, where it is used to transform node 
sets $\{x_i\}_{i\in I}\subset [0,1]^d$ to so-called tent-transformed node sets $\{\phi(x_i)\}_{i\in I}\subset [0,1]^d$.
In a non-periodic setting, using the latter often improves the performance of cubature rules or sampling reconstruction schemes. The corresponding procedures are called tent-transformed procedures. Taking the viewpoint that instead of the points the
functions are transformed, the better performance of the tent-transformed procedures can be explained by the higher regularity preserved under periodization by reflection compared to periodization by repetition. 
Hickernell, in 2002, was the first who realized the advantages of the tent transform in the analysis of cubature rules. 
In~\cite[Thm.~3]{Hick2002} he used it
as a tool to prove $\mathcal{O}(n^{-2+\varepsilon})$
convergence of tent-transformed randomly shifted lattice rules in certain Hilbert spaces, including $H^{2}_{\rm mix}([0,1]^d)$. Hereby, $n$ is the number of utilized samples. Goda, Suzuki, and Yoshiki later, in 2019, obtained the same rate~\cite[Thm.~2, Cor.~1]{goda2019lattice}
for tent-transformed lattice rules without shifting, building upon the investigations by Dick, Kuo, Nuyens, Pillichshammer, and Suryanarayana in~\cite{Dick2013,CKNS2016}.
Similarly, for tent-transformed digital nets with random digital shifts, $\mathcal{O}(n^{-2+\varepsilon})$ convergence was shown by Cristea, Dick, Leobacher, and Pillichshammer~\cite[Thm.~5]{CrisDickLeoPill2007} in 2007. Another area, where the tent transform is useful in applications, is function reconstruction and approximation from point samples.
This topic has been considered e.g.\ in~\cite{CKNS2016,KMNN2021}.

For technical reasons,
we will in the sequel in most cases parameterize the torus $\T^d$ by $[-1,1]^d$ instead of $[0,1]^d$. Instead of $\phi$, we then use the operator $\mathcal{P}$ given by
\begin{align}\label{def1:periodization_reflection}
\mathcal{P}f(x) := f(|x_1|,\ldots,|x_d|)  \;,\quad x=(x_1,\ldots,x_d)\in[-1,1]^d \,.
\end{align}
Like $\phi$, it shall be viewed as an operator that periodizes functions $f:[0,1]^d\to\C$ by mapping them to the torus $\T^d\cong[-1,1]^d$. It reflects them along the coordinate axes and
is related to the tent transform by the equality
\begin{align}\label{rel_phi_P}
f  \circ (1-\phi) = (\mathcal{P}f) \circ \tau \,, 
\end{align}
where $\tau:=\tau_1 \otimes \ldots \otimes \tau_1$ is the $d$-variate tensorized version of $\tau_1$ from~\eqref{eqdef:transmap}.

\subsection{Half-period cosine function spaces}

Half-period cosine coefficients can be used to define function classes based on the coefficients' properties,
similar as Fourier 
coefficients.
For $f\in L_1([0,1]^d)$ the half-period cosine coefficients $\hat{f}_{\rm hpc}(\bk)$ are calculated by the formula
\begin{align}\label{eqdef:hpc_coeff}
\hat{f}_{\rm hpc}(\bk) 
:= \int_{[0,1]^d} f(x) c_{\bk}(x) \,\dint x \,.
\end{align}  
A typical criterion is their decay and summability behaviour.
In~\cite{Dick2013,CKNS2016,goda2019lattice} the $d$-variate half-period cosine space
\begin{align*}
\operatorname{C}_{d,r,\bgamma} :=\Big\{ f\in L_2([0,1]^d) ~:~ 
 \|f\|^2_{\mathcal{\operatorname{C}}_{d,r,\bgamma}} := \sum_{\bk\in\N_0^d} |\hat{f}_{\rm hpc}(\bk)|^2  {\bf w}_{r,\bgamma}(\bk) < \infty \Big\}
\end{align*}
is defined for parameters $r>1/2$ and $\bgamma=(\gamma_1,\ldots,\gamma_d)>0$
with weights of the form
\[
{\bf w}_{r,\bgamma}(\bk) := {\rm w}_{r,\gamma_1}(k_1)\cdots  {\rm w}_{r,\gamma_d}(k_d)  \;,\quad {\rm w}_{r,\gamma_j}(k) :=\begin{cases}  1 \quad&,\,k=0\,, \\
|k|^{2r}/\gamma_j &,\, k\in\Z\backslash\{0\}\,.
\end{cases}
\]
It is related to a corresponding class of periodic functions based on classical Fourier coefficients, namely
\begin{align*}
\operatorname{E}_{d,r,\bgamma} :=\Big\{ f\in L_2([0,1]^d) ~:~ 
 \|f\|^2_{\operatorname{E}_{d,r,\bgamma}} := \sum_{\bk\in\Z^d} |\hat{f}(\bk)|^2  {\bf w}_{r,\bgamma}(\bk)< \infty \Big\} 
\end{align*}
with
\begin{align*}
\hat{f}(\bk) := \int_{[0,1]^d} f(x) \exp(-2\pi\ii \bk\cdot x) \,\dint x \,.
\end{align*} 
The latter
is called weighted Korobov space in~\cite{Dick2013,CKNS2016,goda2019lattice}. 
Noting that $(1 + |k|)/2 \le \max\{1,|k|\}\le 1 + |k|$ for all $k\in\Z$ and hence, for $j\in\{1,\ldots,d\}$,
\[
2^{-2r}\max\{1,\gamma_j\}^{-1} (1+|k|)^{2r} \le {\rm w}_{r,\gamma_j}(k)  \le \min\{1,\gamma_j\}^{-1} (1+|k|)^{2r} \,,
\]
this space is equivalent to the periodic
Sobolev space $H^{r}_{\rm mix}(\T^d):=S^{r}_{2,2}B(\T^d)$ (see Definition~\ref{def:PerBesov}) of dominating mixed smoothness on the torus $\T^d$. The dimension weights $\bgamma$ do not change $\operatorname{E}_{d,r,\bgamma}$ in terms of equivalence. 

For integer $r\in\N$, it is analyzed in~\cite{Dick2013,goda2019lattice} how the spaces $\operatorname{C}_{d,r,\bgamma}$ and $\operatorname{E}_{d,r,\bgamma}$, as well as their sum $\operatorname{C}_{d,r,\bgamma}+\operatorname{E}_{d,r,\bgamma}$, embed into the non-periodic Sobolev space $H^{r}_{\rm mix}([0,1]^d):=S^{r}_{2,2}B([0,1]^d)$ (see Definition~\ref{def:BesovDomain}). Due to $H^{r}_{\rm mix}(\T^d) \subset H^{r}_{\rm mix}([0,1]^d)$, it clearly holds $\operatorname{E}_{d,r,\bgamma} \subset H^{r}_{\rm mix}([0,1]^d)$ for $r\in\N$, which is the first observation by Dick, Nuyens, and Pillichshammer in~\cite{Dick2013}. For the half-period cosine space $\operatorname{C}_{d,r,\bgamma}$, these authors then show 
the strict embedding $\operatorname{C}_{d,r,\bgamma}\subsetneq H^{r}_{\rm mix}([0,1]^d)$ for integer $r>1$ and the equivalence $\operatorname{C}_{d,r,\bgamma}\asymp H^{r}_{\rm mix}([0,1]^d)$ for $r=1$.
Their results are summarized in~\cite[Thm.~1]{Dick2013}. Goda, Suzuki, and Yoshiki refine this analysis in~\cite{goda2019lattice} and show in detail how $\operatorname{C}_{d,r,\bgamma}$ and the sum $\operatorname{C}_{d,r,\bgamma}+\operatorname{E}_{d,r,\bgamma}$ are located in $H^{r}_{\rm mix}([0,1]^d)$. In particular, they prove $\operatorname{C}_{d,r,\bgamma}+\operatorname{E}_{d,r,\bgamma}\subsetneq H^{r}_{\rm mix}([0,1]^d)$ for integer $r>1$, which was left open in~\cite{Dick2013}. A summary of their results is given in~\cite[Thm.~1]{goda2019lattice}.

The idea to use the coefficients~\eqref{eqdef:hpc_coeff} to define function classes is universal and not restricted to the class $\operatorname{C}_{d,r,\bgamma}$.
A different kind of half-period cosine
space, but in the same spirit as $\operatorname{C}_{d,r,\bgamma}$, has been considered recently in~\cite{KMNN2021}. 
It contains all functions on $[0,1]^d$ with absolutely convergent half-period cosine series. 
The main motivation for 
such spaces are their natural properties, for instance, 
with respect to tent-transformed lattice rules or half-period cosine approximation.
A main result of~\cite{Dick2013} is the proof of the existence of lattice rules for tent-transformed
QMC integration in $\operatorname{C}_{d,r,\bgamma}$ with a quasi-optimal rate of $\mathcal{O}(n^{-r+\varepsilon})$, where $\varepsilon>0$ can be
arbitrarily small.

\subsection{Our contribution}

We pick up the ideas from~\cite{Dick2013,CKNS2016,goda2019lattice}
to define half-period cosine
spaces and introduce and systematically analyze a corresponding Besov-type scale, denoted by $\hpc$.
This scale includes the class $\operatorname{C}_{d,r,\bgamma}$ recalled in the previous subsection. 
It holds $\operatorname{C}_{d,r,\bgamma}\asymp\hpc[r][2][2][[0,1]^d]$ in the sense of equivalent norms. The periodic counterparts 
of $\hpc$ are the spaces $S^{r}_{p,q}B(\T^d)$
of dominating mixed smoothness on the torus $\T^d$. They correspond to $\operatorname{E}_{d,r,\bgamma}$ from the previous subsection in the Hilbert case $p=q=2$.

In the first main part of this article, \textbf{Section~\ref{sec:func_spaces}}, we give the definition
of $\hpc$ in Definition~\ref{d1}. Hereby the complete range $0<p,q\le\infty$ is covered, 
but we restrict the regularity to $r>\sigma_p$. This shall guarantee the embedding into proper function spaces and avoids technical difficulties in connection with the handling of distributions. For most practical applications, in sampling theory for instance, this restriction is not severe and even natural. 
The main result of this section, \textbf{Theorem~\ref{thm:relation_to_even_periodic}}, builds a bridge from $\hpc$ to periodic functions given by the bounded isomorphism
\begin{align*}
	\mathcal{P}:\; \hpc \cong S^r_{p,q}B(\T^d)_{\rm even} \quad\text{if }\sigma_p:=\max\{0,\tfrac{1}{p}-1\}<r \,.
\end{align*}
The periodization $\mathcal{P}$ is the operator from~\eqref{def1:periodization_reflection} and $S^r_{p,q}B(\T^d)_{\rm even}$ the subspace of $d$-fold even
functions in $S^r_{p,q}B(\T^d)$. The notion of $d$-fold even
function is introduced in Definition~\ref{def:nfoldeven}.
Theorem~\ref{thm:relation_to_even_periodic} is a periodization principle, which 
provides a useful translation mechanism from the non-periodic to the periodic domain.  
A notable consequence is that tent-transformed procedures in $\hpc$ exhibit the same performance as the original procedures in $S^r_{p,q}B(\T^d)_{\rm even}$. An upper bound for such procedures is thus always the original performance in $S^r_{p,q}B(\T^d)$, a
fact that was already observed in~\cite{CKNS2016}.

In \textbf{Subsection~\ref{ssec:app_theo_id}},
we apply the periodization principle to half-period cosine approximation, sampling reconstruction, and tent-transformed cubature. 
Using tent-transformed digital nets with $n$ integration nodes, we deduce the 
optimal convergence rate $n^{-r}(\log n)^{(d-1)(1-1/q)}$ for cubature in $S^{r}_{p,q}B([0,1]^d)$, provided the parameters are in the range $1\le p,q\le\infty$, $\tfrac{1}{p}<r<2$. As a special case, this implies the rate $n^{-r}(\log n)^{(d-1)/2}$ in the non-periodic Sobolev space $H^{r}_{\rm mix}([0,1]^d)$ if $\tfrac{1}{p}<r<2$. In \textbf{Theorem~\ref{thm:least_squares}}, we further analyze weighted least-squares approximation in $S^{3/2}_{2,1}B([0,1]^d)$ from $n$ point samples. Here we obtain the asymptotic approximation rate $n^{-3/2}(\log n)^{2(d-1)}$ for reconstruction in the half-period cosine basis. Due to the embedding $H^{3/2+\varepsilon}_{\rm mix}([0,1]^d)\hookrightarrow S^{3/2}_{2,1}B([0,1]^d)$ for $\varepsilon>0$, this approximation rate in particular holds true for functions $f\in H^{r}_{\rm mix}([0,1]^d)$ with $r>\tfrac{3}{2}$.

For the deduction of the above results,
in addition to the periodization principle, we use embeddings of the spaces  $S^{r}_{p,q}B([0,1]^d)$ into the half-period cosine scale $\hpc$. These embedding relations are analyzed in 
\textbf{Section~\ref{sec:main}}, the second main part of the paper.
Let us give here a short summary of the relevant results. 
According to \textbf{Theorem~\ref{thm:embedding_hpc}},   
we have for all $0<p,q\le\infty$ and $r>\sigma_p$ the continuous embedding
 \begin{align*}
           \operatorname{Id}:\;  \hpc \hookrightarrow  S^{r}_{p,q}B([0,1]^d) \,,
\end{align*}
where $\operatorname{Id}$ is the identity operator.
In \textbf{Theorem~\ref{thm:identification_hpc}} a partial converse is shown, namely
\begin{align*}
            \operatorname{Id}:\;    S^{r}_{p,q}B([0,1]^d)  \hookrightarrow \hpc \quad \text{in the restricted range}\quad \sigma_p<r<\min\{1+\tfrac{1}{p},2\}\,.
\end{align*}
Both theorems together establish a range of equivalence that substantially extends the earlier findings in~\cite[Thm.~1]{Dick2013}, where the authors restrict to integer smoothness $r\in\N$. 
A relevant endpoint result for $r=\min\{1+\tfrac{1}{p},2\}$ is obtained in \textbf{Theorem~\ref{thm:endpoint_result}}. It is the embedding
\begin{align*}
            \operatorname{Id}:\;    S^{r}_{p,\min\{p,1\}}B([0,1]^d) \hookrightarrow \hpc[r][p][\infty][[0,1]^d] \quad\text{if $\tfrac{1}{3}<p<1$ or $1< p\le\infty$}\,, 
\end{align*}
where the case $p=1$ remains open due to the limitations of our proof technique.

The last part of the paper, \textbf{Section~\ref{sec:CW_char}}, is a little 
detached from the rest.
It stands on its own and is concerned with Chui-Wang wavelet analysis, which we need for our main proof in Section~\ref{sec:main}. 
The central achievement, \textbf{Theorem~\ref{thm:CW_characterization}}, yields Chui-Wang characterizations for $S^{r}_{p,q}B(\R^d)$ in a certain parameter range.
Such characterizations have 
been given
in~\cite[Thm.~5.1]{DerUll19} for the $1$-dimensional case, however with an incorrect statement in the boundary case $p=\infty,q<\infty$.
We give a more thorough analysis here and handle the boundary cases with special care.

\subsection{Basic notation}

The symbols $\N$, $\Z$, $\R$, $\C$ have their usual meaning. Further, we use $\N_0:=\N\cup\{0\}$, $\N_{-1}:=\N_0\cup\{-1\}$, $\R_{\ge0}:=[0,\infty)$, and $\R_{>0}:=(0,\infty)$. For $d\in\N$, we put $[d]:=\{1,\ldots,d\}$.
Multi-indices $\balpha=(\alpha_1,\ldots,\alpha_d)\in\Z^d$
are denoted with a bar on top to distinguish them from scalar indices. For such multi-indices we define $b^{\balpha}:=b^{\alpha_1}\cdots b^{\alpha_d}$, where $b>0$. The notation $(x)_\pm$ stands for either $\max\{x,0\}$ or $\min\{x,0\}$ in case $x\in\R$. 
If $x=(x_1,\ldots,x_d)\in\R^d$ it stands for the vector where $(\cdot)_\pm$ is applied to each component. Often the bracket in $(x)_\pm$ is omitted for readability and the $\pm$ may be put at the top or bottom of the symbol for convenience.
With $|x|_p$, where $0<p\leq \infty$, we denote the standard $p$-(quasi-)norm of $x$.
The quantity $|x|_0$ gives the number of non-zero entries, $|x|$ denotes the vector $(|x_1|,\ldots,|x_d|)$ of the absolute coordinate values. 
By $x=(x_1,\ldots,x_d)>0$
we mean that each coordinate is positive.
The Euclidean inner product of $x$ with another vector $y$ is denoted by $x\cdot y$.
For two sequences $\{a_n\}_n$, $\{b_n\}_n\subset\R$ we write
$a_n \lesssim b_n$ if $a_n \leq c\,b_n$ for a fixed constant $c>0$ and all $n$. We write $a_n \asymp b_n$ if $a_n \lesssim b_n$ and $b_n\lesssim a_n$.

Throughout the paper, we will extensively use the bracket notation
\begin{align}\label{eqdef:integration_prod}
\langle f,g\rangle_{D} 
:= \int_{D} f(x) g(x) \,\dint x \,.
\end{align} 
for the integral with respect to the Lebesgue measure over a domain $D$, where $D$ is either the $d$-torus  $\T^d$ or $D\subset\R^d$. 
Since $D=[0,1]^d$ is our 
standard domain, we further abbreviate $\langle f,g\rangle := \langle f,g\rangle_{[0,1]^d}$. 
The torus $\T^d$ is usually represented in the Euclidean space $\R^d$
by the cube $[-1,1]^d$, with antipodal points on the boundary identified. At some places, which are indicated, we also use the identification $\T^d\cong[0,1]^d$.
The meaning of $\langle f,g \rangle_{\T^d}$ then depends on the utilized identification.

The Lebesgue spaces $L_p(D)$, $0<p\leq \infty$, are defined as usual. They contain 
all Lebesgue measurable functions $f:D\rightarrow \C$ satisfying
$$
     \|f\|_{L_p(D)} :=
     \Big(\,\int_{D} |f(x)|^p\,\dint x\Big)^{1/p} < \infty \,,
$$
with the usual modification in case $p=\infty$. 
The space $C(D)$ denotes the collection of all
continuous and bounded functions on $D$ equipped with the uniform norm. For 
$r\in\N\cup\{\infty\}$, the space $C^r(D)$ is comprised of the $r$-times differentiable functions on $D$, and
$C_{c}^{r}(D)$ denotes the subspace of compactly supported functions.

For $f\in L_1(\R^d)$ we use the following version of the Fourier transform,
\begin{align*}
\widehat{f}(x) := \int_{\R^d} f(\xi) \exp(-2\pi\ii x\cdot\xi) \,\dint\xi  \,.
\end{align*}
It is extended as usual to tempered distributions $f\in\mathcal{S}^\prime(\R^d)$, i.e.\ elements of the topological dual $\mathcal{S}^\prime(\R^d)$ of the Schwartz space $\mathcal{S}(\R^{d})$ of rapidly decreasing functions.  
Its inverse is denoted by $f^\vee$. 
For the dual pairing between tempered distributions $f$ and Schwartz functions $g$ we write $\langle f,g\rangle_{\mathcal{S}^\prime\times\mathcal{S}}$. 
More generally, we denote with $\langle f,g\rangle_{\mathscr{F}\times\tilde{\mathscr{F}}}$ the duality product on pairs $\mathscr{F}\times\tilde{\mathscr{F}}$ of conjugate spaces, see Definition~\ref{def:duality_prod}.


\section{The half-period cosine system}
\label{sec:hpc_system}

Let us, in this preliminary section, inspect more closely the half-period cosine system $\mathcal{C}_d=\{ c_{\bk}\}_{\bk\in\N^d_{0}}$
defined in~\eqref{def:hpc_system}. 
Hereby 
a periodization principle plays a key role that connects non-periodic functions on the $d$-dimensional hypercube $[0,1]^d$
with periodic functions on $[-1,1]^d$, however not in the standard way. It is realized by
the non-standard periodization operator $\mathcal{P}$ from~\eqref{def1:periodization_reflection} based on reflections rather than repetitions of the original function. More regularity is preserved this way, for instance the continuity of the input function. 

With the help of the auxiliary map 
\begin{align}\label{eqdef:aux_rho_1}
\rho_1:\: \R\to \R_{\ge0} \;,\quad x\mapsto 
 \min\big\{ x\bmod 2  , (-x)\bmod 2 \big\}
\end{align}
and its multivariate extension $\rho:=\rho_1 \otimes \ldots \otimes  \rho_1$,
this periodization operator $\mathcal{P}$ can be defined more conveniently than in~\eqref{def1:periodization_reflection} as follows.

\begin{definition}[Periodization]\label{def:PeriodizationOP}
	Let $f:D\to\C$ be a function with $[0,1]^d\subset D$. We denote by $\mathcal{P}$ the operator
	\begin{align*}
		\mathcal{P}:\; f\mapsto f\circ\rho|_{[-1,1]^d} \,.
	\end{align*}
\end{definition}

\noindent
Recall that we here identify $\T^d\cong [-1,1]^d$ and therefore parameterize the $1$-dimensional torus $\T$ by the extended unit interval $[-1,1]$, with endpoints identified.

The following restriction operator also plays an important role.

\begin{definition}[Restriction]
	\label{def:RestrictionOP}
        Let $f:D\to\C$ be a function with $[0,1]^d\subset D$.
	We denote by $\mathcal{R}$ the operator
	\begin{align*}
	\mathcal{R}:\; f\mapsto f|_{[0,1]^d} \,.
	\end{align*}
\end{definition}

\noindent
It is easy to check that 
\begin{align*}
\mathcal{P} \circ \mathcal{P} = \mathcal{P} \quad\text{and}\quad \mathcal{R} \circ \mathcal{R} = \mathcal{R} \quad\text{and}\quad \mathcal{R} \circ \mathcal{P} = \mathcal{R} \quad\text{and}\quad \mathcal{P} \circ \mathcal{R} = \mathcal{P} \,.
\end{align*}
These relations imply that $\mathcal{P}$ and $\mathcal{R}$ are pseudo-inverse to each other, fulfilling
$\mathcal{R} \circ \mathcal{P} \circ \mathcal{R} = \mathcal{R}$ and $\mathcal{P} \circ \mathcal{R} \circ  \mathcal{P} = \mathcal{P}$. Furthermore, for periodizations of functions $f:[0,1]^d\to\C$ via $\mathcal{P}$, i.e.\ applications of $\mathcal{P}$ to functions $f$ with domain $D=[0,1]^d$, the restriction operator $\mathcal{R}$ from  Definition~\ref{def:RestrictionOP} is a left-inverse.

\subsection{Periodization principle}

We are now ready to formulate the basic aspects of the periodization principle. 
One observation is that the periodizations $\mathcal{P}(f)$ lead to $d$-fold even functions on $\T^d$ with respect to the parametrization $\T^d\cong[-1,1]^d$.
Let us put this notion in a formal definition.
\begin{definition}\label{def:nfoldeven}
	A function $f:\T^d\cong [-1,1]^d \to\C$ is called \emph{$d$-fold even} if it satisfies
 \begin{align*}
f(x_1,\ldots,x_d) = f(\pm x_1,\ldots,\pm x_d)  \quad\text{for all}\quad (x_1,\ldots,x_d)\in[-1,1]^d \,.
\end{align*}
\end{definition}

\noindent
It is obvious that
\[
\mathcal{P}:\; L_1([0,1]^d) \to L_1(\T^d) 
\]
is an isometric operator with $\|\mathcal{P}(f)\|_{L_1(\T^d)} = 2^{d}\|f\|_{L_1([0,1]^d)}$.
Together with the observation that the image consists precisely of the $d$-fold even functions in $L_1([0,1]^d)$ we obtain the next lemma.

\begin{lemma}\label{lem:PeriodizationOP}
	The operator $\mathcal{P}$ from Definition~\ref{def:PeriodizationOP} is an isometric isomorphism
	\begin{align*}
		\mathcal{P}:\; L_1([0,1]^d) \cong L_1(\T^d)_{\rm even}  \,.
	\end{align*}
\end{lemma}
\noindent
Hereby $L_1(\T^d)_{\rm even}$ means the subspace of  $L_1(\T^d)$
comprised of the $d$-fold even functions.

Let us next have a look at the restriction operator $\mathcal{R}: f\mapsto f|_{[0,1]^d} $ from Definition~\ref{def:RestrictionOP} when applied to periodic functions on $\R^d$. Since $[0,1]^d\subset[-1,1]^d$, 
the interpretation $\T^{d} \cong [-1,1]^d$ indeed allows to apply this operator to functions $f\in L_1(\T^d)$.
One obtains $\mathcal{R}f\in L_1([0,1]^d)$ with $\|\mathcal{R}f\|_{L_1([0,1]^d)} \le \|f\|_{L_1(\T^d)}$ and hence 
the following lemma.

\begin{lemma}\label{lem:RestrictionOP}
The operator $\mathcal{R}$ from Definition~\ref{def:RestrictionOP} is bounded as operator
	\begin{align*}
		\mathcal{R}:\; L_1(\T^d) \to L_1([0,1]^d)  \,.
	\end{align*}
\end{lemma}

\noindent
Since $\mathcal{R}$ is a left-inverse to $\mathcal{P}$ on $L_1([0,1]^d)$, we see,
counting in Lemma~\ref{lem:PeriodizationOP}, that $\mathcal{P}$ and
$\mathcal{R}$ establish a $1$-$1$-correspondence between
$L_1([0,1]^d)$ and $L_1(\T^d)_{\rm even}$. Utilizing these operators, one can move back and forth between these spaces and it is this translation mechanism which is the foundation of the periodization principle. 

We next establish a relation of $\cC_d=\{c_{\bk}\}_{\bk\in\nd}$ to the standard Fourier basis in $L_2(\T^d)$ consisting of the exponential functions $\exp_{\bk}:\T^d \to \C$, $\bk\in\Z^d$. With respect
to the parametrization $\T^d \cong[-1,1]^d$ these
exponentials take the form
\begin{align*}
     \exp_{\bk}(x) := 2^{-d/2} \exp(\ii\pi \bk\cdot x) \;,\quad x=(x_1,\ldots,x_d)\in\T^d \,. 
\end{align*}
The normalization constant $2^{-d/2}$ ensures that the exponential system on $[-1,1]^d$,
\begin{align}\label{eqdef:exp_system}
\mathcal{E}([-1,1]^d) := \big\{ \exp_{\bk}\big\}_{\bk\in\Z^d}  \,,
\end{align}
is an orthonormal basis of $L_2(\T^d)$.

In addition, we define $\cos_{\bk}: \T^d \to \R$ 
for $\bk=(k_1,\ldots,k_d)\in\Z^d$ via 
\begin{align*}
\cos_{\bk}(x) &:= 2^{-d/2} \cos(\pi k_1 x_1)\cdots \cos(\pi k_d x_d)  \;,\quad x=(x_1,\ldots,x_d)\in\T^d \,. 
\end{align*}
As the next lemma states, the functions $c_{\bk}$ and $\cos_{\bk}$ are related by the periodization operator $\mathcal{P}$.

\begin{lemma}\label{lem:rel_ck_cosk}
For all $\bk\in\Z^d$ it holds
\[
\mathcal{P}(c_{|\bk|}) = 2^{(|\bk|_0+d)/2} \cos_{\bk} \,,
\]
where $|\bk|_0$ counts the number of non-zero entries of $\bk$
and $|\bk|=(|k_1|,\ldots,|k_d|)\in\nd$.

\end{lemma}
\begin{proof}
By definition, $c_{\bk}(x)=2^{|\bk|_0/2} \cos(\pi k_1x_1)\cdots \cos(\pi k_dx_d)$ for $\bk\in\nd$ and $x\in[0,1]^d$. Inserting the definition of $\mathcal{P}$, we obtain for $x\in\T^d$
\begin{align*}
    \mathcal{P}(c_{\bk})(x) = 2^{|\bk|_0/2} \cos(\pi k_1 |x_1|)\cdots \cos(\pi k_d|x_d|)  =
     2^{|\bk|_0/2} \cos(\pi k_1 x_1)\cdots \cos(\pi k_dx_d) \,.
\end{align*}
Recalling the definition of $\cos_{\bk}$ and the symmetry $\cos_{\bk}=\cos_{|\bk|}$, we are finished.
\end{proof}
\noindent
Let us remark that the formula in Lemma~\ref{lem:rel_ck_cosk} is analogous to the connection, described in~\eqref{rel:tent_hpc}, of $c_{\bk}$ to the tensorized cosine functions $\cos(2\pi k_1 \cdot)\cdots \cos(2\pi k_d \cdot)$ on $[0,1]^d$ via the tent transform, namely 
\[
c_{|\bk|} \circ \phi = 2^{|\bk|_0/2} \cos(2\pi k_1 \cdot)\cdots \cos(2\pi k_d \cdot) \,.
\]

\noindent
Finally, drawing a connection between $\cos_{\bk}$ and $\exp_{\bk}$, note that by Euler's formula we have the expansion
\begin{align}\label{eq:expansion_expk}
  \exp_{\bk}(x) 
  = 2^{-d/2} \prod_{j=1}^d \exp(\ii \pi k_j x_j) &= \sum_{\beps\in\{0,1\}^d} \ii^{|\beps|_1} g_{\bk}^{(\beps)}(x)
  \nonumber \\
  &= \cos_{\bk}(x) + \sum_{\beps\in\{0,1\}^d \,, |\beps|_1>0} \ii^{|\beps|_1} g_{\bk}^{(\beps)}(x) \,, 
\end{align}
where for each $\beps=(\varepsilon_1,\ldots,\varepsilon_d)\in\{0,1\}^d$
\begin{align*}
   g_{\bk}^{(\beps)}(x) := g_{k_1}^{(\varepsilon_1)}(x_1)\cdots g_{k_d}^{(\varepsilon_d)}(x_d) \quad\text{with}\quad g_{k_j}^{(\varepsilon_j)}(x_j) := \begin{cases} 2^{-1/2}\cos(\pi k_j x_j) &,\,\varepsilon_j=0 \,, \\
    2^{-1/2}\sin(\pi k_j x_j)  &,\,\varepsilon_j=1 \,.
    \end{cases}
\end{align*}
Hence, the function $\cos_{\bk}$ can be considered as the $d$-fold even part of $\exp_{\bk}$.

\subsection{Half-period cosine coefficients}

The half-period cosine coefficients $\hat{f}_{\rm hpc}(\bk)$ of a function $f\in L_1([0,1]^d)$ have been defined in~\eqref{eqdef:hpc_coeff}.
The computation of the Fourier coefficients of an integrable $d$-variate periodic function on $\T^d\cong[-1,1]^d$ is performed by the formula 
$$
   \hat{f}(\bk) = 2^{-d/2} \int_{\T^d}
  f(x)\ee^{-\ii\pi \bk \cdot x}\,\dint x \,,\quad \bk\in \Z^d\,.
$$
Using the bracket notation introduced in~\eqref{eqdef:integration_prod}, we can write both coefficients short-hand as
\begin{align*}
\hat{f}_{\rm hpc}(\bk) =  \langle f, c_{\bk}\rangle \quad,\quad \hat{f}(\bk)= \langle f, \overline{\exp}_{\bk}  \rangle_{\T^d} \,.
\end{align*}

\noindent
We know from Lemma~\ref{lem:PeriodizationOP} that $\mathcal{P}(f)\in L_1(\T^d)_{\rm even}$ and $\mathcal{P}(g)\in L_1(\T^d)_{\rm even}$ for any $f,g\in L_1([0,1]^d)$. Another
basic but important observation is that for any $f,g\in L_1([0,1]^d)$ we have  
\begin{align}\label{eq:scalarprod_rel}
  \langle f, g \rangle= 2^{-d}  \langle \mathcal{P}(f), \mathcal{P}(g) \rangle_{\T^d} \,.
\end{align}
From this fact and Lemma~\ref{lem:rel_ck_cosk} we get the following result that relates
the Fourier coefficients of $\mathcal{P}(f)$ to the half-period cosine coefficients of $f$.

\begin{lemma}\label{lem:aux_Fourier}
Recall that $|\bk|_0$ counts the non-zero entries of $\bk\in\Z^d$ and that $|\bk|=(|k_1|,\ldots,|k_d|)\in\nd$. It holds 
\begin{align*}  
 \langle f, c_{|\bk|}\rangle =  2^{(|\bk|_0-d)/2}  \langle \mathcal{P}(f), \overline{\exp}_{\bk}\rangle_{\T^d}\,. 
\end{align*}
\end{lemma}
\begin{proof}
Put $\tilde{f}:=\mathcal{P}(f)$. From Lemma~\ref{lem:rel_ck_cosk} and \eqref{eq:scalarprod_rel} we get
\begin{align*}
    \langle f,c_{|\bk|} \rangle = 2^{(|\bk|_0-d)/2} \langle \tilde{f},\cos_{\bk} \rangle_{\T^d}  \,.
\end{align*}

\noindent
Further, we have according to the expansion~\eqref{eq:expansion_expk}
\begin{align*}
     \langle \tilde{f},\exp_{\bk} \rangle_{\T^d}
    = \langle \tilde{f},\cos_{\bk} \rangle_{\T^d}  + \sum_{\beps\in\{0,1\}^d,\, |\beps|_1>0 } \langle \tilde{f}, \ii^{|\beps|_1}
    g_{\bk}^{(\beps)} \rangle_{\T^d} 
    = \langle \tilde{f},\cos_{\bk} \rangle_{\T^d} \,.
\end{align*}

Indeed, let $\beps\in\{0,1\}^d$ with $|\beps|_1>0$, where without loss of generality we assume $\varepsilon_1=1$. Then
\begin{align*}
\langle  \tilde{f}, g_{\bk}^{(\beps)} \rangle_{\T^d} &=
  \int\limits_{x_d=-1}^1 \cdots  \int\limits_{x_2=-1}^1 \bigg( \int\limits_{x_1=-1}^1 \tilde{f}(x_1,\ldots,x_d)  g_{k_1}^{(\varepsilon_1)}(x_1) \,\dint x_1 \bigg) 
 g_{k_2}^{(\varepsilon_2)}(x_2)\cdots g_{k_d}^{(\varepsilon_d)}(x_d) \,\dint x_2\cdots\,\dint x_d 
\end{align*}
and, since $\varepsilon_1=1$ and $\tilde{f}$ is even in each component,
\[
\int\limits_{x_1=-1}^1 \tilde{f}(x_1,\ldots,x_d) g_{k_1}^{(\varepsilon_1)}(x_1) \,\dint x_1 = \int\limits_{x_1=-1}^1  \tilde{f}(x_1,\ldots,x_d) 2^{-1/2} \sin(\pi k_1 x_1) \,\dint x_1  = 0 \,. 
\]
The final observation, which finishes the proof, is $\langle \tilde{f}, \exp_{\bk}\rangle_{\T^d}=\langle \tilde{f}, \overline{\exp}_{\bk}\rangle_{\T^d}$ due to $\overline{\exp}_{\bk}=\exp_{-\bk}$.
\end{proof}

\subsection{Half-period cosine expansions}

Since $\mathcal{C}_{d}=\{ c_{\bk}\}_{\bk\in\N^d_{0}}$ is an orthonormal basis of $L^{2}([0,1]^d)$ each $f\in L_2([0,1]^d)$
has an orthonormal expansion
\begin{align}\label{eq:HPC-Expansion}
f = \sum_{\bk\in\nd}  \langle f, c_{\bk} \rangle c_{\bk} 
\end{align}
with $L_2$ convergence and the coefficients satisfy $\{\langle f, c_{\bk} \rangle\}_{\bk}\in\ell_{2}(\N^d_{0})$.
We are interested in an analogous result for $f\in L_1([0,1]^d)$.
The first observation is that, since $c_{\bk}\in L_{\infty}([0,1]^d)$, the half-period cosine coefficients $\hat{f}_{\rm hpc}(\bk)=\langle f, c_{\bk} \rangle$ still make sense for every $f\in L_1([0,1]^d)$. In this more general case, however, one merely has $\{\hat{f}_{\rm hpc}(\bk)\}_{\bk}\in\ell_{\infty}(\N^d_{0})$ and the question of convergence in~\eqref{eq:HPC-Expansion} arises.

We will show in Lemma~\ref{Lemma HPC Reihe konvergiert im Distributionensinn} below that, when the functions $f$ and $c_{\bk}$ are considered as functions on $\R^d$ trivially extended by $0$'s outside of $[0,1]^d$, one has weak* convergence in $\mathcal{S}^{\prime}(\R^{d})$, where $\mathcal{S}^{\prime}(\R^{d})$ is the space of tempered distributions on $\R^{d}$.
To prove Lemma~\ref{Lemma HPC Reihe konvergiert im Distributionensinn}, we first analyze the decay of the
half-period cosine coefficients $\hat{f}_{\rm hpc}(\bk)$ of functions $f=\eta\cdot\chi_{[0,1]^d}$, where $\eta\in\mathcal{S}(\R^{d})$ is a Schwartz function.

\begin{lemma}\label{Lemma Abfallrate HPC Koeffizienten Schwartz Funktion}
Let $\eta \in \mathcal{S}(\R^{d})$ and $\br{x} := (1 + |x|^{2} )^{\frac{1}{2}}$ be the `analyst's bracket'. For every $\bk \in \N_{0}^{d}$ we have the estimate 
\begin{equation}\label{eq:decayHPCcoeffSchwartz}
    \abs{\br{\eta, c_{\bk}}}
    \lesssim \frac{1}{\prod_{i=1}^{d} \br{k_{i}}^{2}} \sum_{ \substack{\balpha \in \N_{0}^{d}\\\abs{\balpha}_{\infty} \leq 2} } \norm{\partial^{\balpha} \eta}_{L_{\infty}(\R^{d})} \,.
\end{equation}
\end{lemma}
\begin{proof}
The proof is by induction on the dimension $d$. Our starting point is the case $d = 1$. Take $\eta \in \mathcal{S}(\R)$. Then 
\begin{equation*}
    \abs{\br{\eta,c_{0}}}
    \leq \int\limits_{0}^{1} \abs{\eta(x)} \dd x
    \leq \norm{\eta}_{L_\infty(\R)} \,.
\end{equation*}
For $k \in \N$ we get, integrating by parts twice, 
\begin{equation*}
    \br{\eta,c_{k}}
    = \int\limits_{0}^{1} \eta(x)c_{k}(x) \dd x
    = \frac{\sqrt{2}}{\pi^{2} k^{2}} \left( (-1)^{k} \partial \eta(1) - \partial \eta(0) \right) - \frac{1}{\pi^{2} k^{2}} \int\limits_{0}^{1} c_{k}(x) \partial^{2}\eta(x) \dd x \,.
\end{equation*}
Consequently, we have~\eqref{eq:decayHPCcoeffSchwartz} in the univariate case.
Let us next assume that~\eqref{eq:decayHPCcoeffSchwartz} is true up to dimension $d \in \N$. If $x \in \R^{d+1}$ we 
subsequently write $x = (\widehat{x}, t)$ with $\widehat{x} \in \R^{d}$ and $t \in \R$. Similarly, we write $\bk = (\widehat{k}, l) \in \N_{0}^{d} \times \N_{0}$ for $\bk\in\N_{0}^{d+1}$. 
Considering $\eta \in \mathcal{S}(\R^{d+1})$, we have
\begin{equation*}
    \br{\eta,c_{\bk}}
    = \br{ c_{\widehat{k}}(\widehat{x})c_{l}(t), \eta(\widehat{x}, t)}
    = \br{ c_{l}(t), \br{c_{\widehat{k}}(\widehat{x}), \eta(\widehat{x}, t)}_{\widehat{x}} }_{t} \,,
\end{equation*}
where $c_{\widehat{k}} = \bigotimes_{i = 1}^{d} c_{k_{i}}$. The map $t \mapsto \br{c_{\widehat{k}}(\widehat{x}), \eta(\widehat{x}, t)}_{\widehat{x}}$ belongs to $\mathcal{S}(\R)$ and so our previous calculations yield 
\begin{equation*}
    \begin{aligned}
    \abs{\br{c_{\bk}, \eta}}
    \lesssim 
    \frac{1}{\br{l}^{2}} \sum_{\beta = 0}^{2} \, \sup_{t \in \R} \abs{ \br{c_{\widehat{k}}(\widehat{x}), \partial_{t}^{\beta} \eta(\widehat{x}, t)}_{\widehat{x}}} \,,
    \end{aligned}
\end{equation*}
where, by induction hypothesis,
\begin{equation*}
    \abs{ \br{c_{\widehat{k}}(\widehat{x}), \partial_{t}^{\beta} \eta(\widehat{x}, t)}_{\widehat{x}}}
    \lesssim \frac{1}{\prod_{i = 1}^{d} \br{k_{i}}^{2}} \sum_{\substack{\widehat{\alpha} \in \N_{0}^{d}\\\abs{\widehat{\alpha}}_{\infty} \leq 2}} \sup_{\widehat{x} \in \R^{d}} \abs{ \partial_{t}^{\beta} \partial_{\widehat{x}}^{\widehat{\alpha}} \eta(\widehat{x}, t) } \,.
\end{equation*}
Hence, we obtain altogether
\begin{equation*}
    \begin{aligned}
    \abs{\br{c_{\bk}, \eta}}
    \lesssim \frac{1}{\br{l}^{2}} \sum_{\beta = 0}^{2} \, \sup_{t \in \R} \frac{1}{\prod_{i = 1}^{d} \br{k_{i}}^{2}} \sum_{\substack{\widehat{\alpha} \in \N_{0}^{d}\\\abs{\widehat{\alpha}}_{\infty} \leq 2}} \sup_{\widehat{x} \in \R^{d}} \abs{ \partial_{t}^{\beta} \partial_{\widehat{x}}^{\widehat{\alpha}} \eta(\widehat{x}, t) } 
    \leq \frac{1}{\br{l}^{2} \prod_{i=1}^{d} \br{k_{i}}^{2}} \sum_{ \substack{\balpha \in \N_{0}^{d+1}\\\abs{\balpha}_{\infty} \leq 2} } \norm{\partial^{\balpha} \eta}_{L_{\infty}(\R^{d+1})} \,.
    \end{aligned}
\end{equation*}
Our proof is complete.
\end{proof}

\noindent
Now we are ready to prove Lemma~\ref{Lemma HPC Reihe konvergiert im Distributionensinn}.

\begin{lemma}\label{Lemma HPC Reihe konvergiert im Distributionensinn}
    Let $f \in L_{1}([0,1]^{d})$ and identify $f$ with its trivial extension to $\R^{d}$. Then we have
    \begin{equation*}
        f
        = \sum_{\bk \in \N_{0}^{d}} \br{f, c_{\bk}} c_{\bk}
        \quad \text{weak*ly in } \mathcal{S}^{\prime}(\R^{d}) \,.
    \end{equation*}   
\end{lemma}

\begin{proof}
    Take $\eta \in \mathcal{S}(\R^{d})$. Then $\mathcal{P}(\eta) \in L_{2}(\T^{d})$ for the periodization operator $\mathcal{P}$ from Definition~\ref{def:PeriodizationOP} and, by the multivariate version of the celebrated Carleson-Hunt theorem, see e.g.~\cite[Thm~4.3.16]{Grafakos2014}, the function $\mathcal{P}(\eta)$ coincides with its Fourier series pointwise almost everywhere on $\T^{d}$. Due to symmetry, this Fourier series is equal to the half-period cosine series of $\eta$ on $[0,1]^{d}$. In other words, on the unit cube the latter converges to $\eta$ pointwise almost everywhere. We thus have
    \begin{equation}\label{eq:pointwise_conv_HPC}
        \eta(x)
        = \sum_{\bk \in \N_{0}^{d}} \br{\eta, c_{\bk}} c_{\bk}(x)
        = \lim_{N \rightarrow \infty} \sum_{|\bk|_{\infty} \leq N} \br{\eta, c_{\bk}} c_{\bk}(x)  \quad \text{for a.e. } x \in [0,1]^{d} \,.
    \end{equation}

    \noindent
    Consider now an arbitrary $f\in L_{1}([0,1]^{d})$ and identify it with its trivial extension to $\R^{d}$. Then we have
    \begin{equation*}
        \begin{aligned}
            \bigg| \Big\langle f - \sum_{|\bk|_{\infty} \leq N} \br{f, c_{\bk}} c_{\bk}, \eta \Big\rangle \bigg|
            &= \bigg| \int_{[0,1]^{d}} \Big( f(x) - \sum_{|\bk|_{\infty} \leq N} \br{f, c_{\bk}} c_{\bk}(x) \Big) \, \eta(x) \dd x \bigg| \\
            &\leq \int_{[0,1]^{d}} \abs{f(x)} \, \bigg| \eta(x) - \sum_{|\bk|_{\infty} \leq N} \br{\eta, c_{\bk}} c_{\bk}(x) \bigg| \dd x \:\longrightarrow 0 \quad (N \longrightarrow \infty) \,.
        \end{aligned}
    \end{equation*}

    \noindent
    The convergence in the last step follows from Lebesgue's dominated convergence theorem, since~\eqref{eq:pointwise_conv_HPC} holds true and, due to Lemma~\ref{Lemma Abfallrate HPC Koeffizienten Schwartz Funktion}, the
    function
    \[
    f(\cdot)\Big( \|\eta\|_{L_\infty(\R^d)} + \sum\limits_{\bk\in \N_{0}^{d}} |\langle\eta, c_{\bk}\rangle| \|c_{\bk}\|_{L_\infty(\R^d)} \Big)
    \]
    is an integrable majorant.
\end{proof}


\subsection{Further aspects of the periodization principle}
\label{ssec:further_aspects}

In this final part of Section~\ref{sec:hpc_system} we want to present some additional aspects of the periodization principle from an applications' point of view.
In contrast to the rest of the paper, the identification $\T^d\cong[0,1]^d$ will be used here and the tent transform $\phi$ instead of $\mathcal{P}$ for the periodization. This 
has the advantage that then both $\T^d$ and $[0,1]^d$ have unit measure. It is further more common in the half-period cosine literature (see e.g.~\cite{Dick2013,CKNS2016,goda2019lattice}). 
The tent transform $\phi$ was explained in Subsection~\ref{ssec:per_and_tent}.
We set $f^\phi:=f\circ\phi$ for functions $f:[0,1]^d\to\C$.

As our first application, let us consider the problem of $L_p$ approximation of functions $f\in L_1([0,1]^d)$. 
Hereby we assume that we have a selection rule $\sigma:\N \to \mathscr{P}(\N_0^d)$ for the half-period cosine coefficients $\hat{f}_{\rm hpc}(\bk)$ of $f$, where $\mathscr{P}(\N_0^d)$ denotes the
power set of $\N_0^d$. The $N$th half-period cosine approximant of $f$ with respect to $\sigma$ is given by 
\begin{align*}
    A^{(\sigma)}_{{\rm hpc},N} f(x) := \sum_{\bk\in\sigma(N)} \hat{f}_{\rm hpc}(\bk)  c_{\bk}(x)  \,.
\end{align*}
For periodic functions $g\in L_1(\T^d)$ we further define the $N$th Fourier approximant
\begin{align*}
    A^{(\sigma)}_N g(x) :=  \sum_{|\bk|\in\sigma(N)} \langle g, \exp(-2\pi\ii\bk\cdot\,) \rangle_{\T^d}  \exp(2\pi\ii\bk\cdot x) 
\end{align*}
subject to the selection rule $\sigma$.
The following result then holds true.

\begin{lemma}
\label{app:approx}
For $0<p\le\infty$
\begin{align*}
\| f - A^{(\sigma)}_{{\rm hpc},N} f \|_{L_p([0,1]^d)} = \| f^\phi - A^{(\sigma)}_N f^\phi \|_{L_p(\T^d)} \,.
\end{align*}
\end{lemma}
\begin{proof}
Due to relation~\eqref{rel_phi_P} between $\phi$ and $\mathcal{P}$, it suffices to show
\begin{align}\label{eq:app:approx_2}
\| f - A^{(\sigma)}_{{\rm hpc},N} f \|_{L_p([0,1]^d)} = 2^{-d/p} \| \mathcal{P}(f) - \widetilde{A}^{(\sigma)}_N[\mathcal{P}(f)] \|_{L_p(\T^d)} \,,
\end{align}
where here in the proof we let $\T^d\cong[-1,1]^d$ as usual and set
\begin{align*}
\widetilde{A}^{(\sigma)}_N[\mathcal{P}(f)] :=  \sum_{|\bk|\in\sigma(N)}  \langle \mathcal{P}(f), \overline{\exp}_{\bk}\rangle_{\T^d}  \exp_{\bk} \,.
\end{align*}
From Lemma~\ref{lem:rel_ck_cosk} and Lemma~\ref{lem:aux_Fourier} we obtain 
\begin{align*}
\mathcal{P}(A^{(\sigma)}_{{\rm hpc},N} f) = \sum_{\bk\in\sigma(N)} \langle f, c_{\bk} \rangle  \mathcal{P}(c_{\bk})
=  \sum_{\bk\in\sigma(N)}  2^{|\bk|_0}  \langle \mathcal{P}(f), \overline{\exp}_{\bk}\rangle_{\T^d}  \cos_{\bk}
=\widetilde{A}^{(\sigma)}_N[\mathcal{P}(f)] \,.
\end{align*}
Further
\begin{align*}
\| f - A^{(\sigma)}_{{\rm hpc},N} f \|^p_{L_p([0,1]^d)}
=  2^{-d} \| \mathcal{P}(f) - \mathcal{P}(A^{(\sigma)}_{{\rm hpc},N} f) \|^p_{L_p(\T^d)} = 2^{-d} \| \mathcal{P}(f) - \widetilde{A}^{(\sigma)}_N[\mathcal{P}(f)] \|^p_{L_p(\T^d)}  \,,
\end{align*}
which verifies~\eqref{eq:app:approx_2}.
\end{proof}

\noindent
Lemma~\ref{app:approx} provides the following insight:
For $f\in L_1([0,1]^d)$ an approximation scheme $\{A^{(\sigma)}_{{\rm hpc},N}\}_{\raisebox{-0.6ex}{$\scriptstyle{N\in\N}$}}$ based on half-period cosine functions works precisely as well as the corresponding Fourier approximation scheme $\{A^{(\sigma)}_{N}\}_{\raisebox{-0.6ex}{$\scriptstyle{N\in\N}$}}$ for $f^\phi$ on $\T^d$.

The cubature of functions $f\in C([0,1]^d)$ is the second application we are interested in. Here one tries to approximate the integral
\begin{align*}
I(f):= \int_{[0,1]^d} f(x) \,\dint x     
\end{align*}
from a finite number of function samples.
For $n\in\N$, let us fix a set of sampling points 
$X_n=\{x_j\}_{j=1}^{n}\subset[0,1]^d$ and an associated weight set 
$\Lambda_n=\{\lambda_j\}_{j=1}^{n}\subset \R$.
The pair $(X_n,\Lambda_n)$ determines a 
cubature rule $\Lambda_n(X_n,f)$ for functions $f\in C([0,1]^d)$
by
\begin{align*}
    \Lambda_n(X_n,f) := \sum_{x_{j}\in X_n} \lambda_{j} f(x_{j})  \,.
\end{align*}

\noindent
The tent-transformed node set $X^\phi_n=\{x^\phi_j\}_{j=1}^{n}$ 
is given by $x^\phi_j := \phi(x_j)$ and
we call $\Lambda_n(X^\phi_n,f)$ 
the corresponding tent-transformed cubature rule for $f$.
We have the following result.

\begin{lemma}
\label{app:cub}
It holds $|\Lambda_n(X_n^\phi,f) - I(f)| = |\Lambda_n(X_n,f^\phi) - I(f^\phi)|$.
\end{lemma}
\begin{proof}
The statement follows from $I(f)=I(f^\phi)$ and
\begin{gather*}
    \Lambda_n(X_n^\phi,f) =  \sum\limits_{j=1}^{n} \lambda_{j} f(x^\phi_{j})  = \sum\limits_{j=1}^{n} \lambda_{j} f^\phi(x_{j}) = \Lambda_n(X_n,f^\phi) \,. \qedhere
\end{gather*}
\end{proof}

\noindent
From Lemma~\ref{app:cub} we learn that a tent-transformed cubature rule performs as well for a non-periodic function $f\in C([0,1]^d)$ 
as the original rule for the periodized function $f^\phi\in C(\T^d)$. 

Concrete applications of Lemmas~\ref{app:approx} and~\ref{app:cub} are provided in Subsection~\ref{ssec:app_theo_id}.



\section{Besov-type function spaces} 
\label{sec:func_spaces}

The Fourier analytic definition of the classical univariate Besov spaces $B^{r}_{p,q}(\R)$ is based on dyadic partitions of the Fourier domain. A central notion in this respect are
dyadic decompositions of unity $\{\varphi_{j}\}_{j \in \N_{0}}$, i.e.\ families of non-negative functions $\varphi_{j}\in \mathcal{S}(\R)$ subject to the following three conditions:
\begin{itemize}
	\item[(i)] There is $c > 0$ such that $\supp(\varphi_{0}) \subset \{x \in \R : \abs{cx} \leq 2 \}$ and $\supp (\varphi_{j}) \subset \{ x \in \R : 2^{j-1} \leq \abs{cx} \leq 2^{j+1} \}$ for $l \in \N$.
	\item[(ii)] For every $m \in \N_{0}$ there exists a constant $c_{m} > 0$ such that $\sup_{x\in\R,j\in\N_0} 2^{mj} \abs{ \partial^{m} \varphi_{j}(x) } \leq c_{m}$.
	\item[(iii)] $\displaystyle \sum_{j \in \N_{0}} \varphi_{j}(x) = 1$ at every point $x \in \R$.
\end{itemize}
Dyadic decompositions like this give rise to corresponding decompositions of functions $f:\R\to\C$, which are at the center of Littlewood-Paley theory and form the Fourier analytic basis for the  construction of a wide range of function spaces, prominently among them the Besov and the Triebel-Lizorkin scale (see e.g.~\cite{Tr10,Sawano18}).

\begin{remark}
The standard example of a dyadic decomposition of unity is as follows. One lets $\varphi_{0} \in C_{c}^{\infty}(\R)$ be non-negative with $\varphi_{0}(x) = 1$ on $[-1,1]$ and $\varphi_{0}(x) = 0$ if $\abs{x} > 2$. For $j \in \N$ one further sets
\begin{equation}
	\varphi_{j}(x)
	:= \varphi_{0}(2^{-j}x) - \varphi_{0}(2^{-j+1}x) \,. \notag
\end{equation}
\end{remark}

\noindent
There are two main directions for the generalization of the univariate spaces $B^{r}_{p,q}(\R)$ to the multivariate domain: the isotropic Besov spaces $B^{r}_{p,q}(\R^{d})$ and the Besov spaces of dominating mixed smoothness $S^{r}_{p,q}B(\R^{d})$. Both directions differ in the utilized concept of multivariate dyadic decompositions of unity on $\R^d$.
For the isotropic spaces $B^{r}_{p,q}(\R^{d})$ the Fourier domain is partitioned into rotation-symmetric dyadic annuli (see e.g.~\cite{Tr10,Sawano18}). 
The spaces $S^{r}_{p,q}B(\R^{d})$, on the other side, 
follow the idea of tensorization (see e.g.~\cite[Ch.~2]{Schmeisser1987} or~\cite[Ch.~1]{Vy06}). 

\subsection{The classical Besov spaces of dominating mixed smoothness} 

In this paper we will
deal with the spaces $S^{r}_{p,q}B(\R^{d})$ whose definition rests on multivariate partitions of unity obtained by tensorization.
Those are often referred to as hyperbolic in the literature. One starts with $d\in\N$, possibly different, dyadic partitions $\{\varphi^{(i)}_{j}\}_{j\in\N_{0}}$ on $\R$, indexed by $i\in[d]$, and then defines for $\bj=(j_1,\ldots,j_d)\in \N_{0}^{d}$
\begin{equation*}
	\varphi_{\bj}(x)
	:= \varphi^{(1)}_{j_{1}}(x_{1}) \cdot \ldots \cdot \varphi^{(d)}_{j_{d}}(x_{d}) \;,\quad x=(x_1,\ldots,x_d)\in R^{d} \, .
\end{equation*}
The class of families $\{\varphi_{\bj}\}_{\bj \in \N_{0}^{d}}$ of such construction will subsequently be denoted
by $\Phi_{\rm hyp}(\R^d)$.

Associated to each family $\{\varphi_{\bj}\}_{\bj \in \N_{0}^{d}}\in\Phi_{\rm hyp}(\R^d)$ is a family $\{\Phi_{\bj}\}_{\bj \in \N_{0}^{d}}$ of kernels $\Phi_{\bj}$ connected to $\varphi_{\bj}$ by the relation
\begin{align}\label{eqdef:ass_kernels}
\varphi_{\bj}(x)=\widehat{\Phi}_{\bj}(x) = \int_{\R^d} \Phi_{\bj}(\xi) \exp(-2\pi\ii x\cdot\xi) \,\dint\xi   \,.
\end{align} 
Since $\{\varphi_{\bj}\}_{\bj \in \N_{0}^{d}}$ is contained in $\mathcal{S}(\R^d)$
the kernels $\{\Phi_{\bj}\}_{\bj \in \N_{0}^{d}}$ are also elements of the Schwartz space $\mathcal{S}(\R^d)$.

\subsubsection{The spaces on $\R^d$}  

As mentioned in the beginning of this section,
dyadic decompositions of unity give rise to corresponding decompositions of functions.
More generally, 
let us now consider tempered distributions $f\in \mathcal{S}^{\prime}(\R^{d})$, i.e.\ elements of the dual space of $\mathcal{S}(\R^{d})$, and decompositions $\{\varphi_{\bj}\}_{\bj \in \N^{d}_{0}}\in\Phi_{\rm hyp}(\R^d)$.   
The components of the decomposition of $f$ are defined as $f_{\bj}:=(\varphi_{\bj}\widehat{f}\,)^\vee$ and are called Littlewood-Paley blocks.
They can be written as convolutions 
\begin{align}\label{eqdef:LPops}
f_{\bj}(x)  = \Phi_{\bj} \ast f(x) :=  \br{f, \Phi_{\bj}(x - \cdot)}_{\mathcal{S}^\prime\times \mathcal{S}}  
\end{align}
with the associated kernels $\Phi_{\bj}$, where $\langle\cdot,\cdot\rangle_{\mathcal{S}^\prime\times \mathcal{S}}$ denotes the duality product on $\mathcal{S}^\prime(\R^d)\times\mathcal{S}(\R^d)$.

The corresponding decomposition of $f$, 
\begin{align}\label{LPdecomp}
	f = \sum_{\bj\in\nd} f_{\bj} = \sum_{\bj\in\nd} \Phi_{\bj} \ast f \,,
\end{align}
which holds true with at least weak* convergence in $\mathcal{S}^{\prime}(\R^{d})$, 
is called Littlewood-Paley decomposition of $f$. 
With this notion we are ready to give the definition of 
$S^{r}_{p,q}B(\R^{d})$.

\begin{definition}[{cf.~\cite[Def.~2]{Schmeisser1987},~\cite[Def.~3.2.2]{TDiff06}}]\label{def:domix_Besov_Rd}
	Let $0 < p,q \leq \infty$, $r\in\R$, and $\{\Phi_{\bj}\}_{\bj \in \N_{0}^{d}}$ a family of kernels associated to $\{\varphi_{\bj}\}_{\bj \in \N_{0}^{d}}\in\Phi_{\rm hyp}(\R^d)$. The space $S^{r}_{p,q}B(\R^{d})$ is defined as the collection of all $f \in \mathcal{S}^{\prime}(\R^{d})$ for which
	\begin{equation*}
		\norm{f}_{S^{r}_{p,q}B(\R^{d})}
		:= \begin{cases}
			 \Big( \sum\limits_{\bj \in \N_{0}^{d}} 2^{r q |\bj|_{1}} \norm{\Phi_{\bj} \ast f}_{L_{p}(\R^{d})}^{q} \Big)^{1/q} \; , \quad &q < \infty \; , \\[4ex]
		\displaystyle \sup_{\bj \in \N_{0}^{d}} 2^{r |\bj|_{1}} \norm{\Phi_{\bj} \ast f}_{L_{p}(\R^{d})}  \; , \quad &q = \infty \; ,
		\end{cases}
	\end{equation*}
	is finite. The expression $\norm{f}_{S^{r}_{p,q}B(\R^{d})}$ constitutes a (quasi-)norm in $S^{r}_{p,q}B(\R^{d})$.
\end{definition}

\noindent
The spaces $S^{r}_{p,q}B(\R^{d})$ do not depend on the choice of family $\{\varphi_{\bj}\}_{\bj}\in\Phi_{\rm hyp}(\R^d)$ in the definition. Different families 
lead to the same class of functions and equivalent (quasi-)norms. In case $\min\{p,q\}\geq 1$ the
defined spaces are Banach spaces,
whereas they are quasi-Banach in case $u:=\min\{p,q\} < 1$. More precisely, they are then
$u$-Banach spaces since the quasi-norm satisfies the $u$-triangle inequality 
\[
\norm{f+g}^u_{S^{r}_{p,q}B(\R^{d})} \le \norm{f}^u_{S^{r}_{p,q}B(\R^{d})} + \norm{g}^u_{S^{r}_{p,q}B(\R^{d})} \,.
\]
For these facts and many other properties of these spaces we refer to the literature (e.g.~\cite{Schmeisser1987,Vy06}).

Let us also recall the concept of Besov-type sequence spaces. Spaces of this kind, $s^r_{p,q}b(\Gamma^d)$ with index set $\Gamma^d:= \N_{-1}^d\times \Z^d$, naturally associated to $S^{r}_{p,q}B(\R^{d})$ have been considered for $d=1$ in~\cite[Def.~A.4]{DerUll19}.

\begin{definition}[{cf.~\cite[Def.~A.4]{DerUll19} and Remark~\ref{rem:Besov_seq_spaces}}]
Let $0<p,q\leq \infty$, $r\in \R$. 
The space
$s^r_{p,q}b(\Gamma^d)$ is
the collection of all sequences $\{\lambda_{\bj,\bk}\}_{(\bj,\bk)\in \Gamma^d}$ in $\C$ with index set  $\Gamma^d:= \N_{-1}^d\times \Z^d$
such that
\begin{equation*}
  \|\{\lambda_{\bj,\bk}\}_{\bj,\bk}\|_{s^r_{p,q}b(\Gamma^d)}
  := \begin{cases}
			\Big(\sum\limits_{\bj\in
\N_{-1}^d}2^{|\bj|_1(r-\frac{1}{p})q}\Big(\sum\limits_{\bk\in
\Z^d}|\lambda_{\bj,\bk}|^p\Big)^{\frac{q}{p}} \Big)^{\frac{1}{q}} \; , \quad &q < \infty \; , \\[4ex]
		\displaystyle \sup_{\bj \in \N_{-1}^{d}} 2^{\abs{\bj}_{1}(r-\frac{1}{p})} \Big(\sum\limits_{\bk\in\Z^d}|\lambda_{\bj,\bk}|^p\Big)^{\frac{1}{p}}  \; , \quad &q = \infty \; ,
		\end{cases}
\end{equation*}
is finite. The expression $\|\{\lambda_{\bj,\bk}\}_{\bj,\bk}\|_{s^r_{p,q}b(\Gamma^d)}$ constitutes a (quasi-)norm.
\end{definition}

\begin{remark}\label{rem:Besov_seq_spaces}
If one defines $\chi_{\bj,\bk}:=\chi_{\raisebox{-0.5ex}{$\scriptstyle{Q_{\bj,\bk}}$}}$ as the characteristic function of the rectangle 
\begin{align}\label{def:chi_cubes}
 Q_{\bj,\bk} := Q_{j_1,k_1} \times\cdots\times  Q_{j_d,k_d} \quad\text{with}\quad Q_{j_i,k_i} := [2^{j_i^+}k_i,2^{j_i^+}(k_i+1)] \quad\text{for } i\in[d] \,,
\end{align}
it holds, with the usual modification if $q=\infty$, 
\begin{align*}
 \|\{\lambda_{\bj,\bk}\}_{\bj,\bk}\|_{s^r_{p,q}b(\Gamma^d)} 
 \asymp \Big(  \sum_{\bj\in\N^{d}_{-1}} 2^{|\bj|_1 rq} \Big\| \sum_{\bk\in\Z^d}  |\lambda_{\bj,\bk}|  \chi_{\bj,\bk}(\cdot)  \Big\|^q_{L_p(\R^d)}  \Big)^{1/q} \,.
\end{align*}
\end{remark}

\noindent
A Hölder-type inequality is given by the following lemma.

\begin{lemma}\label{lem:Hölder_discr}
Let $0<p,q\leq \infty$, $r\in \R$, and $p',q',r'$ the conjugate parameters as in Definition~\ref{def:conj_space}.
Then
\begin{align*}
\sum_{(\bj,\bk)\in \Gamma^d} |\lambda_{\bj,\bk}\mu_{\bj,\bk}|
\le  \| \{\lambda_{\bj,\bk}\}_{\bj,\bk}\|_{s^r_{p,q}b(\Gamma^d)} \cdot \| \{\mu_{\bj,\bk}\}_{\bj,\bk}\|_{s^{r'+1}_{p',q'}b(\Gamma^d)} \,.
\end{align*}
\end{lemma}
\begin{proof}
The proof is a straight-forward calculation. We give it for $p,q,p',q'<\infty$. Then
\begin{align*}
\sum_{(\bj,\bk)\in\Gamma^d} |\lambda_{\bj,\bk}\mu_{\bj,\bk}|&\le \sum_{\bj\in\N_{-1}^d}
\Big( \sum_{\bk\in\Z^d} |\lambda_{\bj,\bk}|^p \Big)^{\frac{1}{p}} \Big( \sum_{\bk\in\Z^d} |\mu_{\bj,\bk}|^{p'} \Big)^{\frac{1}{p'}} \nonumber \\
&\le \Big( \sum_{\bj\in\N_{-1}^d} 2^{|\bj|_1(r-\frac{1}{p})q}
\Big( \sum_{\bk\in\Z^d} |\lambda_{\bj,\bk}|^p \Big)^{\frac{q}{p}} \Big)^{\frac{1}{q}} \cdot \Big(  \sum_{\bj\in\N_{-1}^d} 2^{|\bj|_1(\frac{1}{p}-r)q'} \Big( \sum_{\bk\in\Z^d} |\mu_{\bj,\bk}|^{p'} \Big)^{\frac{q'}{p'}} \Big)^{\frac{1}{q'}} \nonumber \,. 
\end{align*} 
Due to $\frac{1}{p}-r = 1 + r^{\prime} - \frac{1}{p'}$ the assertion follows.
\end{proof}

\subsubsection{Restriction to a subdomain}

Let us next turn to Besov spaces of dominating mixed smoothness on subdomains $D\subset\R^d$, denoted by $S^{r}_{p,q}B(D)$. Those 
are commonly defined in an extrinsic fashion, i.e.\ by restricting distributions $f\in S^{r}_{p,q}B(\R^{d})$ to $D$. In this paper we will a-priori only consider
the case 
$r > \sigma_{p}$
with 
\begin{align}\label{eqdef:sigma_p}
\sigma_{p} := \big(\tfrac{1}{p} - 1\big)_+ \,.
\end{align} 
Then $S^{r}_{p,q}B(\R^{d})$ is comprised of regular distributions $f\in L_1^{loc}(\R^d)\cap \mathcal{S}^{\prime}(\R^{d})$ by Proposition~\ref{prop:embed_regular}, which 
simplifies the definition of $S^{r}_{p,q}B(D)$ since the restriction operation 
\begin{align}\label{eqdef:gen_restriction}
\mathcal{R}_D:\; f\mapsto f|_{D} 
\end{align}
is naturally defined on functions.
Specifically, we are interested in functions on the unit cube $[0,1]^d$.
Note that the restriction~\eqref{eqdef:gen_restriction} of $f\in  L^{loc}_{1}(\R^{d})$ to $D=[0,1]^d$ is realized by the operator $\mathcal{R}$ from~Definition~\ref{def:RestrictionOP} and leads to
a function in $L_{1}([0,1]^{d})$.  
The scale we consider can hence be defined as follows.

\begin{definition}[{cf.~\cite[App.~B]{SickUll2011}}]\label{def:BesovDomain}
	Let $0 < p,q \leq \infty$, $r > \sigma_{p}$. The space $S^{r}_{p,q}B([0,1]^{d})$ consists of all functions $f \in L_{1}([0,1]^{d})$ for which an extension $F \in S^{r}_{p,q}B(\R^{d})$ exists with $f=F\vert_{[0,1]^{d}}$. The expression
	\begin{equation*}
		\norm{f}_{S^{r}_{p,q}B([0,1]^{d})}
		:= \inf_{\substack{F \in S^{r}_{p,q}B(\R^{d}) \\ F\vert_{[0,1]^{d}} = f}} \norm{F}_{S^{r}_{p,q}B(\R^{d})}
	\end{equation*}
	is a (quasi-)norm in $S^{r}_{p,q}B([0,1]^{d})$.
\end{definition}

\noindent
 Due to the extrinsic definition these spaces inherit many properties from the spaces $S^{r}_{p,q}B(\R^{d})$, for example their completeness, which follows from standard arguments from functional analysis regarding quotient spaces.
Also the proposition below is a direct consequence of the extrinsic definition of $S^{r}_{p,q}B([0,1]^{d})$.

\begin{prop}\label{prop:restriction1}
	For $0<p,q\le\infty$, $r>\sigma_p$, the restriction $\mathcal{R}:f\mapsto f|_{[0,1]^d}$ from Definition~\ref{def:RestrictionOP} is a bounded operator
	\begin{align*}
		\mathcal{R}:\; S^{r}_{p,q}B(\R^{d}) \to S^{r}_{p,q}B([0,1]^d) \,.
	\end{align*}
\end{prop}

\noindent
For more information on Besov spaces on bounded domains $D\subset\R^d$ we refer to e.g.~\cite[Sec.~1.2.6]{Tr10a}).

\subsubsection{Periodic setting}

At last, we recall the definition of the Besov spaces $S^{r}_{p,q}B(\T^{d})$ of dominating mixed smoothness on the torus $\T^d$ (see e.g.~\cite[Sec.~4]{TDiff06} or~\cite[Ch.~3]{Schmeisser1987}). 
As before, we only consider
the range $r>\sigma_p$, where they are, according to~\cite[Lem.~3.2.2 \&~Rem.~4.2.1]{TDiff06}, embedded in $L_{1}(\T^{d})$.  We hence only deal with 
proper functions and not general distributions.
It is important to remark that there are two common ways to view the space $S^{r}_{p,q}B(\T^{d})$.
On the one hand, its elements are considered as functions on $\T^d$. If we use the identification $\T^{d} \cong [-1,1]^d$ those are functions on the domain $[-1,1]^d$. In an alternative viewpoint, those can also be understood as periodic functions on $\R^d$, however, with the period $[-1,1]^d$. 
Both viewpoints have their merits and we will notationally not distinguish between them.  

Now we come to the definition of $S^{r}_{p,q}B(\T^{d})$.
Since every function $f \in L_{1}(\T^{d})$ corresponds to a periodic function in $L^{loc}_{1}(\R^{d})\cap \mathcal{S}^{\prime}(\R^{d})$,
the Littlewood-Paley blocks~\eqref{eqdef:LPops} are well-defined and the definition below makes sense.

\begin{definition}\label{def:PerBesov}
Let $0 < p,q \leq \infty$, $r > \sigma_{p}$, and $\{\Phi_{\bj}\}_{\bj \in \N_{0}^{d}}$ a family of kernels associated to $\{\varphi_{\bj}\}_{\bj \in \N_{0}^{d}} \in\Phi_{\rm hyp}(\R^d)$. Then $S^{r}_{p,q}B(\T^{d})$ is defined as the space of all $f \in L_{1}(\T^{d})$ for which
	\begin{equation*}
		\norm{f}_{S^{r}_{p,q}B(\T^{d})}
		:= \begin{cases}
			\displaystyle \Big( \sum_{\bj \in \N_{0}^{d}} 2^{r q |\bj|_{1}} \norm{\Phi_{\bj} \ast f}_{L_{p}(\T^{d})}^{q} \Big)^{1/q} \, , \quad &q < \infty \; , \\[4ex]
			\displaystyle \sup_{\bj \in \N_{0}^{d}} 2^{r |\bj|_{1}} \norm{\Phi_{\bj} \ast f}_{L_{p}(\T^{d})}  \; , \quad &q = \infty \,,
		\end{cases}
	\end{equation*}
	is finite. The expression $\norm{f}_{S^{r}_{p,q}B(\T^{d})}$ constitutes a (quasi-)norm in $S^{r}_{p,q}B(\T^{d})$.
\end{definition}

\noindent
As the spaces on the cube, these periodic spaces share many properties with $S^{r}_{p,q}B(\R^{d})$.
For instance, their definition is independent of the chosen partition of unity $\{\varphi_{\bj}\}_{\bj}\in\Phi_{\rm hyp}(\R^d)$ in
the sense of equivalent (quasi-)norms. Also, they are $u$-Banach spaces for $u:=\min\{1,p,q\}$.

    In the literature, see e.g.~\cite[Def.~3.3.3]{DuTeUl2018} or~\cite[Def.~4.2.1]{TDiff06}, another form of the Littlewood-Paley blocks $\Phi_{\bj} \ast f$ is often used.
	To derive this form, we first note that
	every function $f\in L_1(\T^d)$ can be represented by a distributional Fourier series
	\[
	f = \sum_{\bk\in\Z^d}  \hat{f}(\bk) \exp_{\bk} 
	\] 
	with bounded coefficients $\{\hat{f}(\bk)\}_{\bk}\in\ell_\infty(\Z^d)$. The coefficients are thereby given by $\hat{f}(\bk)=\langle f, \overline{\exp}_{\bk} \rangle_{\T^d}$ with exponentials from the system $\mathcal{E}([-1,1]^d)$ defined in~\eqref{eqdef:exp_system}. Since the Fourier representation of $f$ is  weak* convergent in $\mathcal{S}^{\prime}(\R^{d})$, we can calculate
	\begin{equation*}
		\Phi_{\bj} \ast f(x)
		= \br{f, \Phi_{\bj}(x-\cdot)}_{\mathcal{S}^\prime\times \mathcal{S}}
		= \sum_{\bk \in \Z^{d}} \hat{f}(\bk) \br{\exp_{\bk}, \Phi_{\bj}(x-\cdot)}_{\mathcal{S}^\prime\times \mathcal{S}}
		= \sum_{\bk \in \Z^{d}} \hat{f}(\bk) \, \varphi_{\bar{j}} \! \big(\bk /2 \big) \exp_{\bk}(x)  \,.
	\end{equation*}	 
    The last equality is due to~\eqref{eqdef:ass_kernels} and
    \[
    \br{\exp_{\bk}, \Phi_{\bj}(x-\cdot)}_{\mathcal{S}^\prime\times \mathcal{S}} = \int_{\R^d}  \Phi_{\bj}(x-\xi) \exp_{\bk}(\xi) \,\dint\xi  
    = \int_{\R^d} 2^{-d/2}  \Phi_{\bj}(\xi) \exp(\pi\ii \bk\cdot(x-\xi)) \,\dint\xi = \widehat{\Phi}_{\bar{j}}  \! \big(\bk /2 \big) \exp_{\bk}(x) \,.
    \]
    Setting $\psi_{\bj}:=\varphi_{\bj}(\cdot/2)$ we can thus write the Littlewood-Paley blocks~\eqref{eqdef:LPops} in the alternative form
	\begin{align}\label{buildingblocks}
	f_{\bj} = \sum_{\bk\in\Z^{d}} \psi_{\bj}(\bk)\hat{f}(\bk) \exp_{\bk} \,, 
	\end{align}
  where $\{\psi_{\bj}\}_{\bj \in \N_{0}^{d}}$ still belongs to $\Phi_{\rm hyp}(\R^d)$. Since the definition of $S^{r}_{p,q}B(\T^{d})$ is independent of the utilized family in $\Phi_{\rm hyp}(\R^d)$, we can substitute $\psi_{\bj}$ by $\varphi_{\bj}$ in~\eqref{buildingblocks} and use this sum as an alternative to $\Phi_{\bj} \ast f$ in the definition of the (quasi-)norm of $S^{r}_{p,q}B(\T^{d})$.

Let us finally have a look at the operator $\mathcal{R}: f\mapsto f|_{[0,1]^d}$ from Definition~\ref{def:RestrictionOP}. Utilizing the identification
$\T^d\cong[-1,1]^d$, it can be applied to functions $f\in L_1(\T^d)$.
Recall Lemma~\ref{lem:RestrictionOP}.
An important mapping property of $\mathcal{R}$ for functions
$f\in S^{r}_{p,q}B(\T^d)$ is given in Proposition~\ref{prop:restriction2}.
But let us first recall a result, which we will use in its proof.

\begin{lemma}[{\cite[Thm.~1.3]{NUUChangeVariable2015}}]\label{lem:point_multiplier}
	Let $0 < p,q \le \infty$,
	$r > \sigma_p$. Let further $\psi\in C_c^k(K)$ be a function with 
	$k\ge \lfloor r \rfloor+1$
	for some compact set $K\subset \R^d$. 
	Then for all $f\in S^{r}_{p,q}B(\T^d)$
	\begin{align*}
		\| \psi\cdot f  \|_{S^{r}_{p,q}B(\R^d)} \lesssim \|  f  \|_{S^{r}_{p,q}B(\T^d)} \,.
	\end{align*}
\end{lemma}

\noindent
From this lemma we can deduce Proposition~\ref{prop:restriction2}, which is analogous
to Proposition~\ref{prop:restriction1}. 
It concludes our introductory part on 
classical Besov spaces.

\begin{prop}\label{prop:restriction2}
	Let $0<p,q\le\infty$, $r>\sigma_p$. With the identification $\T^d \cong [-1,1]^d$, it holds that
	\begin{align*}
		\mathcal{R}:\; S^{r}_{p,q}B(\T^d) \to S^{r}_{p,q}B([0,1]^d)
	\end{align*}
	is bounded.
\end{prop}
\begin{proof}
	We interpret $f\in S^{r}_{p,q}B(\T^d)$ as a periodic function on $\R^d$ and choose $\psi\in C_c^\infty(\R^d)$ such that $\psi(x)=1$ on $[0,1]^d$. With Lemma~\ref{lem:point_multiplier} we obtain
	\begin{gather*}
		\| f \|_{S^{r}_{p,q}B([0,1]^d)}  
		\le \| \psi\cdot f \|_{ S^{r}_{p,q}B(\R^d)}
		\lesssim \| f \|_{S^{r}_{p,q}B(\T^d)} \,. \qedhere
	\end{gather*}
\end{proof}



\subsection{Half-period cosine spaces of dominating mixed Besov-type}

The function spaces,
we are about to define in this subsection, are new and represent
a generalization of half-period cosine spaces considered earlier in the literature, e.g.\ in~\cite{Dick2013, goda2019lattice}.
They are of dominating mixed Besov-type and shall be denoted by $\hpc$. Their 
definition closely resembles that of the periodic spaces 
$S^{r}_{p,q}B(\T^{d})$.

\subsubsection{Definition}

We restrict to spaces comprised of regular distributions, i.e.\ distributions that can be interpreted as proper functions. Since those are locally integrable, we can assume $f \in L_{1}([0,1]^{d})$.
Based on a decomposition of unity
$\{\varphi_{\bj}\}_{\bj\in\N_{0}^{d}} \in\Phi_{\rm hyp}(\R^d)$, we then
define the half-period cosine building blocks of $f$ by
\begin{equation}\label{f2}
	f^{\mathrm{hpc}}_{\bj} := \sum\limits_{\bk\in \N_{0}^d}
	\varphi_{\bj}(\bk) \hat{f}_{\rm hpc}(\bk) c_{\bk}\,,\quad \bj\in \N_{0}^d \,,
\end{equation}
with the half-period cosine coefficients $\hat{f}_{\rm hpc}(\bk) := \langle f, c_{\bk}\rangle$ as in~\eqref{eqdef:hpc_coeff}.
The following lemma holds true.

\begin{lemma}\label{lem:LPexpansionHPC}
For any $f \in L_{1}([0,1]^{d})$ we have the identity   
    \begin{equation*}
        f = \sum_{\bj\in\nd}  f^{\mathrm{hpc}}_{\bj} \quad \text{weak*ly in } \mathcal{S}^{\prime}(\R^{d}) \,.
    \end{equation*}   
\end{lemma}
\begin{proof}
    Starting from Lemma \ref{Lemma HPC Reihe konvergiert im Distributionensinn}, we see that in the sense of tempered distributions we have  
    \begin{equation*}
        \begin{aligned}
            f
            = \sum_{\bk \in \N_{0}^{d}} \br{f, c_{\bk}} c_{\bk} 
            = \sum_{\bk \in \N_{0}^{d}} \sum_{\bj \in \N_{0}^{d}} \varphi_{\bj}(\bk) \br{f, c_{\bk}} c_{\bk} 
            = \sum_{\bj \in \N_{0}^{d}} \sum_{\bk \in \N_{0}^{d}} \varphi_{\bj}(\bk) \br{f, c_{\bk}} c_{\bk} 
            = \sum_{\bj\in\nd}  f^{\mathrm{hpc}}_{\bj} \,.
        \end{aligned}
    \end{equation*}    
    This calculation is justified by the fact that for every $\eta \in \mathcal{S}(\R^{d})$ the series   
    \begin{equation*}
        \sum_{\bj \in \N_{0}^{d}} \sum_{\bk \in \N_{0}^{d}} \varphi_{\bj}(\bk) \br{f, c_{\bk}} \br{c_{\bk}, \eta}
    \end{equation*}
    is absolutely convergent.   
\end{proof}

\noindent
The definition of $f^{\mathrm{hpc}}_{\bj}$ in~\eqref{f2} is inspired by the form~\eqref{buildingblocks} of the Littlewood-Paley blocks $f_{\bj}$ and the expansion of $f$ in Lemma~\ref{lem:LPexpansionHPC} is analogous to the Littlewood-Paley expansion~\eqref{LPdecomp}.
The half-period cosine spaces  $S^{r}_{p,q}B_{\mathrm{hpc}}([0,1]^d)$  
of dominating mixed Besov-type 
are now defined in analogy to the periodic spaces $S^{r}_{p,q}B(\T^{d})$.

\begin{definition}[Half-period cosine spaces of dominating mixed Besov-type]\label{d1} 
    Let $\mathcal{C}_d=\{ c_{\bk} \}_{\bk\in\N^d_{0}}$ be the half-period cosine system given in \eqref{def:hpc_system}. Further, let $0< p,q\leq
    \infty$ and $r>\sigma_p$. Then $\hpc$ is defined
as the collection of all $f\in L_1([0,1]^d)$ such that
    \begin{equation*}
    \norm{f}_{\hpc}
    := \begin{cases}
    	\displaystyle \Big( \sum_{\bj \in \N_{0}^{d}} 2^{r q |\bj|_{1}} \big\|\fj\big\|_{L_{p}([0,1]^{d})}^{q} \Big)^{1/q} \; , \quad &q < \infty \, , \\[4ex]
    	\displaystyle \sup_{\bj \in \N_{0}^{d}} 2^{r |\bj|_{1}} \big\|\fj\big\|_{L_{p}([0,1]^{d})}  \; , \quad &q = \infty  \,,
    \end{cases}
    \end{equation*}
    is finite. The expression $\norm{f}_{\hpc}$ constitutes a (quasi-)norm in $\hpc$.
\end{definition} 

\noindent
It is not clear right from Definition~\ref{d1} that these spaces are independent of the chosen system $\{\varphi_{\bj}\}_{\bj}\in \Phi_{\rm hyp}(\R^d)$ in~\eqref{f2}. Also
other properties, such as the (quasi-)Banach space property, are not immediate.
However, these spaces fulfill analogous properties as their classical Besov space counterparts. This will be confirmed as a side result of Theorem~\ref{thm:relation_to_even_periodic} proved in the next subsection.

\subsubsection{Identification}

Our first main result on half-period cosine spaces, Theorem~\ref{thm:relation_to_even_periodic}, establishes an isomorphism that is essential for the periodization principle. It can be viewed as a refinement of Lemma~\ref{lem:PeriodizationOP}.

\begin{theorem}\label{thm:relation_to_even_periodic}
	Let $0<p,q\le\infty$, $r>\sigma_p$. Then 
	\[
	\mathcal{P}:\; \hpc \cong S^r_{p,q}B(\T^d)_{\rm even} \,.
	\]
\end{theorem}
\begin{proof}
Let $\{\varphi_{\bj}\}_{\bj}\in\Phi_{\rm hyp}(\R^d)$ be a fixed decomposition of unity and $\hpc$ the associated half-period cosine space. The functions $\varphi_{\bj}$ have the structure
$\varphi_{\bj}=\varphi^{(1)}_{j_1} \otimes \ldots \otimes\varphi^{(d)}_{j_d}$ with dyadic partitions of unity $\{\varphi^{(i)}_{j}\}_{j\in\N_0}$ for each $i\in[d]$. 
For $j\in\N_{0}$ and $x\in\R$ we now define $\psi^{(i)}_{j}(x):=\varphi^{(i)}_{j}(|x|)$ and put $\psi_{\bj}:=\psi^{(1)}_{j_1} \otimes \ldots \otimes\psi^{(d)}_{j_d}$. Then $\{\psi_{\bj}\}_{\bj} \in\Phi_{\rm hyp}(\R^d)$ and this decomposition is symmetric in the sense 
\begin{align}\label{feat:symmetry}
	\psi_{\bj}(x_1,\ldots,x_d) = \psi_{\bj}(\pm x_1,\ldots,\pm x_d) \quad\text{for all}\quad (x_1,\ldots,x_d)\in\R^d \,.
\end{align} 
Let 
$\tilde{f}:=\mathcal{P}f$ be the periodization of $f\in\hpc$ on $\T^d\cong[-1,1]^d$ by reflection according to Definition~\ref{def:PeriodizationOP}. 
Since $S^r_{p,q}B(\T^d)$ is independent of the utilized
decomposition of unity, up to equivalence of (quasi-)norms, we have 
\begin{align*}
\|\tilde{f}\|_{S^r_{p,q}B(\T^d)}  \asymp 
    \Big(\sum\limits_{\bj\in \N_0^d}
2^{|\bj|_1rq}\|\tilde{f}_{\bj}\|_{L_p(\T^d)}^q\Big)^{1/q} 
\end{align*}
with Littlewood-Paley blocks $\tilde{f}_{\bj}$ given by
\begin{equation*}
	\tilde{f}_{\bj} := \sum\limits_{\bk\in \Z^d}
	\psi_{j_1}(k_1)\cdots\psi_{j_d}(k_d)\langle \tilde{f}, \overline{\exp}_{\bk} \rangle_{\T^d} \exp_{\bk} \,,\quad \bj\in\nd \,.
    \end{equation*}

\noindent
By Lemma~\ref{lem:aux_Fourier}, it holds $\langle \tilde{f}, \overline{\exp}_{\bk} \rangle_{\T^d} = 2^{(d-|\bk|_0)/2} \langle f, c_{|\bk|}\rangle$ for every $\bk\in \Z^d$.
Taking into account~\eqref{feat:symmetry}, we get
\begin{equation*}
	\tilde{f}_{\bj} = \sum\limits_{\bk\in \nd}
	\varphi_{j_1}(k_1)\cdots\varphi_{j_d}(k_d) \cdot 2^{(d-|\bk|_0)/2} \cdot \langle f, c_{\bk} \rangle \cdot  2^{-(d-|\bk|_0)} \cdot \sum_{\beps\in\{-1,1\}^d} \exp_{(\varepsilon_1k_1,\ldots,\varepsilon_dk_d)} \,, 
    \end{equation*}
and for $x=(x_1,\ldots,x_d)\in\T^d$
\begin{align*}
    \sum_{\beps\in\{-1,1\}^d} \exp_{(\varepsilon_1k_1,\ldots,\varepsilon_dk_d)}(x) &= 2^{-d/2} \sum_{\beps\in\{-1,1\}^d}  \prod_{j=1}^d \exp(\ii\pi \varepsilon_j k_j x_j) \\
    &= 2^{-d/2} \big( \exp(\ii\pi k_d x_d) + \exp(-\ii\pi k_d x_d)\big) \sum_{\beps\in\{-1,1\}^{d-1}}  \prod_{j=1}^{d-1} \exp(\ii\pi \varepsilon_j k_j x_j) \\
    &= 2^{1-d/2} \cos(\pi k_d x_d) \sum_{\beps\in\{-1,1\}^{d-1}}  \prod_{j=1}^{d-1} \exp(\ii\pi \varepsilon_j k_j x_j)
    = \ldots = 2^{d} \cos_{\bk}(x) \,.
    \end{align*}

\noindent
We thus arrive at
\begin{align*}
    \tilde{f}_{\bj}= 2^{d/2} \sum\limits_{\bk\in \nd}
	\varphi_{j_1}(k_1)\cdots\varphi_{j_d}(k_d) \cdot  2^{|\bk|_0/2} \cdot  \langle f, c_{\bk} \rangle \cdot  \cos_{\bk} \,,
\end{align*}
and further, using Lemma~\ref{lem:rel_ck_cosk},
\[
\tilde{f}_{\bj} = 2^{d/2} \sum\limits_{\bk\in \nd}
	\varphi_{j_1}(k_1)\cdots\varphi_{j_d}(k_d) \cdot  2^{|\bk|_0/2} \cdot  \langle f, c_{\bk} \rangle \cdot    2^{-(|\bk|_0+d)/2} \cdot \mathcal{P}(c_{\bk}) = \mathcal{P}(\fj) \,.
\]
Finally, due to the symmetry of $\tilde{f}_{\bj}=\mathcal{P}(\fj)$,
\begin{align*}
    \| \tilde{f}_{\bj} \|^p_{L_p(\T^d)}
     = \|\mathcal{P}(\fj) \|^p_{L_p(\T^d)}= 2^{d} \| \mathcal{P}(\fj) \|^p_{L_p([0,1]^d)} = 2^{d} \|\fj \|^p_{L_p([0,1]^d)}  \,.
\end{align*}
Note that the restriction operator $\mathcal{R}$ is here a left-inverse of the periodization $\mathcal{P}$.
\end{proof}

\noindent
Several structural properties of $\hpc$ can be easily deduced from the identification given in Theorem~\ref{thm:relation_to_even_periodic}.

\begin{itemize}[leftmargin=*]
\item
\textbf{(Quasi-)Banach space property}

We know that $S^r_{p,q}B(\T^d)_{\rm even}$, as a closed subspace of $S^r_{p,q}B(\T^d)$, is a (quasi-)Banach space. Applying Theorem~\ref{thm:relation_to_even_periodic} now directly
yields the (quasi-)Banach space property of $\hpc$.

\item
\textbf{Embeddings}

Starting from an embedding $S^{r_1}_{p_1,q_1}B(\T^d) \hookrightarrow S^{r_2}_{p_2,q_2}B(\T^d)$, restriction to the even functions first yields the embedding $S^{r_1}_{p_1,q_1}B(\T^d)_{\rm even} \hookrightarrow S^{r_2}_{p_2,q_2}B(\T^d)_{\rm even}$. Theorem~\ref{thm:relation_to_even_periodic} then
leads to $\hpc[r_1][p_1][q_1]\hookrightarrow\hpc[r_2][p_2][q_2]$.

\item
\textbf{Independence of the partition of unity}

To establish the independence of 
$\hpc$ from the utilized partition of unity, up to equivalence of
(quasi-)norms, 
let us start with two different partitions $\{\varphi_{\bj}\}_{\bj}$ and $\{\tilde{\varphi}_{\bj}\}_{\bj}$ from $\Phi_{\rm hyp}(\R^d)$ and let $\hpc^{\varphi}$ and $\hpc^{\tilde{\varphi}}$ denote the associated half-period cosine spaces.
As in the proof of Theorem~\ref{thm:relation_to_even_periodic}, we further construct
corresponding partitions $\{\psi_{\bj}\}_{\bj}$ and $\{\tilde{\psi}_{\bj}\}_{\bj}$ via 
$\psi_{\bj}:=\varphi_{\bj}(|\cdot|)$ and $\tilde{\psi}_{\bj}:=\tilde{\varphi}_{\bj}(|\cdot|)$. Letting $S^r_{p,q}B(\T^d)^{\psi}$ and $S^r_{p,q}B(\T^d)^{\tilde{\psi}}$
denote the associated Besov spaces, we then have
\begin{align*}
\hpc^{\varphi} \cong S^r_{p,q}B(\T^d)_{\rm even}^{\psi} \cong S^r_{p,q}B(\T^d)_{\rm even}^{\tilde{\psi}} \cong \hpc^{\tilde{\varphi}}   
\end{align*}
since each partition from $\Phi_{\rm hyp}(\R^d)$ yields the same space $S^r_{p,q}B(\T^d)$, up to equivalence of (quasi-)norms.
\end{itemize}

\subsection{Applications of Theorem~\ref{thm:relation_to_even_periodic}}
\label{ssec:app_theo_id}

Let us finally take the time to demonstrate the application of Theorem~\ref{thm:relation_to_even_periodic} in conjunction with the periodization principle established in Section~\ref{sec:hpc_system}. For this note that, utilizing the identification $\T^d\cong[-1,1]^d$ and the restriction operator $\mathcal{R}$ from Definition~\ref{def:RestrictionOP},
Theorem~\ref{thm:relation_to_even_periodic} can be restated in the form
\begin{align}\label{restate_Thm}
\hpc \cong \mathcal{R}\big(S^r_{p,q}B(\T^d)_{\rm even}\big) \,.
\end{align}
Building on Lemma~\ref{app:approx} and Lemma~\ref{app:cub}, respectively, we give
two examples for the translation mechanism from $S^r_{p,q}B(\T^d)_{\rm even}$ to $\hpc$.
The first example is half-period cosine approximation, where Theorem~\ref{thm:least_squares} provides an upper bound for the convergence rate of least-squares approximation from point samples, the second example is cubature.

\subsubsection{Half-period cosine approximation in $\hpc$}

In view of~\eqref{restate_Thm} and Lemma~\ref{app:approx},
a half-period cosine approximation scheme $\{A^{(\sigma)}_{{\rm hpc},N}\}_{\raisebox{-0.6ex}{$\scriptstyle{N\in\N}$}}$ in $\hpc$ performs as well as the corresponding Fourier approximation scheme $\{A^{(\sigma)}_{N}\}_{\raisebox{-0.6ex}{$\scriptstyle{N\in\N}$}}$ in $S^r_{p,q}B(\T^d)_{\rm even}$ (see Subsection~\ref{ssec:further_aspects}).
The performance of the latter in $S^r_{p,q}B(\T^d)$ thus provides an upper bound for the performance of $\{A^{(\sigma)}_{{\rm hpc},N}\}_{\raisebox{-0.6ex}{$\scriptstyle{N\in\N}$}}$ in $\hpc$.
Fourier approximation in $S^{r}_{p,q}B(\T^d)$ is well-understood.
Optimal approximation schemes are obtained for the choice $\sigma(N):=\{ |\bk| : \bk\in\Gamma_N\}$, where
$\Gamma_N\subset\Z^d$ are the hyperbolic cross frequencies at level $N$ (see~\cite[page~21]{DuTeUl2018}),
\begin{align*}
\Gamma_N := \Big\{ \bk=(k_1,\ldots,k_d)\in\Z^d ~:~  \prod_{i=1}^{d} (1 + |k_i|)\le N \Big\} \quad\text{with}\quad |\Gamma_N|\asymp N \log(N)^{d-1} \,.  
\end{align*} 
The optimal rates for orthogonal $L_q$ approximation are given by the orthowidths (see~\cite[page~47]{DuTeUl2018}), 
\begin{align*}
    \varphi_n(S^{r}_{p,\theta}B(\T^d),L_q) := \inf\limits_{\substack{u_1,\ldots,u_n \\ \text{orthonormal}}}\sup\limits_{f\in U^{r}_{p,\theta}(\T^d)} \Big\| f - \sum\limits_{j=1}^{n} \langle f, \overline{u}_j\rangle_{\T^d}  u_j\Big\|_{L_q(\T^d)} \,,
\end{align*}
where $U^{r}_{p,\theta}(\T^d)$ denotes the unit ball in $S^{r}_{p,\theta}B(\T^d)$.
Consequently, the rates determined in~\cite[Sect.~4]{DuTeUl2018}
for these quantities 
are upper bounds for the corresponding half-period cosine approximation in $\hpc$.  
Taking Theorem~\ref{thm:identification_hpc} from the next section into consideration, these upper bounds also apply to functions $f\in  S^{r}_{p,q}B([0,1]^d)$ as long as $\sigma_p<r<\min\{1+\tfrac{1}{p},2\}$.

For $f\in H^{r}_{\rm mix}([0,1]^d)= S^{r}_{2,2}B([0,1]^d)$ in the range $0<r<\tfrac{3}{2}$ this implies that the approximation scheme $\{A^{(\sigma)}_{{\rm hpc},N}\}_{\raisebox{-0.6ex}{$\scriptstyle{N\in\N}$}}$ yields the optimal $L_2$ approximation rate of $N^{-r}$, i.e.\ $n^{-r}(\log{n})^{r(d-1)}$ expressed in the dimension $n$ of the space $V_{\mathrm{hpc},N}:=\Span \{ c_{|\bk|} :\bk \in \Gamma_N \}$.
In the boundary case $r=\tfrac{3}{2}$, $p=2$, one gets a rate of $n^{-3/2}(\log{n})^{2(d-1)}$ for $f\in S^{3/2}_{2,1}B([0,1]^d)$. 
Indeed, if we let
$
\alpha(r,p,\theta)_n:=  \varphi_n(S^{r}_{p,\theta}B(\T^d),L_2)
$,
the above rates follow directly from  
\begin{align*}
\alpha(r,2,2)_n \asymp n^{-r}(\log n)^{r(d-1)}
\quad\text{and}\quad
\alpha(\tfrac{3}{2},2,\infty)_n\asymp n^{-3/2}(\log n)^{2(d-1)} \,,
\end{align*} 
established in~\cite[Thm.~4.4.6 \&~Thm.~4.4.4]{DuTeUl2018}, and the embedding $S^{3/2}_{2,1}B([0,1]^d)\hookrightarrow\hpc[3/2][2][\infty]$ in Theorem~\ref{thm:endpoint_result}.

\subsubsection{Least-squares approximation}

Using~\cite[Prop.~1]{KrPozhUllUll23}, we can further derive the existence of associated least-squares reconstruction schemes. To this end, we choose a sequence of subspaces $V_n:=\Span\{\exp(2\pi\ii\bk\cdot):|\bk|\in H_n\}$ 
of $L_2(\T^d)$, where $\T^d\cong[0,1]^d$ and the index sets $H_n\subset\N_{0}^d$ are chosen according to a hyperbolic cross pattern and, in addition, satisfy $|H_n|=n$.
We further define the operator
\begin{align*}
\mathcal{E}:\; L_1([0,1]^d)\to L_1([0,1]^d) \;,\quad  f \mapsto 2^{-d} \sum\limits_{\beps\in\{0,1\}^d} f(\omega_{\varepsilon_1}(x_1),\ldots,\omega_{\varepsilon_d}(x_d)) 
\end{align*}
with
$\omega_0(x):=x$ and $\omega_1(x):=1-x$ for $x\in[0,1]$.
The operator $\mathcal{E}$
makes any function $f\in L_1([0,1]^d)$ `even' in the sense that
$(\mathcal{E}f)(x_1,\ldots,x_d)=(\mathcal{E}f)(\omega_{\varepsilon_1}(x_1),\ldots,\omega_{\varepsilon_d}(x_d))$ for every $\beps\in \{0,1\}^d$. This property corresponds to $d$-fold evenness on $[-1,1]^d$.

For subspaces $V\subset L_2([0,1]^d)$, let us denote the associated orthogonal projection by $\Pi_{V}$.
We then have, for any even $f\in L_2([0,1]^d)$, 
\begin{align*}
\| f - \Pi_{\mathcal{E}V_n} f\|_{L_2([0,1]^d)} \le \| f - \mathcal{E}\Pi_{V_n} f\|_{L_2([0,1]^d)} \le \| f - \Pi_{V_n} f\|_{L_2([0,1]^d)} \,.
\end{align*}
Since the image $f^\phi$ of a function $f\in L_1([0,1]^d)$ under the tent transform $\phi$ is even on $[0,1]^d$, we may define $W_n:=\phi^{-1}\mathcal{E}V_n$ and thus get a sequence $\{W_n\}_{n\in\N}$ of nested subspaces in $L_2([0,1]^d)$ with $\dim(W_n)=n$. The spaces $W_n$ have the form
$W_n=\Span\{c_{\bk}:\bk \in H_n \}$.
Since $(\Pi_{W_n} f)^\phi = \Pi_{\mathcal{E}V_n} f^\phi$ and $\|f^\phi\|_{L_2([0,1]^d)} = \|f\|_{L_2([0,1]^d)}$, it holds for any $f\in L_2([0,1]^d)$
\begin{align}\label{eq:aux_est_proj}
\| f - \Pi_{W_n} f\|_{L_2([0,1]^d)} = \| f^\phi - \Pi_{\mathcal{E}V_n} f^\phi\|_{L_2([0,1]^d)} \le \| f^\phi - \Pi_{V_n} f^\phi\|_{L_2([0,1]^d)} \,.
\end{align}
Next, with $U^{3/2}_{2,1}([0,1]^d)$ the unit ball in $S^{3/2}_{2,1}B([0,1]^d)$ and $U^{3/2}_{2,\infty}([0,1]^d)_{\mathrm{hpc}}$ the unit ball in $\hpc[3/2][2][\infty]$, we derive with Theorem~\ref{thm:endpoint_result}
\begin{align*}
\beta_n := \sup_{f\in U^{3/2}_{2,1}([0,1]^d)} \|f - \Pi_{W_n}f\|_{L_2([0,1]^d)} \lesssim \sup_{f\in U^{3/2}_{2,\infty}([0,1]^d)_{\mathrm{hpc}}} \|f - \Pi_{W_n}f\|_{L_2([0,1]^d)} \,.
\end{align*}
Let $U^{3/2}_{2,\infty}(\T^d)$ denote the unit ball in $S^{3/2}_{2,1}B(\T^d)$ and $U^{3/2}_{2,\infty}(\T^d)_{\rm even}$ its subset of even functions. 
With \eqref{eq:aux_est_proj}, then
\begin{align*}
\beta_n \lesssim \sup_{f\in U^{3/2}_{2,\infty}(\T^d)_{\rm even}} \|f -  \Pi_{V_n} f\|_{L_2(\T^d)} \le \sup_{f\in U^{3/2}_{2,\infty}(\T^d) } \|f - \Pi_{V_n} f\|_{L_2(\T^d)} 
\end{align*}
and, due to the elaborations in~\cite[Sect.~4]{DuTeUl2018} and $n\le\dim(V_n)\le 2^d n$,
\begin{align*}
\sup_{f\in U^{3/2}_{2,\infty}(\T^d) } \|f - \Pi_{V_n} f\|_{L_2(\T^d)} \asymp  \alpha(\tfrac{3}{2},2,\infty)_{\dim(V_n)} \lesssim \alpha(\tfrac{3}{2},2,\infty)_{n} \,.
\end{align*}
Now we can apply~\cite[Prop.~1]{KrPozhUllUll23}. This yields the following theorem.

\begin{theorem}\label{thm:least_squares}
There is an absolute constant $c>0$ such that for all $m\in\N$ 
there are point sets $\{x_j\}_{j=1}^{n}\subset[0,1]^d$ and functions $\{\varphi_j\}_{j=1}^{n} \subset W_{n}$ with $n\le cm$ such that 
\begin{align*}
    \sup_{f\in U^{3/2}_{2,1}([0,1]^d)} \|f - A_m (f)\|_{L_2([0,1]^d)} \le \frac{70}{\sqrt{m}}  \sum_{k>\lfloor m/4\rfloor} \frac{\beta_k}{\sqrt{k}} \lesssim  m^{-3/2}(\log m)^{2(d-1)} \,,
\end{align*}
where $A_m(f):=\sum_{i=1}^{n} f(x_i)\varphi_i$ is a weighted least-squares reconstruction algorithm in the space $W_{n}$ 
\end{theorem}

\noindent
Note that this theorem in particular applies to functions $f\in H^{r}_{\rm mix}([0,1]^d)$ with $r>\tfrac{3}{2}$ due to the embedding $H^{3/2+\varepsilon}_{\rm mix}([0,1]^d)\hookrightarrow S^{3/2}_{2,1}B([0,1]^d)$ for any $\varepsilon>0$. For more details on the algorithm we refer to~\cite[Rem.~7]{KrPozhUllUll23}.

\subsubsection{Tent-transformed cubature}

Turning to tent-transformed cubature, we learn  
from  Lemma~\ref{app:cub} and Theorem~\ref{thm:relation_to_even_periodic}
that tent-transformed rules work as well on $\hpc$ as the untransformed rules on $S^r_{p,q}B(\T^d)_{\rm even}$. An upper performance bound is thus provided by
the convergence rate of the untransformed rules on $S^r_{p,q}B(\T^d)$.
We use this insight to translate~\cite[Thm.~3.1]{DiUl14} and~\cite[Thm.~5.3]{HinMarkOettUll2016} on
cubature in $S^r_{p,q}B(\T^d)$ to tent-transformed cubature in $\hpc[r][p][q][[0,1]^d]$.
The first result~\cite[Thm.~3.1]{DiUl14} deals with Fibonacci cubature in $S^r_{p,q}B(\T^2)$. 
Recall that a Fibonacci rule  
$\Phi_n(f):=\frac{1}{b_n}\sum_{j=1}^{b_n}f(x_j)$ uses the $n$th Fibonacci lattice $F_n:=\{x_1,\ldots,x_{b_n}\}$ as node set in $\T^2$, where $b_n$ is the $n$th Fibonacci number.
According to~\cite[Thm.~3.1]{DiUl14}, these rules provide the optimal rate $b_n^{-r}(\log b_n)^{(1-1/q)_+}$ in the periodic space $S^{r}_{p,q}B(\T^2)$ if $1\le p\le\infty$, $0<q\le\infty$, and $r>\tfrac{1}{p}$.
The associated tent-transformed rules thus provide at least the same rate
in $\hpc[r][p][q][[0,1]^2]$, and due to Theorem~\ref{thm:identification_hpc} and the embedding $S^{r}_{p,q}B([0,1]^2)\hookrightarrow S^{r}_{1,q}B([0,1]^2)$ from Proposition~\ref{prop:emb}
these rates also hold in $S^{r}_{p,q}B([0,1]^2)$ in the range
$1\le p\le\infty$, $0<q\le\infty$, and $\tfrac{1}{p}<r<2$.
To deduce a similar result for $d>2$, we resort to digital nets since for general parameter constellations optimal lattice constructions are not known for the periodic spaces $S^{r}_{p,q}B(\T^d)$. 
As shown in~\cite[Thm.~5.3]{HinMarkOettUll2016}, restricting to $1\le p,q\le\infty$ and $\tfrac{1}{p}<r<2$, nets of order $2$ exist with the optimal rate $n^{-r}(\log n)^{(d-1)(1-1/q)}$  in $S^{r}_{p,q}B(\T^d)$,
confer also~\cite[Rem.~5.4, Thm.~5.5]{HinMarkOettUll2016}.
Applying the translation mechanism as before, we now get
at least these rates for the tent-transformed digital nets
in $\hpc[r][p][q][[0,1]^d]$ as well as in $S^{r}_{p,q}B([0,1]^d)$. 
Let us finally consider the Hilbert space $H^{2}_{\rm mix}([0,1]^d)$. For arbitrarily small $\varepsilon>0$, by Theorem~\ref{thm:identification_hpc} and Proposition~\ref{prop:emb}, there holds
\[
H^{2}_{\rm mix}([0,1]^d)=S^{2}_{2,2}B([0,1]^d) \hookrightarrow S^{2-\varepsilon}_{1,1}B([0,1]^d) \cong \hpc[2-\varepsilon][1][1] \,.
\]
The tent-transformed digital nets from above, resp.\ the Fibonacci rules if $d=2$, thus provide at least $n^{\varepsilon-2}$ as convergence rate. This
fits with numerical observations and other theoretical results on tent-transformed cubature in $H^{2}_{\rm mix}([0,1]^d)$. Hickernell, in 2002, was the first to obtain this rate using randomly shifted lattices~\cite{Hick2002}. A similar result for randomly shifted digital nets~\cite{CrisDickLeoPill2007} was proved in 2007 by Cristea, Dick, Leobacher, and Pillichshammer. 
A notable recent result from 2019 is due to
Goda, Suzuki, and Yoshiki~\cite[Thm.~2, Cor.~1]{goda2019lattice}, who could also confirm it for tent-transformed lattice rules without shifting for all $d\in\N$, using a component-by-component construction. Due to the unknown performance of such lattices in $S^{2-\varepsilon}_{1,q}B(\T^d)$,
our analysis cannot to date reproduce this finding.
But it complements the picture.
The mentioned former results were obtained by Hilbert space methodology.
Our analysis, in contrast, is based on embeddings into the half-period cosine scale $\hpc$. Hereby, 
the necessity to include non-Hilbert spaces becomes apparent as e.g.\
$H^{r}_{\rm mix}([0,1]^d)$ with $\tfrac{3}{2}<r<2$ does not embed into $\hpc[r][2][2]$ but into the slightly larger space $\hpc[r][1][2]$, which gives the rate $n^{-r}(\log n)^{(d-1)/2}$.


\section{Relation of $\hpc$ to $S^r_{p,q}B([0,1]^d)$}
\label{sec:main}

We next analyze the relation of $\hpc$
to the classical non-periodic Besov space 
$S^r_{p,q}B([0,1]^d)$, the latter defined as restriction of the respective space on $\R^d$ (see Definition~\ref{def:BesovDomain}). Recall that our analysis is limited to the case $r>\sigma_p$.
Theorem~\ref{thm:relation_to_even_periodic} and Proposition~\ref{prop:restriction2} directly lead to the following theorem. 

\begin{theorem}\label{thm:embedding_hpc}
For parameters in the range $0<p,q\le \infty$, $\sigma_p<r$, we have the embedding
            \begin{align*}
           \operatorname{Id}:\;   \hpc \hookrightarrow  S^{r}_{p,q}B([0,1]^d) \,.
            \end{align*}
\end{theorem}
\begin{proof}
This is a consequence of $\operatorname{Id}=\mathcal{R}\circ\mathcal{P}$, Theorem~\ref{thm:relation_to_even_periodic}, and Proposition~\ref{prop:restriction2}.
\end{proof}

\noindent
The space $\hpc$ is hence contained in the space
$S^{r}_{p,q}B([0,1]^d)$. As the identification in Theorem~\ref{thm:relation_to_even_periodic} shows, 
functions from $\hpc$ feature additional regularity conditions at the boundary of $[0,1]^d$. This makes these spaces potentially smaller than $S^{r}_{p,q}B([0,1]^d)$. For low regularity, however, they are in fact equal.

\begin{theorem}\label{thm:identification_hpc}
In case $0<p,q\le\infty$ and $\sigma_p<r<\min\{1+\tfrac{1}{p},2\}$
\begin{align*}
            \operatorname{Id}:\;     \hpc \cong S^{r}_{p,q}B([0,1]^d) \,.
\end{align*}
\end{theorem}
\begin{proof}
Assuming that the periodization operator $\mathcal{P}$ from Definition~\ref{def:PeriodizationOP} is a bounded operator from
$S^r_{p,q}B(\R^d)$ to $S^r_{p,q}B(\T^d)$, in which case it is also a bounded operator from
$S^r_{p,q}B([0,1]^d)$ to $S^r_{p,q}B(\T^d)$, we can argue as follows. The image
of $S^r_{p,q}B([0,1]^d)$ under $\mathcal{P}$ is $S^r_{p,q}B(\T^d)_{\rm even}$. Further, by the periodization principle, Theorem~\ref{thm:relation_to_even_periodic}, $S^r_{p,q}B(\T^d)_{\rm even}$ is mapped isomorphically to $\hpc$ by the restriction operator $\mathcal{R}$ from Definition~\ref{def:RestrictionOP}. Since $\operatorname{Id}=\mathcal{R}\circ\mathcal{P}$,
we obtain the result. The proof can be pictured nicely in a commutative diagram:
\begin{center}
\begin{tikzcd}
  S^{r}_{p,q}B([0,1]^d) \arrow{rrrd}[swap]{\operatorname{Id}=\mathcal{R}\circ\mathcal{P}} \arrow{rrr}{\mathcal{P}} & & & S^r_{p,q}B(\T^d)_{\rm even} \arrow{d}{\mathcal{R}}[swap]{\cong} \\
                                          &  & & \hpc \\
\end{tikzcd}
\end{center}
\vspace*{-4ex}
Note that the down-arrow on the right depicts the isomorphism according to the periodization principle. It remains to determine the range of $r$, where $\mathcal{P}:S^r_{p,q}B(\R^d)\to S^r_{p,q}B(\T^d)$ is bounded. This question is answered in Theorem~\ref{thm:periodization}.
\end{proof}

\noindent
The following theorem closes the gap
in the proof of Theorem~\ref{thm:identification_hpc}.

\begin{theorem}\label{thm:periodization} Let $0<p,q\le\infty$ and $\sigma_p<r<\min\{1+\tfrac{1}{p},2\}$. Then
$$
        \mathcal{P}:\; S^r_{p,q}B(\R^d) \to S^r_{p,q}B(\T^d) 
$$
is a bounded linear operator. Further, $\mathcal{P}$ is bounded as operator 
$$
        \mathcal{P}:\; S^{r}_{p,\min\{p,1\}}B(\R^d) \to S^{r}_{p,\infty}B(\T^d) 
$$
for $r=\min\{1+\tfrac{1}{p},2\}$ and either $\tfrac{1}{3}<p<1$ or $1<p\le\infty$.
\end{theorem}

\noindent
A consequence of the second statement in Theorem~\ref{thm:periodization} is an interesting endpoint result.

\begin{theorem}\label{thm:endpoint_result}
Let $r=\min\{1+\tfrac{1}{p},2\}$. In case $\tfrac{1}{3}<p<1$ or $1< p\le\infty$ we have the embedding
\begin{align*}
            \operatorname{Id}:\;   S^{r}_{p,\min\{p,1\}}B([0,1]^d) \hookrightarrow \hpc[r][p][\infty][[0,1]^d]  \,.
\end{align*}
\end{theorem}
\begin{proof}
The proof 
is analogous to the proof of Theorem~\ref{thm:identification_hpc}.
\end{proof}

\noindent
The remainder of this section is devoted to the proof of Theorem~\ref{thm:periodization}. It is
presented in Subsection~\ref{ssec:Proof}.
Before, let us recall some tools in the following two subsections, beginning with the Chui-Wang wavelets of $2$nd order.

\subsection{Chui-Wang wavelets of order $2$}
\label{ssec:CW_wavelets}

In 1992, Charles Chui and Jian-Zhong Wang~\cite{Chui1992}
introduced a class of compactly supported spline wavelets generalizing the well-known Haar wavelet to higher orders $m\in\N$. They constructed bases of such wavelets in $L_2(\R)$, nowadays called Chui-Wang wavelets, for every order $m$. 
These bases consist of $m$th order spline functions, i.e.\ functions that are piecewise assembled from polynomials of degree $m-1$. For $m=1$ they replicate the Haar basis, which is orthonormal. In case $m\ge 2$ the bases are not orthonormal any more, yet, they still represent well-behaved Riesz bases with nice duality properties. They are for instance semi-orthogonal and the dual bases can be calculated explicitly (see e.g.~\cite{Chui1992a, Chui1992} and \cite[Sec.~2.2]{DerUll19}). 

In our proof of Theorem~\ref{thm:periodization} we will make use of multivariate Chui-Wang bases of order $2$. Their relevant features for us are their compact support, two vanishing moments and their piecewise linear structure.
They are constructed from the original univariate bases via tensorization.
Let us therefore recall the original construction in $L_2(\R)$ first.
Following~\cite[Thm.~1]{Chui1992}, the $2$nd order univariate Chui-Wang wavelet $\psi_{2}$ can be written in the form (see formula~(2.3) in~\cite{Chui1992}) 
\begin{equation*}
	\psi_{2}(x)
	= \sum_{l = 0}^{4} q_{l} N_{2}(2x - l)
	\quad\text{with}\quad
	q_{l}
	= \frac{(-1)^{l}}{2} \sum_{i = 0}^{2} \binom{2}{i} N_{4}(l - i + 1) \,,
\end{equation*}
where $N_{2}$ stands for the $2$nd order cardinal $B$-spline $N_{2}:= \chi_{[0,1)} \ast \chi_{[0,1)}$ and $N_{4}:= N_{2} \ast N_{2}$ for the $4$th order cardinal $B$-spline.
For $l \in \N_{0}$ and $k \in \Z$, one then sets
\begin{equation}\label{gen:CWsystem}
	\psi_{-1, k}(x)
	:= N_{2}(x - k)
	\quad \text{and} \quad
	\psi_{l,k}(t)
	:= \psi_{2}(2^{l}x - k) \, .
\end{equation}
The collection $\{ \psi_{l,k} :  l\in \N_{-1}, k \in \Z \}$ is called the univariate Chui-Wang wavelet system of order $2$. 
The generators $\psi_{0,0}=\psi_{2}$ and $\psi_{-1,0}=N_2$ are sometimes referred to as the $2$nd order mother and father Chui-Wang wavelet. They have the
explicit form
\begin{align}\label{eq:explicit_CW}
    \psi_{0,0}(x) = \begin{cases}   
        \tfrac{1}{6}x &,\, 0\le x < \tfrac{1}{2} \,, \\
        -\tfrac{7}{6}x + \tfrac{2}{3} &,\, \tfrac{1}{2} \le x < 1 \,, \\
        \tfrac{8}{3}x - \tfrac{19}{6} &,\, 1 \le x < \tfrac{3}{2} \,, \\
        -\tfrac{8}{3}x + \tfrac{29}{6} &,\,  \tfrac{3}{2} \le x < 2 \,, \\
         \tfrac{7}{6}x - \tfrac{17}{6} &,\, 2 \le x < \tfrac{5}{2} \,, \\
         -\tfrac{1}{6}x + \tfrac{1}{2} &,\, \tfrac{5}{2} \le x < 3 \,, \\
         0 &,\,\text{otherwise}\,.
    \end{cases}
    \quad,\quad
    \psi_{-1,0}(x) = \begin{cases}   
        x &,\, 0\le x < 1 \,, \\
        2 - x  &,\, 1 \le x < 2 \,, \\
         0 &,\,\text{otherwise}\,.
    \end{cases}
\end{align}
Hence, with certain coefficients $c^{(0)}_{\kappa}\in\R$ and  $c^{(-1)}_{\kappa}\in\R$, it holds
\begin{align}\label{represent:CW_wavelets}
   \psi_{0,0}(x) = \sum_{\kappa=0}^{6} c^{(0)}_{\kappa} \big(x - \tfrac{\kappa}{2} \big)_{+}
   \quad\text{and}\quad
   \psi_{-1,0}(x) = \sum_{\kappa=0}^{6}  c^{(-1)}_{\kappa}  \big(x - \tfrac{\kappa}{2} \big)_{+} \,.
\end{align}
Note that, due to the half-integer shifts, we have here
$c^{(-1)}_{\kappa}=0$ for $\kappa\in\{1,3,5,6\}$.

The dual Chui-Wang basis $\{ \tilde{\psi}_{l,k} :  l\in \N_{-1}, k \in \Z \}$ of order $2$ is generated as in~\eqref{gen:CWsystem} from modified generators
\begin{align}\label{def:dualCW}
\tilde{\psi}_{0,0} := \sum_{n\in\Z} a^{(0)}_{n} \psi_{0,n}
\quad\text{and}\quad
\tilde{\psi}_{-1,0} := \sum_{n\in\Z} a^{(-1)}_{n} \psi_{-1,n} \,,
\end{align}
where $\{a^{(0)}_{n}\}_{n\in\Z}$ and $\{a^{(-1)}_{n}\}_{n\in\Z}$ are some specific sequences in $\mathbb{C}$ (see \cite[Sec.~2.2]{DerUll19} for an explicit calculation).
These sequences feature exponential decay, which means there exists $c>1$ so that 
\[
|a^{(\varepsilon)}_{n}| \lesssim c^{-|n|}  \quad\text{for $n\in\Z$,\, $\varepsilon\in\{-1,0\}$}\,.
\]
The relation to the primal Chui-Wang basis is expressed by the biorthogonality relation (see~\cite[Thm.~3]{Chui1992})
\begin{align}\label{rel_biortho_uni}
\langle \psi_{j,k} , \tilde{\psi}_{l,m} \rangle_{\R} = 2^{-(j_+ +l_+)/2} \delta_{j,l} \delta_{k,m} \,.
\end{align}

\noindent
Finally, to arrive at the desired multivariate generalization, we take tensor products,  
i.e.\  for $x=(x_1,\ldots,x_d)\in\R^d$ we let
\begin{align}\label{eqdef:tensor_CW}
    \psi_{\bl,\bk}(x) := \psi_{l_1,k_1}(x_1) \cdots \psi_{l_d,k_d}(x_d) \quad\text{and}\quad 
    \tilde{\psi}_{\bl,\bk}(x) := \tilde{\psi}_{l_1,k_1}(x_1) \cdots  \tilde{\psi}_{l_d,k_d}(x_d) \,.
\end{align}
The primal and dual tensor Chui-Wang wavelet systems on $\R^d$ then look like 
\begin{align*}
   \{ \psi_{\bl,\bk} : (\bl,\bk) \in \Gamma^{d} \} \quad\text{and}\quad \{ \tilde{\psi}_{\bl,\bk} : (\bl,\bk) \in \Gamma^{d} \} \quad\text{with index set } \Gamma^d:=\N_{-1}^d\times \Z^d \,.
\end{align*}
From~\eqref{rel_biortho_uni} we obtain
\begin{align}\label{rel_biortho_multi}
\langle \psi_{\bj,\bk} , \tilde{\psi}_{\bl,\bm} \rangle_{\R^d} = 2^{-(\bj_+ +\bl_+)/2} \delta_{\bj,\bl} \delta_{\bk,\bm} \,.
\end{align}
These systems thus 
constitute a pair of dual Riesz bases and 
every $f\in L_2(\R^d)$ can be expanded in the form
\[
 f= \sum_{(\bj,\bk)\in\Gamma^d} 2^{\bj_+} \langle f, \tilde{\psi}_{\bj,\bk} \rangle_{\R^d}  \psi_{\bj,\bk} = \sum_{(\bj,\bk)\in\Gamma^d} 2^{\bj_+} \langle f, \psi_{\bj,\bk} \rangle_{\R^d}  \tilde{\psi}_{\bj,\bk}
\]
with unconditional convergence in $L_2(\R^d)$. The factor $2^{\bj_+}$ is due to the
biorthogonality relation~\eqref{rel_biortho_multi}.

We further have
\begin{align*}
    \|f \|_{ L_2(\R^d)} \asymp  \| \{ 2^{\bj_+/2} \langle f,\psi_{\bj,\bk} \rangle \}_{\bj,\bk} \|_{\ell_2(\Gamma^d)} \asymp   \|  \{ 2^{\bj_+/2}  \langle f,\tilde{\psi}_{\bj,\bk} \rangle \}_{\bj,\bk}  \|_{\ell_2(\Gamma^d)} \,.
\end{align*}
In Section~\ref{sec:CW_char} we will investigate the generalization of these basis properties to the space $S^{r}_{p,q}B(\R^d)$. The result is formulated in Theorem~\ref{thm:CW_characterization}.

Concluding this subsection, let us note that we can derive from~\eqref{def:dualCW} the representation 
\begin{align}\label{dualCW_representation}
    \tilde{\psi}_{\bl,\bk} = \sum_{\bn\in\Z^d} a^{(\bl)}_{\bn} \psi_{\bl,\bk+\bn} \quad\text{with}\quad  a^{(\bl)}_{\bn} := \prod_{i\in[d]}   a^{(l_i^{-})}_{n_i}   
\end{align}
of the dual wavelets for any $(\bl,\bk)\in\Gamma^d$, where $l_i^{-}=\min\{0,l_i\}$. For the coefficients' decay we get
\begin{align}\label{eq:coeff_decay}
|a^{(\bl)}_{\bn}| \lesssim c^{-|\bn|}  \quad\text{for $\bn\in\Z^d$,\, $\bl\in\N_{-1}^d$}\,.
\end{align}


\subsection{Further preparation}

The Lipschitz constant $L(f)$ of a univariate function $f:\R\to\C$ is defined by
\begin{align*}
 L(f) := \sup\limits_{\substack{x,y\in\R \\ x\neq y}} \frac{|f(x)-f(y)|}{|x-y|} \,. 
\end{align*}
Further, the
univariate forward difference operator on $\R$, denoted by $\Delta_h$ for $h\in\R$, is given by
\begin{align*}
\Delta_h(f,x) := f(x+h) - f(x) \;,\quad x\in\R \,.
\end{align*}
Its multivariate version $\Delta_{h,i}$ acting on the $i$th coordinate of a function $f:\R^d\to\C$, for $i\in[d]$, is set to be
\[
\Delta_{h,i}(f,x) = f(x_1,\ldots,x_i+h,\ldots x_d) -  f(x_1,\ldots, x_d) \;,\quad x=(x_1,\ldots,x_d)\in\R^d \,.
\]
Recursively, we then define the $m$th order difference operator $\Delta^m_{h,i}$
for $m\in\N$ by setting $\Delta^{m}_{h,i}:=\Delta_{h,i}\circ\Delta^{m-1}_{h,i}$
and $\Delta^{0}_{h,i}:=\operatorname{Id}$. This operator acts on the $i$th coordinate as
\[
\Delta^{m}_{h,i}(f,x) = \sum\limits_{j=0}^{m} (-1)^j \binom{m}{j} f(x_1,\ldots,x_{i-1},x_i+(m-j)h,x_{i+1},\ldots, x_d)  \,.
\]
Utilizing this operator, we finally introduce the multivariate difference operator $\Delta^{m,e}_{\bh}$ of order $m$ for a subset $e\subset[d]$ of indices and $\bh=(h_1,\ldots,h_d)\in\R^d$. If $e=\emptyset$ we set $\Delta^{m,\emptyset}_{\bh}:= \operatorname{Id}$ and otherwise
\begin{align}\label{def:multi_diff_op}
    \Delta^{m,e}_{\bh} := \prod_{i\in e} \Delta^{m}_{h_i,i} \,.
\end{align}
This operator plays a crucial role in our proof of Theorem~\ref{thm:periodization} since we rely on a characterization of $S^r_{p,q}B(\R^d)$ by rectangular means of differences.
A few other auxiliary results needed in the proof are listed below.
\begin{itemize}[leftmargin=*]
\item
For a univariate function $f:\R\to\C$ and for $\kappa>0$
\begin{align}\label{rel:Lipschitz1}
    |[\Delta_h^2 f(\kappa \cdot)](x)|
    \le 2L(f)\kappa |h|  \quad\text{and}\quad |[\Delta_h^2 f(\kappa \cdot)](x)| \le 4\|f\|_{L_{\infty}(\R)} \,.
\end{align}
\item It holds $L(f\circ g)\le L(f)L(g)$ for any two univariate functions $f:\R\to\C$ and $g:\R\to\R$. In particular
\begin{align}\label{rel:Lipschitz2}   
L(f\circ\rho_1)\le L(f)
\end{align}
if $\rho_1$ is the function from~\eqref{eqdef:aux_rho_1} since $L(\rho_1)=1$.
\end{itemize}

\noindent
After this preparation, we are now ready for the proof of Theorem~\ref{thm:periodization}.

\subsection{Proof of Theorem~\ref{thm:periodization}}
\label{ssec:Proof}

Let $\mathcal{P}$ be the periodization operator from Definition~\ref{def:PeriodizationOP} and $f\in S^r_{p,q}B(\R^d)$ with parameters $0<p,q\le\infty$ and $\sigma_p<r\le2$, where $\sigma_p=\max\{0,\tfrac{1}{p}-1\}$. 
In this parameter range the (quasi-)norm
\begin{align}\label{fff1}
  \|\mathcal{P}(f) \|_{S^r_{p,q}B(\T^d)} =   \|f\circ\rho|_{[-1,1]^d}  \|_{S^r_{p,q}B(\T^d)}
\end{align}
can be characterized
via rectangular means of differences of $3$rd order. Here we follow~\cite[Sec.~3.3]{NUUChangeVariable2015} (cf.\ also~\cite[Sec.~4.5]{TDiff06}). For a function $f:\R^d\to\C$, a given multiscale $\bt\in\R^d_+$, and a subset of coordinate directions $e\subset[d]$, the order $m$ means of differences are defined locally at position $x\in\R^d$ by
\begin{align*}
  \mathcal{R}^{e}_{m}(f,\bt,x) :=  \int_{[-1,1]^d}  \big| \Delta^{m,e}_{(h_1t_1,\ldots,h_dt_d)}(f,x)\big|  \,\dint\bh 
  =  \int_{[-1,1]^d}  \big| \Big( \prod_{i\in e} \Delta^{m}_{h_it_i,i} \Big)(f,x)\big|  \,\dint\bh \,,
\end{align*}
where $\Delta^{m,e}_{\bh}$ is the operator~\eqref{def:multi_diff_op}.
According to~\cite[Thm.~3.6]{NUUChangeVariable2015} (cf.\ also~\cite[Thm.~4.5.1]{TDiff06} for a continuous version), when we let $e(\bj):=\{ i : j_i\neq 0 \}$ for $\bj\in\N^d_{0}$, 
we have the characterization
\begin{align*}
  \|\mathcal{P}(f) \|_{S^r_{p,q}B(\T^d)}\asymp 
    \Big( \sum_{\bj\in\N^{d}_{0}} 2^{|\bj|_1 rq} \Big\|  \mathcal{R}^{e(\bj)}_{3}(f\circ\rho,2^{-\bj},\cdot) \Big\|_{L_p([-1,1]^d)}^q \Big)^{1/q}  
\end{align*}
of the (quasi-)norm~\eqref{fff1}. Note that $f\circ\rho$ is a function on $\R^d$ that is periodic with respect to the domain $[-1,1]^d$. 
Noting $\mathcal{R}^{e}_{3}(f,\bt,x) \lesssim \mathcal{R}^{e}_{2}(f,\bt,x)$, we deduce
the following estimate
\begin{align*}
  \|\mathcal{P}(f) \|_{S^r_{p,q}B(\T^d)}\lesssim
    \Big( \sum_{\bj\in\N^{d}_{0}} 2^{|\bj|_1 rq} \Big\|  \mathcal{R}^{e(\bj)}_{2}(f\circ\rho,2^{-\bj},\cdot) \Big\|_{L_p([-1,1]^d)}^q \Big)^{1/q} \,. 
\end{align*}
This is our starting point for proving Theorem~\ref{thm:periodization}.
For the first statement, the estimate
\begin{align}\label{end_point}
 \|\mathcal{P}(f)\|_{S^r_{p,q}B(\T^d)} \lesssim \|f\|_{S^r_{p,q}B(\R^d)} 
\end{align}
is what we aim for. Towards this goal, we decompose $f$ into different scales 
\begin{align}\label{decomp_CW}
 f=\sum_{\bl\in\N_{-1}^d} f_{\bl}   
\end{align} 
of level-wise $2$nd order Chui-Wang expansions (see Theorem~\ref{thm:CW_characterization} for convergence properties)
\begin{align*}
  f_{\bl} := \sum_{\bk\in\Z^d} 2^{\bl_+} \langle f,\tilde{\psi}_{\bl,\bk} \rangle_{\R^d} \psi_{\bl,\bk} \,.
\end{align*}
Recall the tensor structure $\psi_{\bl,\bk}(x)=\psi_{l_1,k_1}(x_1) \cdots  \psi_{l_d,k_d}(x_d)$ and $\tilde{\psi}_{\bl,\bk}(x)=\tilde{\psi}_{l_1,k_1}(x_1) \cdots  \tilde{\psi}_{l_d,k_d}(x_d)$ of the Chui-Wang wavelets in~\eqref{eqdef:tensor_CW}. Since we want to sum over all $\bl\in\Z^d$ in~\eqref{decomp_CW}, let us put $\psi_{l_i,k_i}(x_i):=\tilde{\psi}_{l_i,k_i}(x_i):=0$
if $l_i<-1$. Then $\psi_{\bl,\bk}=\tilde{\psi}_{\bl,\bk}\equiv 0$ and $f_{\bl}\equiv0$ in case $\min_{i\in[d]}\{l_i\}<-1$.

Using decomposition~\eqref{decomp_CW}, we estimate with $u:=\min\{p,q,1\}$ and $\tilde{p}:=p/u$, $\tilde{q}:=q/u$, $\tilde{r}:=ru$,
\begin{align}
    \Big[ \eqref{fff1} \Big]^u 
    &\lesssim
    \Big( \sum_{\bj\in\N^{d}_{0}}  2^{|\bj|_1 rq} \Big\| \sum_{\bl\in\Z^d} \mathcal{R}^{e(\bj)}_{2}(f_{\bj+\bl}\circ\rho,2^{-\bj},\cdot) \Big\|^q_{L_p([-1,1]^d)}  \Big)^{1/\tilde{q}}  \notag\\
    &\le 
    \Big( \sum_{\bj\in\N^{d}_{0}}  2^{|\bj|_1 \tilde{r}\tilde{q}} \Big\|  \sum_{\bl\in\Z^d} |\mathcal{R}^{e(\bj)}_{2}(f_{\bj+\bl}\circ\rho,2^{-\bj},\cdot)|^u  \Big\|_{L_{\tilde{p}}([-1,1]^d)}^{\tilde{q}}  \Big)^{1/\tilde{q}}  \,. \label{norm-est0}
    \end{align}
    Since $\tilde{p},\tilde{q}\ge 1$, we may continue with Minkowski's inequality and arrive at
    \begin{align}
      \Big[ \eqref{fff1} \Big]^u 
    &\lesssim   \sum_{\bl\in\Z^d} \Big( \sum_{\bj\in\N^{d}_{0}}  2^{|\bj|_1 \tilde{r}\tilde{q}}   
     \Big\| |\mathcal{R}^{e(\bj)}_{2}(f_{\bj+\bl}\circ\rho,2^{-\bj},\cdot)|^u \Big\|_{L_{\tilde{p}}([-1,1]^d)}^{\tilde{q}} \Big)^{1/\tilde{q}}  \notag\\
      &= \sum_{\bl\in\Z^d} \Big(  \sum_{\bj\in\N^{d}_{0}}    2^{|\bj|_1 rq}   
     \Big\| \mathcal{R}^{e(\bj)}_{2}(f_{\bj+\bl}\circ\rho,2^{-\bj},\cdot) \Big\|_{L_{p}([-1,1]^d)}^{q} \Big)^{1/\tilde{q}}  \,.  \label{norm-est1}
\end{align}

\noindent
Putting $\lambda_{\bl,\bk}:=2^{\bl_+} \langle f,\tilde{\psi}_{\bl,\bk} \rangle_{\R^d}$ for $\bl\in\Z^d$,
where $\lambda_{\bl,\bk}=0$ if $\min_{i\in[d]}\{l_i\}<-1$,
we can further estimate
\begin{align}\label{est5}
    \mathcal{R}^{e(\bj)}_{2}(f_{\bj+\bl}\circ\rho,2^{-\bj},\cdot)
    \lesssim \sum_{\bk\in\Z^d} |\lambda_{\bj+\bl,\bk}| \mathcal{R}^{e(\bj)}_{2}(\psi_{\bj+\bl,\bk}\circ\rho,2^{-\bj},\cdot) 
\end{align}
and, due to the tensor structure of $\psi_{\bl,\bk}$, it holds
\begin{align*}
  \mathcal{R}^{e(\bj)}_{2}&(\psi_{\bj+\bl,\bk}\circ\rho,2^{-\bj},x) =  \int_{[-1,1]^d}  \big| \Delta^{2,e(\bj)}_{(h_12^{-j_1},\ldots,h_d2^{-j_d})}(\psi_{\bj+\bl,\bk}\circ\rho,x)\big|  \,\dint\bh \\
  &= \prod_{i\in e(\bj)} \LaTeXunderbrace{\int\limits_{-1}^{1}  \big| \Delta^{2}_{h_i2^{-j_i}}(\psi_{j_i+l_i,k_i}\circ\rho_1,x_i)\big|  \,\dint h_i}_{F(x_i)}  \cdot
  \prod_{i\in e_0(\bj)} \LaTeXunderbrace{\int\limits_{-1}^{1}  \big| \psi_{l_i,k_i}\circ\rho_1(x_i)\big|  \,\dint h_i}_{G(x_i)}  \,,
\end{align*}
where $e_0(\bj):=[d]\backslash e(\bj)$ and $\rho_1$ is the function~\eqref{eqdef:aux_rho_1}.
The functions $F,\,G:\R\to[0,\infty)$ 
are univariate and $2$-periodic.
Let us put 
\begin{align*}
\text{I}_{j_i,k_i,l_i}(x_i):=\begin{cases} F(x_i) \quad&,\quad j_i>0 \,, \\
G(x_i)&,\quad j_i=0 \,. \end{cases}
\end{align*}
We can then write conveniently 
\begin{align*}
  \mathcal{R}^{e(\bj)}_{2}(\psi_{\bj+\bl,\bk}\circ\rho,2^{-\bj},x) = \prod_{i\in[d]} \text{I}_{j_i,k_i,l_i}(x_i)  \,.
\end{align*} 
We next split 
\begin{align*}
 \text{I}_{j_i,k_i,l_i}(\cdot) = \text{I}^{[0]}_{j_i,k_i,l_i}(\cdot) + \text{I}^{[1]}_{j_i,k_i,l_i}(\cdot)   
\end{align*}
with the summands
\begin{align}\label{eqdef:aux_summands}
\text{I}^{[0]}_{j_i,k_i,l_i} :=  \text{I}_{j_i,k_i,l_i} \cdot \chi_{S_{j_i,k_i,l_i}}  \quad\text{and}\quad  \text{I}^{[1]}_{j_i,k_i,l_i} :=  \text{I}_{j_i,k_i,l_i} - \text{I}^{[0]}_{j_i,k_i,l_i}  \,,
\end{align}
where $\chi_{S_{j_i,k_i,l_i}}$ is the characteristic function of
\begin{align*}
S_{j_i,k_i,l_i}:= \supp(\psi_{j_i+l_i,k_i}\circ\rho_1) \,.
\end{align*}
Then we have for $x=(x_1,\ldots,x_d)\in\R^d$
\begin{align}\label{est6}
  \mathcal{R}^{e(\bj)}_{2}(\psi_{\bj+\bl,\bk}\circ\rho,2^{-\bj},x) = \prod_{i\in[d]} \Big( \text{I}^{[0]}_{j_i,k_i,l_i}(x_i) +  \text{I}^{[1]}_{j_i,k_i,l_i}(x_i) \Big) = \sum\limits_{\beps\in\{0,1\}^d}   \text{I}^{[\beps]}_{\bj,\bk,\bl}(x) 
\end{align}
with functions
\begin{align*}
    \text{I}^{[\beps]}_{\bj,\bk,\bl}(x) :=  \prod_{i\in[d]} \text{I}^{[\varepsilon_i]}_{j_i,k_i,l_i}(x_i) \quad\text{for}\quad \beps=(\varepsilon_1,\ldots,\varepsilon_d)\in \{0,1\}^d \,.
\end{align*}

\noindent
From \eqref{est5} and \eqref{est6} we get for fixed $\bj\in\N_{0}^{d}$ and $\bl\in\Z^d$:
\begin{align*}
\Big\|  \mathcal{R}^{e(\bj)}_{2}(f_{\bj+\bl}\circ\rho,2^{-\bj},\cdot)
     \Big\|_{L_{p}([-1,1]^d)}  
\lesssim \sum\limits_{\beps\in\{0,1\}^d}   \Big\|  \sum_{\bk\in\Z^d} |\lambda_{\bj+\bl,\bk}| \text{I}^{[\beps]}_{\bj,\bk,\bl}(\cdot)\Big\|_{L_{p}([-1,1]^d)} \,.
\end{align*}

\noindent
Let us next analyze some properties of the functions $\text{I}^{[\varepsilon_i]}_{j_i,k_i,l_i}$, which are piecewise linear and $2$-periodic. The periodicity allows to restrict the investigation to the period interval $[-1,1]$.
\begin{itemize}[leftmargin=*]
\item
We first prove the support properties
\begin{align}
\big|\supp(\text{I}^{[0]}_{j_i,k_i,l_i})\cap [-1,1]\big|
 &\lesssim 2^{-j_i-l_i^+} \,, \label{111} \\
\big|\supp(\text{I}^{[1]}_{j_i,k_i,l_i})\cap [-1,1]\big| &\lesssim
2^{-j_i} \,. \label{222}
\end{align}

\begin{proof}[Proof of~\eqref{111} and~\eqref{222}]\renewcommand{\qed}{$\square$}
To verify~\eqref{111}, we use that, by definition, the support of $\text{I}^{[0]}_{j_i,k_i,l_i}$ satisfies
\begin{align*}
\big|\supp(\text{I}^{[0]}_{j_i,k_i,l_i})\cap [-1,1]\big| \le \big| S_{j_i,k_i,l_i}\cap [-1,1] \big| \lesssim
\min\{ 2^{-j_i-l_i} , 1\} = 2^{-(j_i+l_i)_+} \,.
\end{align*}
In case $l_i\ge0$ or if $j_i=0$ this directly implies~\eqref{111}, since then $(j_i+l_i)_+=j_i+l_i^+$.
It remains the case, where both $l_i<0$ and $j_i>0$.
Assuming $j_i>0$, the only contributions to the support of $\text{I}_{j_i,k_i,l_i}$ come from the kink points $k^{(m)}_{j_i+l_i,k_i}$ of $\psi_{j_i+l_i,k_i}\circ\rho_1$.
If $\dist(x_i,k^{(m)}_{j_i+l_i,k_i})>2\cdot2^{-j_i}$ for all kink points $k^{(m)}_{j_i+l_i,k_i}$ 
then clearly 
 \begin{align}\label{kink_points}
 \text{I}_{j_i,k_i,l_i}(x_i)  = \int\limits_{-1}^{1}  \big| \Delta^{2}_{h_i2^{-j_i}}(\psi_{j_i+l_i,k_i}\circ\rho_1,x_i)\big|  \,\dint h_i = 0 \,.
\end{align}
Since $\psi_{j_i+l_i,k_i}$ has either $3$ or $7$ kink points (see~\eqref{eq:explicit_CW}), there are further at most $28$ kink points of $\psi_{j_i+l_i,k_i}\circ\rho_1$ in $[-2,2]$, which is the maximal range that might be relevant for the value of $\text{I}_{j_i,k_i,l_i}(x_i)$ if $x_i\in[-1,1]$ and $j_i>0$. Hence, we obtain for $j_i>0$
\begin{align*}
\big|\supp(\text{I}_{j_i,k_i,l_i})\cap [-1,1]\big| \lesssim
2^{-j_i} \,.
\end{align*}
Since $j_i=j_i+l_i^+$ if $l_i<0$, this settles the proof of~\eqref{111}.

The last estimate also
shows~\eqref{222} for $j_i>0$. For the remaining case $j_i=0$ of~\eqref{222} simply note that 
\begin{gather}\label{trivial}
  \text{I}^{[1]}_{0,k_i,l_i} \equiv 0 \,.  \;\text{\qedhere} 
\end{gather}
\end{proof}
\item 
Next, we prove the upper bounds
\begin{align}
\big\|\text{I}^{[0]}_{j_i,k_i,l_i}\big\|_{L_{\infty}(\R)}
 &\lesssim 2^{l_i^-} \,, \label{333} \\
\big\| \text{I}^{[1]}_{j_i,k_i,l_i} \big\|_{L_{\infty}(\R)} &\lesssim
2^{-|l_i|} \,. \label{444}
\end{align}

\begin{proof}[Proof of~\eqref{333} and~\eqref{444}]\renewcommand{\qed}{$\square$}
We begin with the case $j_i=0$ of~\eqref{333} and estimate
\begin{align*}
  \text{I}^{[0]}_{0,k_i,l_i}(x_i) \le \int\limits_{-1}^{1}  \big| \psi_{l_i,k_i}\circ\rho_1(x_i)\big|  \,\dint h_i \lesssim |\psi_{l_i,k_i}\circ\rho_1(x_i)|  \lesssim \min\{2^{l_i}, 1\} = 2^{l^-_{i}}\,,
\end{align*}
where we use $\psi_{l_i,k_i}=0$ for $l_i<-1$ and $|\psi_{l_i,k_i}\circ\rho_1(x_i)| \lesssim 1$. For proving~\eqref{333} in case $j_i>0$, we record on the one hand
\begin{align}\label{aux1}
  \text{I}_{j_i,k_i,l_i}(x_i) =  \int\limits_{-1}^{1}  \big| \Delta^{2}_{h_i2^{-j_i}}(\psi_{j_i+l_i,k_i}\circ\rho_1,x_i)\big|  \,\dint h_i \lesssim 1 \,.
\end{align}
Hereby, the previous estimate~\eqref{aux1} holds true since, according to~\eqref{rel:Lipschitz1}, 
\begin{align}\label{roman3}
|\Delta_{h_i2^{-j_i}}^2(\psi_{j_i+l_i,k_i}\circ\rho_1,x_i)| \le 4  \|\psi_{j_i+l_i,k_i}\circ\rho_1\|_{L_{\infty}(\R^d)} \le 4  \|\psi_{0,0}\|_{L_{\infty}(\R^d)} \lesssim 1 \,.
\end{align}

On the other hand, we get from the Lipschitz properties of $\psi_{j_i+l_i,k_i}$ and $\rho_1$ (see~\eqref{rel:Lipschitz1} and~\eqref{rel:Lipschitz2}) the estimate
\begin{align*}
    |\Delta_{h_i2^{-j_i}}^2(\psi_{j_i+l_i,k_i}\circ\rho_1,x_i)| \lesssim L(\psi_{j_i+l_i,k_i}\circ\rho_1) |h_i| 2^{-j_i} \le L(\psi_{j_i+l_i,k_i}) |h_i| 2^{-j_i} =   L(\psi_{0,0})  |h_i|2^{l_i} \lesssim
       |h_i|2^{l_i}   
\end{align*}
and thus, if $j_i>0$,
\begin{align}\label{roman4}
    \text{I}_{j_i,k_i,l_i}(x_i) = \int\limits_{-1}^{1}  \big| \Delta^{2}_{h_i2^{-j_i}}(\psi_{j_i+l_i,k_i}\circ\rho_1,x_i)\big|  \,\dint h_i \lesssim  2^{l_i} \,.
\end{align}
Together with~\eqref{aux1}, the last estimate~\eqref{roman4} in particular implies
\begin{align*}
     \text{I}^{[0]}_{j_i,k_i,l_i}(x_i)  \lesssim \min\{2^{l_i}, 1\} = 2^{l^-_{i}} 
     \quad\text{for }j_i>0 \,,
\end{align*}
finishing the proof of~\eqref{333}.

At last, we turn to the proof of~\eqref{444} and recall~\eqref{trivial}
to see that~\eqref{444} is trivially fulfilled in case $j_i=0$.
In order to prove~\eqref{444} in case $j_i>0$ we use~\eqref{roman3} and the fact that, if $x_i\notin S_{j_i,k_i,l_i} =\supp\big(\psi_{j_i+l_i,k_i}\circ\rho_1\big)$,
\begin{align*}
 |\supp(h_i\mapsto\Delta^{2}_{h_i2^{-j_i}}(\psi_{j_i+l_i,k_i}\circ\rho_1,x_i))\cap[-1,1]| &\lesssim
     2^{-l_i} 
\end{align*}
to deduce
\begin{align*}
  \text{I}_{j_i,k_i,l_i}(x_i) =  \int\limits_{-1}^{1}  \big| \Delta^{2}_{h_i2^{-j_i}}(\psi_{j_i+l_i,k_i}\circ\rho_1,x_i)\big|  \,\dint h_i \lesssim 2^{-l_i} 
\end{align*}
in case $x_i\notin S_{j_i,k_i,l_i}$. Combined with estimate~\eqref{roman4}, estimate~\eqref{444} is obtained. \qedhere 
\end{proof}
\end{itemize}

\noindent
Altogether, our estimates~\eqref{111}, \eqref{222}, \eqref{333}, \eqref{444} show that the multivariate functions $\text{I}^{[\beps]}_{\bj,\bk,\bl}$ satisfy
\begin{align}\label{final_estimate_I}
\big| \supp(\text{I}^{[\beps]}_{\bj,\bk,\bl}) \cap [-1,1]^d \big| \lesssim 2^{-\bj-\bl_{+}} A(\beps,\bl)  \qquad\text{and}\qquad
\big\| \text{I}^{[\beps]}_{\bj,\bk,\bl} \big\|_{L_{\infty}(\R^d)} \lesssim 2^{\bl_{-}} A(\beps,\bl)^{-1}  
\end{align}
if we define 
\begin{align*}
A(\beps,\bl):=\prod_{i\in e(\beps)} 2^{l_i^+} \quad\text{with}\quad e(\beps)=\{ i : \varepsilon_i\neq 0 \} \,.
\end{align*}

\noindent
Let us next count the number of local overlaps of the sum $\sum_{\bk\in\Z^d} |\lambda_{\bj+\bl,\bk}| \text{I}^{[\beps]}_{\bj,\bk,\bl}(x)$, i.e.\ its number of non-zero summands when $x\in\R^d$ is fixed. Due to the $2$-periodicity of the functions $\text{I}^{[\beps]}_{\bj,\bk,\bl}$ it hereby suffices to consider $x\in[-1,1]^d$. Further, we only aim for an upper bound. It is therefore enough to consider the values 
\[
O^{[\beps]}_{\bj,\bl}(x) := \big|\big\{ \bk\in\Z^d ~:~ \text{I}^{[\beps]}_{\bj,\bk,\bl}(x) \neq 0 \big\}\big| \,.
\]
The supports of the functions $\text{I}^{[\beps]}_{\bj,\bk,\bl}$ will subsequently be denoted by
\begin{align*}
T^{[\beps]}_{\bj,\bk,\bl} := \supp\big( \text{I}^{[\beps]}_{\bj,\bk,\bl} \big) = \supp\big( \text{I}^{[\varepsilon_1]}_{j_1,k_1,l_1} \big) \times \ldots \times \supp\big( \text{I}^{[\varepsilon_d]}_{j_d,k_d,l_d} \big) \,.
\end{align*}

\noindent
Moving in $i$-direction, $k_{i}\rightarrow k_{i}+1$, the wavelets
$\psi_{j_i+l_i,k_i}$, which are nontrivial only for $j_i+l_i\ge -1$, are shifted by a length of
$2^{-(j_i+l_i)_+}$. Further, only if their support intersects with $[0,1]$
the functions $\text{I}^{[\varepsilon_i]}_{j_i,k_i,l_i}$ do not vanish. Hence, if we let $T^{[\varepsilon_i]}_{j_i,k_i,l_i} := \supp\big( \text{I}^{[\varepsilon_i]}_{j_i,k_i,l_i} \big) $
for $i\in[d]$, we have $T^{[\varepsilon_i]}_{j_i,k_i,l_i} \neq \emptyset $ only for $k_i$ in the range 
\begin{align}\label{range_ki}
k_i \in \{ -3, -1, 0, \ldots, 2^{(j_i+l_i)_+} \} \,.
\end{align}
For this observation, it is also necessary to note that
$\supp\big(\psi_{0,0}\big)=[0,3]$ and $\supp\big(\psi_{-1,0}\big)=[0,2]$ for the generating wavelets $\psi_{0,0}$ and $\psi_{-1,0}$ (see~\eqref{eq:explicit_CW}).

The range~\eqref{range_ki} of relevant $k_i$ gives an initial bound on the maximal number of overlaps. However, locally, we can bound more sharply due to the shrinking support of the wavelets at rising scales. 
Let us first consider the case $\varepsilon_i=0$, 
when $\text{T}^{[0]}_{j_i,k_i,l_i}\subset S_{j_i,k_i,l_i}$ (see~\eqref{eqdef:aux_summands}) and $S_{j_i,k_i,l_i}\cap [-1,1]$ is contained in at most $2$ intervals of length  
$3\cdot 2^{-(j_i+l_i)_+}$ since $|\supp(\psi_{j_i+l_i,k_i})|\le 3\cdot 2^{-(j_i+l_i)_+}$. Indeed, the fact that here the size of the support
$\mathcal{O}(2^{-(j_i+l_i)_+})$ matches the size of the shifts of the wavelets results in just $\mathcal{O}(1)$ local overlaps of the functions $\text{I}^{[0]}_{j_i,k_i,l_i}$, $k_i\in\Z$.
In case $\varepsilon_i=1$, we have $T^{[1]}_{0,k_i,l_i}=\emptyset$
according to~\eqref{trivial} and hence no overlaps if $j_i=0$. 
If $j_i>0$, we know from \eqref{kink_points} that
the set $T^{[1]}_{j_i,k_i,l_i}\cap[-1,1]$ is contained in at most $28$ intervals of length $2^{1-j_i}$ 
around the kink points $k^{(m)}_{j_i+l_i,k_i}$ of $\psi_{j_i+l_i,k_i}\circ\rho_1$ in $[-2,2]$. 
As a result, relating the length of these intervals to the size $2^{-(j_i+l_i)_+}$ of the shifts, we obtain
in this case the upper bound $\mathcal{O}(2^{l_i^+})$ for the number of overlaps of the functions $\text{I}^{[1]}_{j_i,k_i,l_i}$, $k_i\in\Z$.

Altogether, our analysis can be summarized in the estimate
\begin{align}\label{final_estimate_II}
O^{[\beps]}_{\bj,\bl}(x) 
\lesssim A(\beps,\bl) \,, 
\end{align}
where the implicit constant is independent of all parameters $\bj\in\N_{0}^d$,
$\bl\in\Z^d$, $\beps\in\{0,1\}^d$, and $x\in\R^d$.
Building upon the estimate~\eqref{final_estimate_II}, we obtain 
\begin{align*}
    \Big( \sum_{\bk\in\Z^d} |\lambda_{\bj+\bl,\bk}| \text{I}^{[\beps]}_{\bj,\bk,\bl}(x) \Big)^p
    \lesssim A(\beps,\bl)^{(p-1)_+}  \sum_{\bk\in\Z^d}   |\lambda_{\bj+\bl,\bk}|^p   \Big(\text{I}^{[\beps]}_{\bj,\bk,\bl}(x) \Big)^p \,,
\end{align*}
which, in view of~\eqref{final_estimate_I}, yields
\begin{align*}
 \Big\|  \sum_{\bk\in\Z^d} |\lambda_{\bj+\bl,\bk}| \text{I}^{[\beps]}_{\bj,\bk,\bl}\Big\|_{L_{p}([-1,1]^d)} 
 &\lesssim 	A(\beps,\bl)^{(1-1/p)_+}  2^{l_{-}} A(\beps,\bl)^{-1} \Big( \sum_{\bk\in\Z^d} |\lambda_{\bj+\bl,\bk}|^p
\|\chi_{T^{[\beps]}_{\bj,\bk,\bl}}\|_{L_1([-1,1]^d)}\Big)^{1/p} \\
 &\lesssim 	A(\beps,\bl)^{(1-1/p)_+}  2^{l_{-}}  2^{(-\bj-\bl_+)/p} A(\beps,\bl)^{1/p-1} \Big( \sum_{\bk\in\Z^d} |\lambda_{\bj+\bl,\bk}|^p \Big)^{1/p} \,.
\end{align*}
Now note that $\sigma_p=(\tfrac{1}{p}-1)_+=(1-\tfrac{1}{p})_+ + \tfrac{1}{p}-1$ and put
\begin{align*}
	B_p(\bl):=\prod_{i\in[d]} B_p(l_i) \quad\text{with}\quad
	B_p(l_i):= \begin{cases} 2^{l_i} \quad &,\, l_i<0 \,, \\ 2^{-l_i/p} &,\, l_i\ge 0 \,. \end{cases}
\end{align*}
Then we can write
\begin{align*}
\Big\|  \sum_{\bk\in\Z^d} |\lambda_{\bj+\bl,\bk}| \text{I}^{[\beps]}_{\bj,\bk,\bl}(x)\Big\|_{L_{p}([-1,1]^d)}
\lesssim 	A(\beps,\bl)^{\sigma_p} B_p(\bl) 2^{-\bj/p}  \Big( \sum_{\bk\in\Z^d} |\lambda_{\bj+\bl,\bk}|^p \Big)^{1/p} \,.
\end{align*} 
Since $1\le A(\beps,\bl) \le 2^{\bl_+}$,
we finally arrive at 
\begin{align*}
\Big\|   \sum_{\bk\in\Z^d} |\lambda_{\bj+\bl,\bk}| \text{I}^{[\beps]}_{\bj,\bk,\bl}
     \Big\|_{L_{p}([-1,1]^d)}  
 \lesssim 	2^{\sigma_p\bl_+} B_p(\bl) 2^{-\bj/p}  \Big( \sum_{\bk\in\Z^d} |\lambda_{\bj+\bl,\bk}|^p \Big)^{1/p} \,,
\end{align*}
where the right-hand side is independent of $\beps$. Altogether, we thus obtain
\begin{align}\label{est:central3}
\Big\|  \mathcal{R}^{e(\bj)}_{2}(f_{\bj+\bl}\circ\rho,2^{-\bj},\cdot)
     \Big\|_{L_{p}([-1,1]^d)}  
 \lesssim 	2^{\sigma_p\bl_+} B_p(\bl) 2^{-\bj/p}  \Big( \sum_{\bk\in\Z^d} |\lambda_{\bj+\bl,\bk}|^p \Big)^{1/p} \,.
\end{align}

\noindent
We plug this estimate into~\eqref{norm-est1}. This yields
\begin{align*}
    \Big[ \eqref{fff1} \Big]^u &\lesssim  \sum_{\bl\in\Z^d} \Big( \sum_{\bj\in\N^{d}_{0}} \Big( 2^{\sigma_p\bl_+} B_p(\bl) 2^{-\bl(r-1/p)}  2^{|\bj+\bl|_1 (r-1/p)}      \Big( \sum_{\bk\in\Z^d} |\lambda_{\bj+\bl,\bk}|^p \Big)^{1/p} \Big)^{q} \Big)^{u/q}    \\
    &\lesssim   \sum_{\bl\in\Z^d} \Big( B_p(\bl) 2^{\sigma_p\bl_+} 2^{-\bl (r-1/p)} \Big)^u \cdot  \Big( \sum_{\bj\in\N^{d}_{-1}}  2^{|\bj|_1 (r-1/p)q}   \Big( \sum_{\bk\in\Z^d} |\lambda_{\bj,\bk} \big|^{p}   \Big)^{q/p} \Big)^{u/q}  \\
 &\lesssim \Big( \sum\limits_{\bj\in
\N_{-1}^d}2^{|\bj|_1(r-1/p)q}\Big(\sum\limits_{\bk\in\Z^d}|\lambda_{\bj,\bk}|^p\Big)^{q/p} \Big)^{u/q} \,,  \end{align*}
since for $\sigma_p<r<1+\tfrac{1}{p}$ we get  
\begin{align*}
     \sum_{\bl\in\Z^d} \Big( B_p(\bl) 2^{\sigma_p\bl_+} 2^{-\bl(r-1/p)} \Big)^u = \sum_{\bl\in\Z^d} \Big( \prod_{i\in[d]} 2^{l^{-}_{i}(1+1/p-r)} 2^{-l^{+}_{i}(r-\sigma_p)} \Big)^u   < \infty \,.
\end{align*}
Hence, we have $\|\mathcal{P}(f)\|_{S^r_{p,q}B(\T^d)} \lesssim \| \{\lambda_{\bl,\bk}\}_{\bl,\bk} \|_{s^r_{p,q}b(\Gamma^d)}$. Invoking Proposition~\ref{prop:analysis_CW} from Section~\ref{sec:CW_char} then shows~\eqref{end_point} in case $r<2$. The proof of the first statement
of Theorem~\ref{thm:periodization} is thus finished.

To prove the second statement, we distinguish between the cases $\tfrac{1}{3}<p<1$
and $1<p\le\infty$. Assuming $\tfrac{1}{3}<p<1$, 
note that in this range $r=2$ and $\sigma_p<2<1+\tfrac{1}{p}$.
Hence, we obtain as above
$\|\mathcal{P}(f)\|_{S^r_{p,\infty}B(\T^d)} \lesssim \| \{\lambda_{\bl,\bk}\}_{\bl,\bk} \|_{s^r_{p,\infty}b(\Gamma^d)}$. Lemma~\ref{lem:analysis_CW} then yields the result in this case.
Assuming $1<p\le\infty$, we need to show 
\begin{align*}
 \|\mathcal{P}(f)\|_{S^{r}_{p,\infty}B(\T^d)} \lesssim \|f\|_{S^{r}_{p,1}B(\R^d)}  \quad\text{with $r=1+\tfrac{1}{p}$}\,.
\end{align*}
Here, we argue as before until estimate~\eqref{norm-est0}. From there, we get
\begin{align*}
 \|\mathcal{P}(f)\|_{S^{r}_{p,\infty}B(\T^d)} \lesssim
     \sup_{\bj\in\N^{d}_{0}} \sum_{\bl\in\Z^d}  2^{|\bj|_1r} \Big\|  \mathcal{R}^{e(\bj)}_{2}(f_{\bj+\bl}\circ\rho,2^{-\bj},\cdot) \Big\|_{L_{p}([-1,1]^d)}   \,.
\end{align*}
Plugging in~\eqref{est:central3} with $\sigma_p=0$ and $r=1+\tfrac{1}{p}$, the right-hand side can be further bounded from above by
\begin{align*}
     \sup_{\bj\in\N^{d}_{0}} \sum_{\bl\in\Z^d}   B_p(\bl) 2^{-\bl}  2^{|\bj+\bl|_1}      \Big( \sum_{\bk\in\Z^d} |\lambda_{\bj+\bl,\bk}|^p \Big)^{1/p} 
     \lesssim
     \sup_{\bj\in\N^{d}_{0}} \sum_{\bl\in\Z^d}  2^{|\bj+\bl|_1}      \Big( \sum_{\bk\in\Z^d} |\lambda_{\bj+\bl,\bk}|^p \Big)^{1/p} =  \| \{\lambda_{\bl,\bk}\}_{\bl,\bk} \|_{s^{1+1/p}_{p,1}b(\Gamma^d)}\,.
\end{align*}
Note that $B_p(\bl)\le B_\infty(\bl)=2^{\bl_-}$ and therefore $\max_{\bl\in\Z^d} B_p(\bl) 2^{-\bl}  =1$.
Taking Proposition~\ref{prop:analysis_CW} into account, the proof of this case is also finished.
\hfill\qed


\section{Chui-Wang characterization of $S^{r}_{p,q}B(\R^{d})$}
\label{sec:CW_char}

Let $\{\psi_{\bj,\bk}\}_{\bj,\bk}$ and $\{\tilde{\psi}_{\bj,\bk}\}_{\bj,\bk}$
be the primal and dual Chui-Wang wavelet systems from Subsection~\ref{ssec:CW_wavelets}. 
In Theorem~\ref{thm:CW_characterization} below, we give 
characterizations of
functions $f\in S^{r}_{p,q}B(\R^{d})$ 
in terms of their Chui-Wang coefficients
\begin{align}\label{def:CWcoeff}
\lambda_{\bj,\bk}(f) := 2^{\bj_+} \dual{f,\psi_{\bj,\bk}} \quad\text{and}\quad 
\tilde{\lambda}_{\bj,\bk}(f) := 2^{\bj_+}\dual{f,\tilde{\psi}_{\bj,\bk}} \,,\quad\text{where }\; (\bj,\bk)\in \Gamma^d=\N_{-1}^d\times \Z^d \,.
\end{align}
The product $\dual{\cdot,\cdot}$ is hereby an extension of the usual integral product $\langle\cdot,\cdot\rangle_{\R^d}$ to the dual pairing
$\mathscr{B}\times\tilde{\mathscr{B}}$. The precise definition is given in~\eqref{eqdef:gen_dual_prod} and~\eqref{def:BB}.

A univariate version of~Theorem~\ref{thm:CW_characterization}, which includes the Triebel-Lizorkin case, is due to Ullrich and Derevianko~\cite[Thm.~5.1]{DerUll19}. 
However, the statement there concerning strong convergence in $B^{r-\varepsilon}_{p,q}(\R)$ in the boundary situation $\max\{p,q\}=\infty$ is wrong if $p=\infty$. At present, this case is lacking a proper treatment.
Our theorem 
generalizes the Besov case of~\cite[Thm.~5.1]{DerUll19} to the multivariate dominating mixed setting
and now provides a correct statement.

\begin{theorem}\label{thm:CW_characterization}
Let $0<p,q\le\infty$ and $f\in S^{r}_{p,q}B(\R^{d})$ with $r\in\R$ in the range
\begin{align*}
\frac{1}{p}-2 < r < \min\Big\{  \frac{1}{p}+1, 2 \Big\} \,.
\end{align*}
Then we have the equivalence of (quasi-)norms
\begin{align*}
    \|f\|_{S^{r}_{p,q}B(\R^{d})} \;\asymp\;  \big\| \big\{\lambda_{\bj,\bk}(f) \big\}_{\bj,\bk} \big\|_{s^{r}_{p,q}b(\Gamma^d)}  \;\asymp\;   \big\|  \big\{ \tilde{\lambda}_{\bj,\bk}(f) \big\}_{\bj,\bk} \big\|_{s^{r}_{p,q}b(\Gamma^d)} \,.
\end{align*}
Further, $f$ can be expanded in the form
\[
 f= \sum_{(\bj,\bk)\in\Gamma^d} \tilde{\lambda}_{\bj,\bk}(f)\psi_{\bj,\bk} = \sum_{(\bj,\bk)\in\Gamma^d} \lambda_{\bj,\bk}(f) \tilde{\psi}_{\bj,\bk}
\]
with unique expansion coefficients from $s^{r}_{p,q}b(\Gamma^d)$. In general, one has weak convergence with respect to the conjugate space $S^{r'}_{p',q'}B(\R^{d})$ (see Definition~\ref{def:conj_space}). If $\max\{p,q\}<\infty$, one has unconditional convergence in $S^{r}_{p,q}B(\R^{d})$. 
\end{theorem}

\begin{remark}
The notion of weak convergence with respect to the conjugate space is explained in Appendix~\ref{ssec:Appendix_duality}. 
In the parameter range of Theorem~\ref{thm:CW_characterization}, weak convergence in $S^{r}_{p,q}B(\R^{d})$ with respect to $S^{r'}_{p',q'}B(\R^{d})$ is a stronger notion of convergence than weak convergence in $\funpool$ with respect to $\testpool$ (a pair of conjugate spaces defined in~\eqref{def:BB} further down), which in turn is stronger than weak* convergence in $\mathcal{S}^\prime(\R^d)$.
\end{remark}

\noindent
Since $f$ in Theorem~\ref{thm:CW_characterization} can be distributional in $\mathcal{S}^\prime(\R^{d})$ we may not use the integral product $\langle  \cdot,  \cdot \rangle_{\R^d}$ in~\eqref{def:CWcoeff} to calculate the Chui-Wang coefficients.  
Let us hence now comment on the precise definition of the product $\dual{\cdot,\cdot}$.
It is a duality product on a suitable pair of conjugate spaces $\mathscr{B}$ 
and $\tilde{\mathscr{B}}$ (see Appendix~\ref{ssecA:Calderon}). The formula reads
\begin{align}\label{eqdef:gen_dual_prod}
\langle f,g \rangle_{\mathscr{B}\times\tilde{\mathscr{B}}} := \sum_{\bl\in\N_{0}^{d}} \langle  \Phi_{\bl}\ast f,
	\Lambda_{\bl}^{\boldsymbol{-}}\ast g \rangle_{\R^d} \qquad\text{for $f\in \mathscr{B}$\,, $g\in\tilde{\mathscr{B}}$}\,,
\end{align}
with $\Lambda_{\bl}^{\boldsymbol{-}}:=\Lambda_{\bl}(-\cdot)$ and certain kernel families $\{\Phi_{\bl}\}_{\bl}$ and 
$\{\Lambda_{\bl}\}_{\bl}$ in $\mathcal{S}(\R^{d})$. The details are given in Appendix~\ref{ssecA:Calderon}. Hereby we choose the family $\{\Lambda_{\bl}\}_{\bl}$ as local mean kernels with compact support on the spatial side and band-limited kernels $\{\Phi_{\bl}\}_{\bl}$.  
The corresponding Calder\'on reproducing formula (see Appendix~\ref{ssecA:Calderon} for more information) is the expansion
\begin{align}\label{def:Calderon_concrete}
	f=\sum_{\bl\in\N_{0}^{d}}  \Phi_{\bl}\ast\Lambda_{\bl}\ast f   \quad\text{weak*ly for $f\in \mathcal{S}^\prime(\R^{d})$}\,.  
\end{align}

\noindent
It remains to fix $\mathscr{B}$ and $\tilde{\mathscr{B}}$. In analogy to~\cite[eq.~(1.12)]{GSU22}, we choose
\begin{align}\label{def:BB}
\mathscr{B}:= S^{-2}_{\infty,1}B(\R^d)  \quad,\quad \tilde{\mathscr{B}}:=S^{2}_{1,\infty}B(\R^d) \,,  
\end{align}
the space $\mathscr{B}$ serving as our reservoir and $\tilde{\mathscr{B}}$ as our test space. 
This pair satisfies
\[
\mathcal{S}(\R^{d}) \hookrightarrow \tilde{\mathscr{B}} \hookrightarrow L_2(\R^{d}) \hookrightarrow \mathscr{B} \hookrightarrow \mathcal{S}^\prime(\R^{d}) 
\]
as well as
$\testpool \hookrightarrow S^{r'}_{p',q'}B(\R^{d})$ 
and $S^{r}_{p,q}B(\R^{d}) \hookrightarrow \funpool$ in the considered parameter range. Furthermore, the Chui-Wang systems $\{\psi_{\bj,\bk}\}_{\bj,\bk}$ and $\{\tilde{\psi}_{\bj,\bk}\}_{\bj,\bk}$ are contained in $\tilde{\mathscr{B}}$ so that the coefficients in~\eqref{def:CWcoeff} are well-defined.

Let us remark that
$\funpool$ is not identical to the dual space $\testpool^\prime$.
Since $\mathcal{S}(\R^{d})$ is not dense in $\testpool$, its dual $\testpool^\prime$ is in fact not even a subspace of $\mathcal{S}^\prime(\R^{d})$.
This is the case, however, if we use the closure of $\mathcal{S}(\R^{d})$ in $\testpool$, denoted by $\testpool_0$, as test space. 
Its dual $\testpool_0^\prime$ is then a subspace of 
$\mathcal{S}^\prime(\R^{d})$ and one can show the identification
\begin{align*}
   \funpool \cong \testpool_0^\prime   \,.
\end{align*}
The product $\dual{\cdot,\cdot}$ hence truly extends the duality product $\langle\cdot,\cdot\rangle_{\testpool_0^\prime\times\testpool_0}$ to a larger test space.

After these preliminary considerations, 
let us turn back to Theorem~\ref{thm:CW_characterization}. Its proof relies on two propositions, Proposition~\ref{prop:analysis_CW} and Proposition~\ref{prop:synthesis_CW},
and their dual statements, Corollary~\ref{cor:analysis_CW} and  Corollary~\ref{cor:synthesis_CW}.
We start with Proposition~\ref{prop:analysis_CW}. It gives an estimate for the primal Chui-Wang coefficients.

\begin{prop}\label{prop:analysis_CW}
Let $0<p,q\le\infty$ and $\frac{1}{p}-2 < r < 2$.
Then for any $f\in S^{r}_{p,q}B(\R^{d})$ we have
\begin{align}\label{lem5.1(ii)}
\big\| \{\lambda_{\bj,\bk}(f)\}_{\bj,\bk}  \big\|_{s^{r}_{p,q}b(\Gamma^d)} 
\lesssim \|f\|_{S^{r}_{p,q}B(\R^{d})} 
\end{align}
for the primal Chui-Wang coefficients $\lambda_{\bj,\bk}(f)$ defined as in~\eqref{def:CWcoeff}.
\end{prop}
\begin{proof}
    To use $\Z^d$ as index set in~\eqref{eqdef:gen_dual_prod} and in~\eqref{def:Calderon_concrete}, we append the kernel families $\{\Lambda_{\bl}\}_{\bl}$ and $\{\Phi_{\bl}\}_{\bl}$ by the trivial functions $\Lambda_{\bl}:=\Phi_{\bl}:=0$
for $\bl\in\Z^d\backslash\N_0^d$. Further, we put $\Lambda_{-l}:=\Phi_{-l}:=0$
for $l\in\N$.

    Using the reproducing formula~\eqref{def:Calderon_concrete}, we now estimate 
	\begin{align}\label{def:alpha_jk}
	|\lambda_{\bj,\bk}(f)| = 2^{\bj_+} |\dual{f,\psi_{\bj,\bk}}|
	 \lesssim \sum_{\bl\in\Z^{d}} \int_{\R^{d}} 2^{\bj_+} |(\Phi_{\bj+\bl}\ast f)(y)| |(\Lambda_{\bj+\bl}\ast \psi_{\bj,\bk})(y)| \,\dint y=: \alpha_{\bj,\bk}(f) \,.
	\end{align}
        Hence we can show~\eqref{lem5.1(ii)} by proving the stronger statement
        \begin{align}\label{stronger_analysis}
        \big\| \{\alpha_{\bj,\bk}(f)\}_{\bj,\bk}  \big\|_{s^{r}_{p,q}b(\Gamma^d)} 
\lesssim \|f\|_{S^{r}_{p,q}B(\R^{d})} \,.
        \end{align}
 This is what we will do in the subsequent proof. Let us begin by investigating the integral in the definition of $\alpha_{\bj,\bk}(f)$. 
 It is useful for this to define
  $S_{\bj,\bk,\bl}:= \supp \big(\Lambda_{\bj+\bl}\ast\psi_{\bj,\bk} \big)$
        and to note that
        \begin{align}\label{def:conv_support}
	S_{\bj,\bk,\bl} = S_{j_1,k_1,l_1} \times \cdots \times S_{j_d,k_d,l_d} \quad\text{with}\quad S_{j_i,k_i,l_i} :=\supp \big(\Lambda_{j_i+l_i}\ast \psi_{j_i,k_i}\big) 
	\quad\text{for }i\in[d]  \,.
	\end{align}
 We next 
        distinguish the cases $\min_{i\in[d]}\{j_i+l_i\}<0$ and  $\min_{i\in[d]}\{j_i+l_i\}\ge0$. In the first case, the integral vanishes since 
        $S_{\bj,\bk,\bl} =\emptyset$.
	In case $\min_{i\in[d]}\{j_i+l_i\}\ge0$, we obtain for fixed $a>0$ and any $x\in\R^{d}$
	\begin{align}\label{est:Calderon}
	 &\int_{\R^{d}} 2^{\bj_+} |(\Phi_{\bj+\bl}\ast f)(y)| |(\Lambda_{\bj+\bl}\ast \psi_{\bj,\bk})(y)| \,\dint y \\
	&\le P_{\bb(\bj+\bl),a}(\Phi_{\bj+\bl}\ast f)(x) \cdot \sup_{y\in S_{\bj,\bk,\bl}} \Big[ \prod_{i=1}^{d} (1+2^{j_i+l_i}|x_i-y_i|)^a \Big]
	\cdot \int_{\R^{d}} 2^{\bj_+}  |(\Lambda_{\bj+\bl}\ast \psi_{\bj,\bk})(z)| \,\dint z \,, \notag
	\end{align}
	where (see Appendix~\ref{ssec:Appendix_Peetre} for more information)
	\begin{align*}
	P_{\bb(\bj+\bl),a}(\Phi_{\bj+\bl}\ast f)(x) := \sup_{y\in\R^{d}} \Big[ |(\Phi_{\bj+\bl}\ast f)(y)| \cdot \prod_{i=1}^{d} (1+2^{j_i+l_i}|x_i-y_i|)^{-a} \Big]  
	\end{align*}
       is the Peetre maximal function of $\Phi_{\bj+\bl}\ast f$ 
	with respect to the parameters $a$ and $\bar{b}(\bj+\bl):=(2^{j_1+l_1},\ldots,2^{j_d+l_d})$.

  Let $Q_{\bj,\bk}$ denote the cubes defined in~\eqref{def:chi_cubes}. We next want to prove, for $x$ restricted to $Q_{\bj,\bk}$, 
  \begin{align}\label{est:Calderon2}
  \sup_{y\in S_{\bj,\bk,\bl}} \Big[ \prod_{i=1}^{d} (1+2^{j_i+l_i}|x_i-y_i|)^a \Big]
	\cdot \int_{\R^{d}} 2^{\bj_+}  |(\Lambda_{\bj+\bl}\ast \psi_{\bj,\bk})(z)| \,\dint z  \lesssim  A(\bl,a)  \,,
  \end{align}
  where
	\begin{align*}
	A(\bl,a):=\prod_{i\in[d]} A(l_i,a) \quad\text{with}\quad
	A(l_i,a):= \begin{cases} 2^{2l_i} \quad &,\, l_i<0 \,, \\ 2^{(a-2)l_i} &,\, l_i\ge 0 \,. \end{cases}
	\end{align*}
 
	Let us first estimate the second term in~\eqref{est:Calderon2}. We have for $\bj\in\N^{d}_{-1}$
    \begin{align*}
	 \int_{\R^{d}} 2^{\bj_+}  |(\Lambda_{\bj+\bl}\ast \psi_{\bj,\bk})(z)| \,\dint z  \lesssim  
	 \int_{\R} 2^{j^+_1}|(\Lambda_{j_1+l_1}\ast \psi_{j_1,k_1})(t)|\,\dint t \cdots \int_{\R} 2^{j^+_d}|(\Lambda_{j_d+l_d}\ast \psi_{j_d,k_d})(t)|\,\dint t
	\end{align*}
    due to the tensor-product structure
	\begin{align*}
	\Lambda_{\bj+\bl}\ast \psi_{\bj,\bk} = (\Lambda_{j_1+l_1}\ast \psi_{j_1,k_1}) \otimes \cdots \otimes (\Lambda_{j_d+l_d}\ast \psi_{j_d,k_d}) \,.
	\end{align*}
Now we make the two observations (i) and (ii) below.
 \begin{enumerate}[leftmargin=*]
	    \item[(i)] For $i\in[d]$ with $l_i\le0$ we have 
	\begin{align}\label{l-}
	    \|\Lambda_{j_i+l_i}\ast \psi_{j_i,k_i}\|_{L_\infty(\R)} \lesssim 2^{3l_i} \qquad\text{and}\qquad\|\chi_{S_{j_i,k_i,l_i}}\|_{L_1(\R)} \lesssim 2^{-j_i-l_i} \,.
	\end{align} 
        \begin{proof}[Proof of~\eqref{l-}]\renewcommand{\qed}{$\square$} 
        First note that clearly $\|\chi_{S_{j_i,k_i,l_i}}\|_{L_1(\R)} \lesssim 2^{-j_i-l_i}$, where $ S_{j_i,k_i,l_i} =\supp \big(\Lambda_{j_i+l_i}\ast \psi_{j_i,k_i}\big)$. Next, $\| \Lambda_{j_i+l_i}\ast \psi_{j_i,k_i} \|_{L_\infty(\R)} \lesssim 2^{3l_i}$ is shown using the two vanishing moments of the Chui-Wang wavelet $\psi_{j_i,k_i}$ when $j_i\ge0$, which is implied by
        our assumption $\min_{i\in[d]}\{j_i+l_i\}\ge0$.
        Due to these vanishing moments
	\begin{align*}
	|(\Lambda_{j_i+l_i}&\ast \psi_{j_i,k_i})(x)| =  \Big| \int_{\R}  \Lambda_{j_i+l_i}(x-y)   \psi_{j_i,k_i}(y) \,\dint y\Big| \\
	&= \Big| \int_{\R} \Big( \Lambda_{j_i+l_i}(x-y) - \Lambda_{j_i+l_i}(x-2^{-j_i}k_i) - \Lambda^\prime_{j_i+l_i}(x-2^{-j_i}k_i)(2^{-j_i}k_i-y) \Big) \psi_{j_i,k_i}(y) \,\dint y \Big| \\
    &\le \sup_{y\in \supp(\psi_{j_i,k_i})}  |2^{-j_i}k_i-y|^{2}   \|\Lambda^{\prime\prime}_{j_i+l_i}\|_{L_\infty(\R)} \|\psi_{j_i,k_i}\|_{L_1(\R)}
	\lesssim 2^{3l_i} \,,
	\end{align*}
       where in the last line a second-order Taylor estimate was applied. 
      Also note that, by definition, $\{\Lambda_{j}\}_{j\in\N_{0}}$ and $\{\psi_{j,k}\}_{(j,k)\in\N_{-1}\times\Z}$ are $L_1$ resp.\ $L_\infty$ normalized function systems.
        \end{proof}
	\item[(ii)] 
        For $i\in[d]$ with $l_i\ge0$ we have 
	\begin{align}\label{l+}
	    \|\Lambda_{j_i+l_i}\ast \psi_{j_i,k_i}\|_{L_\infty(\R)} \lesssim 2^{-l_i} 
\qquad\text{and}\qquad\|\chi_{S_{j_i,k_i,l_i}}\|_{L_1(\R)} \lesssim 2^{-j_i-l_i} \,.
	\end{align}
	\begin{proof}[Proof of~\eqref{l+}]\renewcommand{\qed}{$\square$}
        By substitution, we obtain
        \begin{align*}
	|(\Lambda_{j_i+l_i}\ast \psi_{j_i,k_i})(x)| = \Big| \int_{\R}  \Lambda_{j_i+l_i}(x-y)   \psi_{j_i,k_i}(y) \,\dint y\Big| 
	= \Big| \int_{\R}  \Lambda_{l_i}(2^{j_i} x- y)   \psi_{0,k_i}(y) \,\dint y\Big| \,.
    \end{align*}
    Then we use the representation~\eqref{represent:CW_wavelets}, which yields 
	\begin{align*}
   \psi_{0,k_i}(y) = \sum_{\kappa=0}^6 c_{\kappa} \big(y - k_i - \tfrac{\kappa}{2} \big)_{+} \,.
    \end{align*}
    Finally, taking into account $|\supp(\Lambda_{l_i})|\lesssim 2^{-l_i}$, we arrive at
	\begin{align*}
	|(\Lambda_{j_i+l_i}\ast \psi_{j_i,k_i})(x)| 
    &\le \sum_{\kappa=0}^6 |c_{\kappa}|  \cdot  \Big| \int_{\R} \Lambda_{l_i}(2^{j_i} x- y)  \big(y - k_i - \tfrac{\kappa}{2} \big)_{+} \,\dint y\Big|   \\
    &\le \sum_{\kappa=0}^6 |c_{\kappa}| \cdot   2^{l_i} \max\big\{ \|\Lambda_{0}\|_{L_\infty(\R)}, \|\Lambda_{1}\|_{L_\infty(\R)} \big\} \int\limits_{\supp(\Lambda_{l_i})}  \big(y - 2^{j_i} x - k_i - \tfrac{\kappa}{2} \big)_{+} \,\dint y  \lesssim 2^{-l_i} \,.
	\end{align*}
    As a consequence of the two vanishing moments of $\Lambda_{j_i+l_i}$,
    the support $S_{j_i,k_i,l_i}$ of $\Lambda_{j_i+l_i}\ast \psi_{j_i,k_i}$  merely stems from the kinks of the function $\psi_{j_i,k_i}$. It is thus contained in at most 7 intervals of length $\mathcal{O}(2^{-j_i-l_i})$.
    \end{proof}
	\end{enumerate}

	\noindent
	From~\eqref{l-} and~\eqref{l+}, we deduce
	\begin{align}\label{auxformula1}
	\int_{\R} 2^{j^+_i}|(\Lambda_{j_i+l_i}\ast \psi_{j_i,k_i})(t)| \,\dint t
	\lesssim 2^{-2|l_i|} \,.
	\end{align}
	

	\noindent
	Let us next turn our attention to the first term in~\eqref{est:Calderon2}. Here we show for $x_i\in Q_{j_i,k_i}$, with $Q_{j_i,k_i}$ as in~\eqref{def:chi_cubes}, 
        \begin{align}\label{auxformula2}
	\sup_{y_i\in S_{j_i,k_i,l_i}} (1+2^{j_i+l_i}|x_i-y_i|)^a \lesssim  \begin{cases} 1 \,&,\, l_i<0 \,, \\  2^{al_i}  &,\,l_i\ge0 \,.  \end{cases} 
	\end{align}
      For this, let $x_i\in Q_{j_i,k_i}=[2^{j_i^+}k_i,2^{j_i^+}(k_i+1)]$ and $y_i\in S_{j_i,k_i,l_i}$ and note
      \[
      S_{j_i,k_i,l_i}=\supp (\Lambda_{j_i+l_i}\ast \psi_{j_i,k_i}) 
	\subset  \supp(\Lambda_{j_i+l_i}) + \supp(\psi_{j_i,k_i}) \,,
      \]
      which allows to write $y_i=z_i + \tilde{z}_i$ 
      with $z_i\in\supp(\Lambda_{j_i+l_i})$ and $\tilde{z}_i\in \supp(\psi_{j_i,k_i})$. Since $|z_i|\lesssim 2^{-j_i-l_i}$ and $|x_i - \tilde{z}_i|\lesssim 2^{-j_i}$ for $x_i\in Q_{j_i,k_i}$, we can hence estimate
      \[
      2^{j_i+l_i}|x_i-y_i| \le 2^{j_i+l_i}\big( |z_i| + |x_i-\tilde{z}_i| \big) \lesssim 1 + 2^{l_i} \,.
      \]
      This directly proves~\eqref{auxformula2}.
      Further, estimate~\eqref{auxformula2} together with estimate~\eqref{auxformula1} yields~\eqref{est:Calderon2}.

   Using~\eqref{est:Calderon2}, we can proceed from~\eqref{est:Calderon} and get for any $x\in Q_{\bj,\bk}$ 
	\begin{align*}
	\int_{\R^{d}} 2^{\bj_+} |(\Phi_{\bj+\bl}\ast f)(y)|| (\Lambda_{\bj+\bl}\ast \psi_{\bj,\bk})(y)| \,\dint y \lesssim
	A(\bl,a) \cdot P_{\bb(\bj+\bl),a}(\Phi_{\bj+\bl}\ast f)(x) \,.
	\end{align*}	
	Hence, with the characteristic function $\chi_{\bj,\bk}$ of $Q_{\bj,\bk}$, we obtain for fixed $\bj\in\N^{d}_{-1}$ and uniformly in $x\in\R^{d}$
	\begin{align}\label{auxformula3}
	\sum_{\bk\in\Z^{d}} |\alpha_{\bj,\bk}(f)| \chi_{\bj,\bk}(x)  \lesssim
	\sum_{\bl\in\Z^{d}} A(\bl,a) \cdot P_{\bb(\bj+\bl),a}(\Phi_{\bj+\bl}\ast f)(x) \,.
	\end{align}
	
    \noindent
    With~\eqref{auxformula3}, we are finally prepared to estimate the expression
    \begin{align}\label{expression1}
    \big\| \{\alpha_{\bj,\bk}(f)\}_{\bj,\bk}  \big\|_{s^{r}_{p,q}b(\Gamma^d)} = \Big( \sum_{\bj\in\N^{d}_{-1}} 2^{\bj_+ rq} \Big\| \sum_{\bk\in\Z^{d}} |\alpha_{\bj,\bk}(f)| \chi_{\bj,\bk}(\cdot)  \Big\|_{L_{p}(\R^{d})}^q \Big)^{1/q}  \,.
    \end{align}
    Let $u:=\min\{1,p,q\}$ and $\tilde{p}:=p/u$, $\tilde{q}:=q/u$, $\tilde{r}:=ru$.
    Inserting~\eqref{auxformula3} in~\eqref{expression1}, we obtain
    \begin{align*}
    \big[ \eqref{expression1} \big]^u &\lesssim  \Big( \sum_{\bj\in\N^{d}_{-1}}  \Big\| \sum_{\bl\in\Z^{d}} 2^{\bj_+ \tilde{r}} A(\bl,a)^u \big|P_{\bb(\bj+\bl),a}(\Phi_{\bj+\bl}\ast f)(\cdot)\big|^u  \Big\|_{L_{\tilde{p}}(\R^{d})}^{\tilde{q}} \Big)^{1/\tilde{q}}  \\
    &\le \sum_{\bl\in\Z^{d}}  2^{-\bl \tilde{r}} A(\bl,a)^u  \Big( \sum_{\bj\in\N^{d}_{-1}} \Big\|  2^{\bl \tilde{r}}  2^{\bj_+ \tilde{r}}  \big|P_{\bb(\bj+\bl),a}(\Phi_{\bj+\bl}\ast f)(\cdot)\big|^u  \Big\|_{L_{\tilde{p}}(\R^{d})}^{\tilde{q}} \Big)^{1/\tilde{q}}  \\
    &\lesssim \sum_{\bl\in\Z^{d}}  2^{-\bl \tilde{r}} A(\bl,a)^u  \Big( \sum_{\bj\in\N^{d}_{-1}} \Big\|  2^{(\bj+\bl)\tilde{r}}    \big|P_{\bb(\bj+\bl),a}(\Phi_{\bj+\bl}\ast f)(\cdot)\big|^u  \Big\|_{L_{\tilde{p}}(\R^{d})}^{\tilde{q}} \Big)^{1/\tilde{q}} \,, 
    \end{align*}
    where we used $|j_i|\le j_i+2$ for $j_i\in\N_{-1}$ in the last line. The index shift $\bj+\bl\rightsquigarrow \bj$ yields at last
    \begin{align*}
    \big[ \eqref{expression1} \big]^u &\lesssim    \sum_{\bl\in\Z^{d}}  2^{-\bl \tilde{r}} A(\bl,a)^u\cdot \Big( \sum_{\bj\in\N^{d}_{0}}    \Big\|    2^{\bj_+ \tilde{r}}  \big| P_{\bb(\bj),a}(\Phi_{\bj}\ast f)(\cdot)\big|^u   \Big\|_{L_{\tilde{p}}(\R^{d})}^{\tilde{q}} \Big)^{1/\tilde{q}}  \,.
    \end{align*}
   
    \noindent
    The assumption $r>\tfrac{1}{p}-2$ allows to choose $a$ such that $1/p<a<2+r$. Since $r<2$ by assumption,
    we get
   \begin{align*}
    \sum_{\bl\in\Z^{d}} 2^{-\bl \tilde{r}} A(\bl,a)^u   = \sum_{\bl\in\Z^{d}}  \prod_{i=1}^{d} 2^{(2-r)ul^-_i}  2^{(a-2- r)ul^+_i}  <\infty \,.
    \end{align*}
    According to Peetre's maximal inequality (Theorem~\ref{peetremax}),
    we further get
    \begin{align*}
    \Big( \sum_{\bj\in\N^{d}_{0}}    \Big\|    2^{\bj_+ \tilde{r}}  \big| P_{\bb(\bj),a}(\Phi_{\bj}\ast f)(\cdot)\big|^u   \Big\|_{L_{\tilde{p}}(\R^{d})}^{\tilde{q}} \Big)^{1/q}
    &=
    \Big( \sum_{\bj\in\N^{d}_{0}}    \Big\|    2^{\bj_+ r}  P_{\bb(\bj),a}(\Phi_{\bj}\ast f)(\cdot)  \Big\|_{L_{p}(\R^{d})}^q \Big)^{1/q}   \\
    &\lesssim  \Big( \sum_{\bj\in\N^{d}_{0}}    \Big\|    2^{\bj_+ r} (\Phi_{\bj}\ast f)(\cdot)  \Big\|_{L_{p}(\R^{d})}^q \Big)^{1/q} \,. 
    \end{align*}
    The proof is finished since the right-hand expression is equivalent to $\|f\|_{S^{r}_{p,q}B(\R^{d})}$. 
\end{proof}


\noindent
From Proposition~\ref{prop:analysis_CW} we can directly deduce a result for the dual Chui-Wang coefficients $\tilde{\lambda}_{\bj,\bk}(f)$. 

\begin{corollary}\label{cor:analysis_CW}
Under the assumptions of Proposition~\ref{prop:analysis_CW} we also have
\begin{align*}
\big\| \{\tilde{\lambda}_{\bj,\bk}(f)\}_{\bj,\bk} \big\|_{s^{r}_{p,q}b(\Gamma^d)} 
\lesssim \|f\|_{S^{r}_{p,q}B(\R^{d})} \,,
\end{align*}
where $\tilde{\lambda}_{\bj,\bk}(f)$ are the dual Chui-Wang coefficients of $f$ defined in~\eqref{def:CWcoeff}.
\end{corollary}
\begin{proof}
Recall the representation~\eqref{dualCW_representation} of the dual Chui-Wang wavelets as linear combinations of the primal wavelets. Further, recall  Remark~\ref{rem:Besov_seq_spaces} and the functions $\chi_{\bj,\bk}$ introduced there. Putting $u:=\min\{1,p,q\}$, we estimate
\begin{align*}
\big\| \{\tilde{\lambda}_{\bj,\bk}(f)\}_{\bj,\bk} \big\|_{s^{r}_{p,q}b(\Gamma^d)}
&\asymp
\Big(  \sum_{\bj\in\N^{d}_{-1}} 2^{\bj_+ rq} \Big\| \sum_{\bk\in\Z^{d}} \Big| \sum_{\bn\in\Z^{d}}  a^{(\bj)}_{\bn} \lambda_{\bj,\bk+\bn}(f) \Big| \chi_{\bj,\bk}(\cdot)  \Big\|_{L_{p}(\R^{d})}^q  \Big)^{1/q}  \\
&\le
\Big(  \sum_{\bj\in\N^{d}_{-1}} 2^{\bj_+ rq} \Big( \sum_{\bn\in\Z^{d}} |a^{(\bj)}_{\bn}|^{u} \Big\| \Big[ \sum_{\bk\in\Z^{d}}  |\lambda_{\bj,\bk+\bn}(f)| \chi_{\bj,\bk}(\cdot)  \Big]^u \Big\|_{L_{p/u}(\R^{d})} \Big)^{q/u}  \Big)^{1/q}  \\
&=
\Big(  \sum_{\bj\in\N^{d}_{-1}} 2^{\bj_+ rq} \Big( \sum_{\bn\in\Z^{d}} |a^{(\bj)}_{\bn}|^{u} \Big\| \Big[\sum_{\bk\in\Z^{d}}  |\lambda_{\bj,\bk}(f)| \chi_{\bj,\bk}(\cdot) \Big]^u  \Big\|_{L_{p/u}(\R^{d})} \Big)^{q/u}  \Big)^{1/q} \,.
\end{align*}
Utilizing the decay~\eqref{eq:coeff_decay} of the coefficients $a^{(\bj)}_{\bn}$, we obtain $\sum_{\bn\in\Z^{d}} |a^{(\bj)}_{\bn}|^{u} \lesssim 1$. Hence,
with Proposition~\ref{prop:analysis_CW},
\begin{gather*}
\big\| \{\tilde{\lambda}_{\bj,\bk}(f)\}_{\bj,\bk} \big\|_{s^{r}_{p,q}b(\Gamma^d)} 
\lesssim
\Big(  \sum_{\bj\in\N^{d}_{-1}} 2^{\bj_+ rq} \Big\| \sum_{\bk\in\Z^{d}} |\lambda_{\bj,\bk}(f)| \chi_{\bj,\bk}(\cdot)  \Big\|_{L_{p}(\R^{d})}^q  \Big)^{1/q} 
\lesssim
 \|f\|_{S^{r}_{p,q}B(\R^{d})} \,. \qedhere
\end{gather*}
\end{proof}


\noindent
An interesting endpoint result is given by the following lemma.

\begin{lemma}\label{lem:analysis_CW}
Let $\tfrac{1}{4}<p\le\infty$, $\frac{1}{p}-2 < r \le 2$, and $f\in S^{r}_{p,v}B(\R^{d})$ with $v=\min\{p,1\}$.
Then 
\begin{align*}
\big\| \{\lambda_{\bj,\bk}(f)\}_{\bj,\bk}  \big\|_{s^{r}_{p,\infty}b(\Gamma^d)} 
\lesssim \|f\|_{S^{r}_{p,v}B(\R^{d})} \quad\text{and}\quad \big\| \{\tilde{\lambda}_{\bj,\bk}(f)\}_{\bj,\bk}  \big\|_{s^{r}_{p,\infty}b(\Gamma^d)} 
\lesssim \|f\|_{S^{r}_{p,v}B(\R^{d})} \,.
\end{align*}
\end{lemma}
\begin{proof}
The proof of the primal statement is for the most part analogous to the proof of Proposition~\ref{prop:analysis_CW}. It only deviates after estimate~\eqref{auxformula3}.
Here we get for $\alpha_{\bj,\bk}(f)$ defined as in~\eqref{def:alpha_jk},
 \begin{align*}
    \big\| \{\alpha_{\bj,\bk}(f)\}_{\bj,\bk}  \big\|_{s^{r}_{p,\infty}b(\Gamma^d)}   &\lesssim  \Big( \sup\limits_{\bj\in\N^{d}_{-1}}  \Big\| \sum_{\bl\in\Z^{d}} 2^{|\bj|_1 r} A(\bl,a) P_{\bb(\bj+\bl),a}(\Phi_{\bj+\bl}\ast f)(\cdot)  \Big\|_{L_{p}(\R^{d})}^{v} \Big)^{1/v}  \\
    &\lesssim \Big(\sup\limits_{\bl\in\Z^{d}}  2^{-\bl r} A(\bl,a) \Big)  \Big( \sup\limits_{\bj\in\N^{d}_{-1}} \sum_{\bl\in\Z^{d}} \Big\|  2^{(\bj+\bl)r}    P_{\bb(\bj+\bl),a}(\Phi_{\bj+\bl}\ast f)(\cdot)  \Big\|_{L_{p}(\R^{d})}^{v} \Big)^{1/v} \\
    &\lesssim  \Big( \sum_{\bj\in\N^{d}_{-1}} \Big\|  2^{|\bj|_1 r}    P_{\bb(\bj),a}(\Phi_{\bj}\ast f)(\cdot)  \Big\|_{L_{p}(\R^{d})}^{v} \Big)^{1/v} \lesssim \|f\|_{S^{r}_{p,v}B(\R^{d})} \,. 
    \end{align*}
The dual statement is proved analogously to Corollary~\ref{cor:analysis_CW}.
\end{proof}

\noindent
Proposition~\ref{prop:synthesis_CW}, the other
essential ingredient in the proof of Theorem~\ref{thm:CW_characterization}, is proved next. It deals with the synthesis of functions from primal Chui-Wang wavelets as building blocks.

\begin{prop}\label{prop:synthesis_CW}
Let $0<p,q\le \infty$. 
Under the condition $\sigma_p-2 < r < 1 + \frac{1}{p}$ with $\sigma_p=(\frac{1}{p}-1)_+$
the sum 
\begin{align}\label{eqdef:f_synthesis}
f := \sum_{(\bj,\bk)\in\Gamma^d}  \mu_{\bj,\bk} \psi_{\bj,\bk}
\end{align}
converges weakly in $S^{r}_{p,q}B(\R^{d})$ for any sequence $\{\mu_{\bj,\bk}\}_{\bj,\bk} \in s^r_{p,q}b(\Gamma^d)$ with respect to the conjugate space $S^{r'}_{p',q'}B(\R^{d})$.
Moreover, the limit functions $f\in S^{r}_{p,q}B(\R^{d})$ fulfill the estimate 
\begin{align}\label{eqref:syn_estimate}
\| f \|_{S^{r}_{p,q}B(\R^{d})} \lesssim 
\big\| \{\mu_{\bj,\bk}\}_{\bj,\bk} \big\|_{s^{r}_{p,q}b(\Gamma^d)}\,.
\end{align}
If $\max\{p,q\}<\infty$ the convergence of~\eqref{eqdef:f_synthesis} is strong and unconditional in $S^{r}_{p,q}B(\R^{d})$.
\end{prop}
\begin{proof}
Let $\{\mu_{\bj,\bk}\}_{\bj,\bk} \in s^r_{p,q}b(\Gamma^d)$ and put 
$\mu_{\bj,\bk}:= 0$ and $\psi_{\bj,\bk}:=0$ if $j_i<-1$ for some $i\in[d]$, so that
\[
\sum_{(\bj,\bk)\in\Z^d\times\Z^d}  \mu_{\bj,\bk} \psi_{\bj,\bk} = \sum_{(\bj,\bk)\in\Gamma^d}  \mu_{\bj,\bk} \psi_{\bj,\bk}  \,.
\]
Further, for the univariate wavelets, we also define $\psi_{j,k}:=0$ if $j<-1$.

\noindent
{\bf Part (i)\,[}{\it Weak* convergence in $\mathcal{S}^\prime(\R^{d})$}{\bf]:}\:
We take a test function $\xi\in S^{r'}_{p',q'}B(\R^{d})$ with parameters as in~\eqref{eqdef:general_conj_exponents}. Since $\frac{1}{p'}-2 < r' < 2$ by the assumptions on $p$ and $r$, we can apply Proposition~\ref{prop:analysis_CW}, which yields
\[
\{2^{\bj_+} \dual{\xi,\psi_{\bj,\bk}}\}_{\bj,\bk} \in s^{r'}_{p',q'}b(\Gamma^d) \quad\text{or equivalently}\quad  \{\dual{\xi,\psi_{\bj,\bk}}\}_{\bj,\bk}\in s^{r'+1}_{p',q'}b(\Gamma^d) \,.
\]
The dual pairing between $s^{r}_{p,q}b(\Gamma^d)$ and $s^{r'+1}_{p',q'}b(\Gamma^d)$ is natural, as shown in Lemma~\ref{lem:Hölder_discr}.
We get
\begin{align}\label{central_weak_est}
\sum_{(\bj,\bk)\in\Z^{d}\times\Z^{d}} |\mu_{\bj,\bk} \dual{\xi,\psi_{\bj,\bk}}|
\le \big\|\{\mu_{\bj,\bk}\}_{\bj,\bk}\big\|_{s^{r}_{p,q}b(\Gamma^d)} \big\| \{\dual{\xi,\psi_{\bj,\bk}}\}_{\bj,\bk} \big\|_{s^{r'+1}_{p',q'}b(\Gamma^d)}  \lesssim \big\|\{\mu_{\bj,\bk}\}_{\bj,\bk}\big\|_{s^{r}_{p,q}b(\Gamma^d)}  \|\xi\|_{S^{r'}_{p',q'}B(\R^d)}\,.
\end{align}
The series~\eqref{eqdef:f_synthesis} hence converges weak*ly in $\big[S^{r'}_{p',q'}B(\R^{d})\big]^\prime$ to some limit $\tilde{f}\in\big[S^{r'}_{p',q'}B(\R^{d})\big]^\prime$, which can naturally be identified with an element $f\in\mathcal{S}^\prime(\R^{d})$, but note that
in case $\max\{p',q'\}=\infty$, i.e.\ when $\mathcal{S}(\R^{d})$ is not dense in $S^{r'}_{p',q'}B(\R^{d})$, the space $\big[S^{r'}_{p',q'}B(\R^{d})\big]^\prime$ is not
a subspace of $\mathcal{S}^\prime(\R^{d})$.
Rereading~\eqref{central_weak_est} for test functions $\varphi\in\mathcal{S}(\R^d)$, we can further deduce that $f\in\big[\overset{\circ}{S}{}^{r'}_{p',q'}B(\R^d)\big]^\prime$, where $\overset{\circ}{S}{}^{r'}_{p',q'}B(\R^d)$ is the closure of $\mathcal{S}(\R^{d})$ in $S^{r'}_{p',q'}B(\R^d)$. 
One can identify (cf.~\eqref{eq:conj_bidual}) 
\begin{align*}
\big[\overset{\circ}{S}{}^{r'}_{p',q'}B(\R^d)\big]^\prime\cong S^{r^{\prime\prime}}_{p^{\prime\prime},q^{\prime\prime}}B(\R^d) = S^{r-\sigma_p}_{\max\{p,1\},\max\{q,1\}}B(\R^d)\,.
\end{align*}
Altogether, we have proven unconditional weak* convergence of~\eqref{eqdef:f_synthesis} in $\big[\overset{\circ}{S}{}^{r'}_{p',q'}B(\R^d)\big]^\prime$ to some $f\in S^{r^{\prime\prime}}_{p^{\prime\prime},q^{\prime\prime}}B(\R^d)$, which in particular entails unconditional weak* convergence to $f$ in $\mathcal{S}^\prime(\R^{d})$, i.e.\
\begin{align}\label{eq:weakstarconv}
\langle f,\varphi\rangle_{\mathcal{S}^\prime\times\mathcal{S}} = \sum_{(\bj,\bk)\in\Z^d\times\Z^{d}}  \mu_{\bj,\bk} \langle\psi_{\bj,\bk},\varphi\rangle_{\mathcal{S}^\prime\times\mathcal{S}} \quad\text{for every $\varphi\in\mathcal{S}(\R^{d})$.}
\end{align}

\noindent
{\bf Part (ii)\,[}{\it Limit properties}{\bf]:}\:
Next, we want to show 
$f\in S^{r}_{p,q}B(\R^{d})$ and
\begin{align*}
\|f\|_{S^{r}_{p,q}B(\R^d)} \lesssim \big\| \{\mu_{\bj,\bk}\}_{\bj,\bk} \big\|_{s^{r}_{p,q}b(\Gamma^d)} \,. 
\end{align*}
We already know from Part~(i) that $f\in S^{r^{\prime\prime}}_{p^{\prime\prime},q^{\prime\prime}}B(\R^d)$, which settles $f\in S^{r}_{p,q}B(\R^{d})$ in the Banach case, when $\min\{p,q\}\ge1$.
In general, however, $f\in S^{r^{\prime\prime}}_{p^{\prime\prime},q^{\prime\prime}}B(\R^d)$ is a weaker statement than $f\in S^{r}_{p,q}B(\R^{d})$ since, according to~\eqref{eq:conj_bidual}, one only has the embedding 
\[
S^{r}_{p,q}B(\R^d) \hookrightarrow S^{r^{\prime\prime}}_{p^{\prime\prime},q^{\prime\prime}}B(\R^d) \,.
\] 
To prove $f\in S^{r}_{p,q}B(\R^{d})$ for the full parameter range $0<p,q\le\infty$, let $\{\Lambda_{\bj}\}_{\bj}\subset \mathcal{S}(\R^d)$ be a family of local mean kernels 
with $L>r$ vanishing
moments and compact support, as specified in Definition~\ref{def:localmeansfam}. Then, the weak* convergence of~\eqref{eq:weakstarconv}
translates to the identity 
\begin{align}\label{eq:weakstarconv2}
    \Lambda_{\bj} \ast f = \Lambda_{\bj} \ast \Big( \sum_{(\bl,\bk)\in\Z^d\times\Z^{d}}    \mu_{\bl,\bk}   \psi_{\bl,\bk} \Big) = \sum_{(\bl,\bk)\in\Z^d\times\Z^{d}}    \mu_{\bl,\bk} (\Lambda_{\bj} \ast \psi_{\bl,\bk}) =  \sum_{(\bl,\bk)\in\Z^d\times\Z^{d}}    \mu_{\bj+\bl,\bk} (\Lambda_{\bj} \ast \psi_{\bj+\bl,\bk}) 
\end{align}
in $\mathcal{S}^\prime(\R^d)$ with weak* convergence of the occurring sums. Our following considerations will show that, almost everywhere, the two right-most sums also converge pointwise absolutely.

Let $i\in[d]$ and $S_{j_i,k_i,l_i}:=\supp (\Lambda_{j_i+l_i}\ast \psi_{j_i,k_i})$ (cf.~\eqref{def:conv_support}). Analogously to \eqref{l-} and \eqref{l+}, we obtain
\begin{align*}
\|\chi_{S_{j_i+l_i,k_i,-l_i}}\|_{L_1(\R)} \lesssim 2^{-j_i} \quad\text{and}\quad |  \Lambda_{j_i} \ast \psi_{j_i+l_i,k_i} (x) | \lesssim \begin{cases} 2^{-3l_i}  \chi_{S_{j_i+l_i,k_i,-l_i}}  &,\, l_i\ge 0 \,, \\ 
   2^{l_i}  \chi_{S_{j_i+l_i,k_i,-l_i}} &,\, l_i< 0 \,. \end{cases}
	\end{align*} 

\noindent
Now, let us define $S_{\bj,\bk,\bl} := S_{j_1,k_1,l_1} \times \cdots \times S_{j_d,k_d,l_d}$ and
\begin{align*}
	A(\bl):=\prod_{i\in[d]}     A(l_i) \quad\text{with}\quad
	A(l_i):= \begin{cases} 2^{-3l_i}  &,\, l_i\ge 0 \,, \\ 2^{l_i} &,\, l_i< 0 \,. \end{cases}
\end{align*}
Further, we put $u:=\min\{1,p,q\}$.
Then, with the $u$-triangle inequality in the last estimation step,
 \begin{align*}
 X &:=
    \Big( \sum_{\bj\in\N^{d}_{0}}  2^{\bj_+ rq}  \Big\|       \sum_{\bl,\bk\in\Z^d}  
    \big|\mu_{\bj+\bl,\bk} (\Lambda_{\bj} \ast \psi_{\bj+\bl,\bk})\big|  \Big\|_{L_p(\R^d)}^q \Big)^{u/q} \\ 
    &\lesssim  \Big( \sum_{\bj\in\N^{d}_{0}}  2^{\bj_+ rq}   \Big\|    \sum_{\bl\in\Z^d} A(\bl)    
     \Big( \sum_{\bk\in\Z^d}  |\mu_{\bj+\bl,\bk}| \chi_{S_{\bj+\bl,\bk,-\bl}}  \Big)  \Big\|_{L_p(\R^d)}^{q} \Big)^{u/q}    \\
      &\le  \sum_{\bl\in\Z^d}  A(\bl)^u \Big( \sum_{\bj\in\N^{d}_{0}}   2^{\bj_+ r q}    \Big\|       \sum_{\bk\in\Z^d}   |\mu_{\bj+\bl,\bk}| \chi_{S_{\bj+\bl,\bk,-\bl}} \Big\|_{L_p(\R^d)}^{q} \Big)^{u/q}  \,.
\end{align*}

Due to $\mathcal{O}(2^{\bl_+})$ local overlaps of the function system $\big\{\chi_{S_{\bj+\bl,\bk,-\bl}}:\bk\in\Z^{d}\big\}$, we get  
\begin{align*}
    \Big( \sum_{\bk\in\Z^{d}}  |\mu_{\bj+\bl,\bk}| \chi_{S_{\bj+\bl,\bk,-\bl}}(x) \Big)^p
    \lesssim 2^{\bl_+(p-1)_+}  \sum_{\bk\in\Z^{d}}  |\mu_{\bj+\bl,\bk}|^p  \chi_{S_{\bj+\bl,\bk,-\bl}}(x) 
\end{align*}
for each fixed $x\in\R^{d}$, which yields
\begin{align*}
     \Big\|       
       \sum_{\bk\in\Z^{d}}  |\mu_{\bj+\bl,\bk}| \chi_{S_{\bj+\bl,\bk,-\bl}} \Big\|_{L_p(\R^d)}
       \lesssim 2^{\bl_+(1-\frac{1}{p})_+} \Big( \sum_{\bk\in\Z^{d}}  |  \mu_{\bj+\bl,\bk}|^p \|  \chi_{S_{\bj+\bl,\bk,-\bl}}\|_{L_1(\R^d)} \Big)^{1/p}
       \lesssim  2^{\bl_+(1-\frac{1}{p})_+} 2^{-\bj_+/p} \Big( \sum_{\bk\in\Z^{d}}  | \mu_{\bj+\bl,\bk}|^p  \Big)^{1/p} \,.
\end{align*}

Noting $2^{\bj_+}=2^{\bj}$ for $\bj\in\N^{d}_{0}$, we obtain 
\begin{align*}
    X
    &\lesssim \sum_{\bl\in\Z^{d}}  A(\bl)^u 2^{u\bl_{+}(1-\frac{1}{p})_+} \Big( \sum_{\bj\in\N^{d}_{0}}    2^{\bj_+(r-\frac{1}{p})q} \Big(  \sum_{\bk\in\Z^{d}} | \mu_{\bj+\bl,\bk}|^p \Big)^{q/p}   \Big)^{1/\tq}  \\ 
     &= \sum_{\bl\in\Z^{d}}  A(\bl)^u 2^{u\bl_{+}(1-\frac{1}{p})_+} \Big( \sum_{\bj\in\N^{d}_{0}}  2^{-\bl(r-\frac{1}{p})q}   2^{(\bj+\bl)(r-\frac{1}{p})q} \Big(  \sum_{\bk\in\Z^{d}} | \mu_{\bj+\bl,\bk}|^p \Big)^{q/p}   \Big)^{1/\tq}  \\ 
    &\lesssim  \sum_{\bl\in\Z^{d}}  A(\bl)^u 2^{u\bl_{+}(1-\frac{1}{p})_+}  2^{-u\bl(r-\frac{1}{p})}  
    \Big( \sum_{\bj\in\N^{d}_{-1}}   2^{\bj_+(r-\frac{1}{p})q}  \Big(  \sum_{\bk\in\Z^{d}} |\mu_{\bj,\bk}|^p \Big)^{q/p}    \Big)^{1/\tq} \,.
    \end{align*}
    
Since $r<1+\frac{1}{p}$ and $r>\sigma_p-2$, with $\sigma_p=(\frac{1}{p}-1)_+=(1-\frac{1}{p})_+-(1-\frac{1}{p})$,    
\begin{align*}
     \sum_{\bl\in\Z^{d}} \Big( A(\bl) 2^{\bl_{+}(1-\frac{1}{p})_+}  2^{-\bl(r-\frac{1}{p})}  \Big)^u = \sum_{\bl\in\Z^{d}} \Big( \prod_{i=1}^{d} 2^{(1+\frac{1}{p}-r)l^-_i}  \cdot 2^{((1-\frac{1}{p})_+-(1-\frac{1}{p})-2-r)l^+_i} \Big)^u <\infty \,.
\end{align*}
We thus obtain $X \lesssim \big\| \{\mu_{\bj,\bk}\}_{\bj,\bk} \big\|^{u}_{s^{r}_{p,q}b(\Gamma^d)}<\infty$, which implies that the right-most and the second  right-most sum in~\eqref{eq:weakstarconv2} converge pointwise almost everywhere, in an absolute sense. Since the pointwise limit function represents the Littlewood-Paley block $\Lambda_{\bj} \ast f$,
we finally conclude
\begin{align*}
\|f\|_{S^{r}_{p,q}B(\R^d)}^{u} \lesssim X \lesssim \big\| \{\mu_{\bj,\bk}\}_{\bj,\bk} \big\|^{u}_{s^{r}_{p,q}b(\Gamma^d)} \,.
\end{align*}
\noindent
{\bf Part (iii)\,[}{\it Unconditional strong convergence}{\bf]:}\:
Let us now handle the strong convergence of~\eqref{eqdef:f_synthesis} when $\max\{p,q\}<\infty$. In this case, the finite sequences are dense in $s^{r}_{p,q}b(\Gamma^d)$. Let us fix some enumeration $\gamma:\N \to \Gamma^d$ of the index set $\Gamma^d$
and define
\[
f_N:= \sum_{n=1}^N \mu_{\gamma(n)} \psi_{\gamma(n)} \;,\quad N\in\N \,.
\]
According to Part~(ii), we have the estimate~\eqref{eqref:syn_estimate} for all $f_N$. Let $M,N\in\N$ with $M\le N$ and define
a sequence $\mu^{(M,N)}_{\bj,\bk}$ with $(\bj,\bk)\in\Gamma^d$ by
\begin{align*}
    \mu^{(M,N)}_{\bj,\bk}  :=\begin{cases} \mu_{\bj,\bk}  &,\, (\bj,\bk)\in \{\gamma(M),\ldots,\gamma(N)\}  \,, \\ 
   0 &,\,(\bj,\bk)\notin \{\gamma(M),\ldots,\gamma(N)\}  \,. \end{cases}
\end{align*}
Then, by Part~(ii), 
\begin{align*}
    \| f_N - f_M \|_{S^{r}_{p,q}B(\R^{d})}  = \Big\|\sum_{(\bj,\bk)\in\Gamma^d} \mu^{(M,N)}_{\bj,\bk} \psi_{\bj,\bk} \Big\|_{S^{r}_{p,q}B(\R^{d})}  \lesssim  \big\|  \big\{\mu^{(M,N)}_{\bj,\bk}\big\}_{\bj,\bk}   \big\|_{s^{r}_{p,q}b(\Gamma^d)}  
\end{align*}
and, since finite sequences are dense in $s^{r}_{p,q}b(\Gamma^d)$, we have further
\[
 \big\|  \big\{\mu^{(M,N)}_{\bj,\bk}\big\}_{\bj,\bk}   \big\|_{s^{r}_{p,q}b(\Gamma^d)}  \to 0 \;,\quad M,N\to \infty \,.
\]
This shows that $\{f_N\}_{N\in\N}$ is a Cauchy sequence and therefore 
convergent in $S^{r}_{p,q}B(\R^{d})$. Consequently, since the enumeration $\gamma$ can be chosen arbitrarily, the series~\eqref{eqdef:f_synthesis} converges unconditionally in $S^{r}_{p,q}B(\R^{d})$. The strong limit thereby coincides with the weak* limit $f$.  

\noindent
{\bf Part (iv)\,[}{\it Weak convergence in $S^{r}_{p,q}B(\R^{d})$ with respect to $S^{r'}_{p',q'}B(\R^{d})$}{\bf]:}\:
Lastly, we prove the weak convergence of~\eqref{eqdef:f_synthesis} to $f$ with respect to $S^{r'}_{p',q'}B(\R^{d})$. This notion of convergence is stronger than weak* convergence in $\mathcal{S}^\prime(\R^{d})$. Let $f\in S^{r}_{p,q}B(\R^{d})$ be the weak* limit of~\eqref{eqdef:f_synthesis}. Using the abbreviations
$\mathscr{F}:=S^{r}_{p,q}B(\R^{d})$ and $\tilde{\mathscr{F}}:=S^{r'}_{p',q'}B(\R^{d})$, the product 
\begin{align*}
\langle f, \xi \rangle_{\mathscr{F}\times\tilde{\mathscr{F}}} 
= \sum_{\bl\in\Z^{d}} \big\langle  \Lambda_{\bl} \ast f \:,\,   \Phi^{\boldsymbol{-}}_{\bl} \ast \xi \big\rangle_{\R^d}
= \sum_{\bl\in\Z^{d}} \Big\langle   \sum_{(\bj,\bk)\in\Gamma^d}  \mu_{\bj,\bk} (\Lambda_{\bl} \ast\psi_{\bj,\bk}) \:,\,   \Phi^{\boldsymbol{-}}_{\bl} \ast \xi \Big\rangle_{\R^d} 
\end{align*}
is well-defined for any test function $\xi\in\tilde{\mathscr{F}}$, where $\{\Lambda_{\bl}\}_{\bl}$ and 
$\{\Phi_{\bl}\}_{\bl}$ are the kernel families from~\eqref{eqdef:gen_dual_prod}.
We know from Part~(ii) that $\sum_{(\bj,\bk)\in\Gamma^d}  \mu_{\bj,\bk} (\Lambda_{\bl} \ast\psi_{\bj,\bk})$ is a pointwise representation of $\Lambda_{\bl} \ast f$ and converges absolutely almost everywhere.  

As an auxiliary tool, let $\{\varphi_{\bm}\}_{\bm\in \N^d_0}$ be a hyperbolic decomposition of unity on $\R^d$. Then we can expand
\begin{align}
\langle f, \xi \rangle_{\mathscr{F}\times\tilde{\mathscr{F}}}
&= \sum_{\bl\in\Z^{d}} \sum_{\bm\in \N^d_0}   \Big\langle   \sum_{(\bj,\bk)\in\Gamma^d}  \mu_{\bj,\bk} (\Lambda_{\bl} \ast\psi_{\bj,\bk}) \:,\,  \varphi_{\bm}\cdot(\Phi^{\boldsymbol{-}}_{\bl} \ast \xi) \Big\rangle_{\mathcal{S}^\prime\times\mathcal{S}} \nonumber \\
&= \sum_{\bl\in\Z^{d}} \sum_{\bm\in \N^d_0} \sum_{(\bj,\bk)\in\Gamma^d}  \mu_{\bj,\bk}  \big\langle     \Lambda_{\bl} \ast\psi_{\bj,\bk} \:,\,  \varphi_{\bm}\cdot(\Phi^{\boldsymbol{-}}_{\bl} \ast \xi) \big\rangle_{\mathcal{S}^\prime\times\mathcal{S}} \label{beta_estimate}\,, 
\end{align}
with the second equality due to $\varphi_{\bm}\cdot(\Phi^{\boldsymbol{-}}_{\bl} \ast \xi) \in\mathcal{S}(\R^d)$
and weak* convergence of $\sum_{(\bj,\bk)\in\Gamma^d}  \mu_{\bj,\bk} (\Lambda_{\bl} \ast\psi_{\bj,\bk})$.

Next, we analyze the term 
\begin{align*}
   \sum_{\bl\in\Z^{d}} \sum_{\bm\in \N^d_0}  \sum_{(\bj,\bk)\in\Gamma^d}  |\mu_{\bj,\bk}| \big\langle     |\Lambda_{\bl} \ast\psi_{\bj,\bk}| \:,\,  \varphi_{\bm}\cdot |\Phi^{\boldsymbol{-}}_{\bl} \ast \xi| \big\rangle_{\R^d}   
  =  \sum_{(\bj,\bk)\in\Gamma^d}  |\mu_{\bj,\bk}| \sum_{\bl\in\Z^{d}} \big\langle     |\Lambda_{\bl} \ast\psi_{\bj,\bk}| \:,\,   |\Phi^{\boldsymbol{-}}_{\bl} \ast \xi| \big\rangle_{\R^d} \,.
\end{align*}
The quantity
\begin{align*}
\beta_{\bj,\bk}(\xi) := \sum_{\bl\in\Z^{d}} 2^{\bj_+} \big\langle     |\Lambda_{\bl} \ast\psi_{\bj,\bk}| \:,\,   |\Phi^{\boldsymbol{-}}_{\bl} \ast \xi| \big\rangle_{\R^d}
\end{align*}
is similar to $\alpha_{\bj,\bk}(\xi)$ from~\eqref{def:alpha_jk}.
In the proof of Proposition~\ref{prop:analysis_CW} we obtained the estimate \eqref{stronger_analysis}. Analogously, we get
$\| \{\beta_{\bj,\bk}(\xi)\}_{\bj,\bk}  \|_{s^{r'}_{p',q'}b(\Gamma^d)} 
\lesssim \|\xi\|_{S^{r'}_{p',q'}B(\R^{d})}$ since interchanging $\Phi_{\bl}$ by $\Phi^{\boldsymbol{-}}_{\bl}$ does not have an effect on the argumentation. This estimate then implies $\{2^{-\bj_+}\beta_{\bj,\bk}(\xi)\}_{\bj,\bk}\in s^{r'+1}_{p',q'}b(\Gamma^d)$
and Lemma~\ref{lem:Hölder_discr} yields
\begin{align*}
\sum_{(\bj,\bk)\in\Gamma^d} 2^{-\bj_+}|\mu_{\bj,\bk}|  \beta_{\bj,\bk}(\xi) \le  \|\{\mu_{\bj,\bk} \}_{\bj,\bk}  \|_{s^{r}_{p,q}b(\Gamma^d)} \cdot \| \{2^{-\bj_+}\beta_{\bj,\bk}(\xi)\}_{\bj,\bk}  \|_{s^{r'+1}_{p',q'}b(\Gamma^d)} < \infty \,.
\end{align*}
Hence, we may reorder the sum~\eqref{beta_estimate} and proceed
\begin{align*}
\langle f, \xi \rangle_{\mathscr{F}\times\tilde{\mathscr{F}}} 
&= \sum_{(\bj,\bk)\in\Gamma^d} \mu_{\bj,\bk} \sum_{\bl\in\Z^{d}}     \sum_{\bm\in \N^d_0}   \big\langle     \Lambda_{\bl} \ast\psi_{\bj,\bk} \:,\,  \varphi_{\bm}\cdot(\Phi^{\boldsymbol{-}}_{\bl} \ast \xi) \big\rangle_{\R^d} \\
&=    \sum_{(\bj,\bk)\in\Gamma^d}  \mu_{\bj,\bk} \sum_{\bl\in\Z^{d}} \big\langle     \Lambda_{\bl} \ast\psi_{\bj,\bk} \:,\,  \Phi^{\boldsymbol{-}}_{\bl} \ast \xi \big\rangle_{\R^d} 
=    \sum_{(\bj,\bk)\in\Gamma^d}  \mu_{\bj,\bk} \langle \psi_{\bj,\bk} , \xi \rangle_{\mathscr{F}\times\tilde{\mathscr{F}}}\,. 
\end{align*}
As a consequence, the limit $f$ of~\eqref{eqdef:f_synthesis} is characterized as the unique element $f\in \mathscr{F}$ with
\begin{align*}
\langle f, \xi \rangle_{\mathscr{F}\times\tilde{\mathscr{F}}}
= \sum_{(\bj,\bk)\in\Gamma^d}  \mu_{\bj,\bk} \langle \psi_{\bj,\bk} , \xi \rangle_{\mathscr{F}\times\tilde{\mathscr{F}}} \quad\text{for all }\xi\in\tilde{\mathscr{F}} \,,
\end{align*}
which means weak convergence to $f$ in $\mathscr{F}$ with respect to $\tilde{\mathscr{F}}$.
The proof is finished.
\end{proof}

\noindent
Similar as in Corollary~\ref{cor:analysis_CW}, we can deduce from Proposition~\ref{prop:synthesis_CW} a dual result.

\begin{corollary}\label{cor:synthesis_CW}
The statements of Proposition~\ref{prop:synthesis_CW} still hold true, when the primal Chui-Wang wavelets are replaced by the dual Chui-Wang wavelets. In particular,
\begin{align*}
 \Big\| \sum_{(\bj,\bk)\in\Gamma^{d}} \mu_{\bj,\bk} \tilde{\psi}_{\bj,\bk} \Big\|_{S^{r}_{p,q}B(\R^{d})}  \lesssim 
\big\| \{\mu_{\bj,\bk}\}_{\bj,\bk} \big\|_{s^{r}_{p,q}b(\Gamma^d)} \,.
\end{align*}
\end{corollary}
\begin{proof}
We use representation~\eqref{dualCW_representation} in terms of the primal wavelets. For $\beps\in\{-1,0\}^d$ we define
\begin{align*}
\Gamma^d_{\beps}:=\Big(I_{\eps_1}\times \cdots \times I_{\eps_d} \Big)\times \Z^d 
\quad\text{with}\quad  I_{\eps}:=\begin{cases} \N_0 \quad&,\, \eps=0\,, \\
 \{-1\} &,\, \eps=-1 \,.\end{cases} 
\end{align*}
Using the $u$-triangle inequality for $u:=\min\{1,p,q\}$ and Proposition~\ref{prop:synthesis_CW}, we then estimate
\begin{align*}
 \Big\|\sum_{(\bj,\bk)\in\Gamma^{d}} \mu_{\bj,\bk} \tilde{\psi}_{\bj,\bk} \Big\|^{u}_{S^{r}_{p,q}B(\R^{d})}
 &= \Big\| \sum_{\beps\in\{-1,0\}^d} \sum_{\bn\in\Z^{d}}  a^{(\beps)}_{\bn}\sum_{(\bj,\bk)\in\Gamma^{d}_{\beps}}   \mu_{\bj,\bk} \psi_{\bj,\bk+\bn} \Big\|^{u}_{S^{r}_{p,q}B(\R^{d})} \\
 &\le \sum_{\beps\in\{-1,0\}^d} \sum_{\bn\in\Z^{d}} |a^{(\beps)}_{\bn}|^{u}  \cdot\Big\|\sum_{(\bj,\bk)\in\Gamma^{d}_{\beps}}  \mu_{\bj,\bk}  \psi_{\bj,\bk+\bn} \Big\|^{u}_{S^{r}_{p,q}B(\R^{d})} \\
 &\lesssim \sum_{\beps\in\{-1,0\}^d} \sum_{\bn\in\Z^{d}} |a^{(\beps)}_{\bn}|^{u} \cdot \big\| \{\mu_{\bj,\bk}\}_{\bj,\bk} \big\|^u_{s^{r}_{p,q}b(\Gamma^d)}  \,.
\end{align*}
Now the statement follows from $\sum_{\bn\in\Z^{d}} |a^{(\beps)}_{\bn}|^{u} <\infty$ for every $\beps\in\{-1,0\}^d$.
\end{proof}



\noindent
We close this section with the proof of Theorem~\ref{thm:CW_characterization}.

\begin{proof}[\bf Proof of Theorem~\ref{thm:CW_characterization}]
By Proposition~\ref{prop:synthesis_CW} the synthesis operator
\[
S: s^{r}_{p,q}b(\Gamma^d) \to S^{r}_{p,q}B(\R^d) \;,\quad \big\{\mu_{\bj,\bk}\big\}_{\bj,\bk} \mapsto \sum_{(\bj,\bk)\in\Gamma^d} \mu_{\bj,\bk} \psi_{\bj,\bk}
\]
is bounded. By Corollary~\ref{cor:analysis_CW} the analysis operator
\[
\widetilde{\Lambda}: S^{r}_{p,q}B(\R^d) \to s^{r}_{p,q}b(\Gamma^d) \;,\quad f \mapsto  \big\{2^{\bj_+} \dual{f,\tilde{\psi}_{\bj,\bk}}\big\}_{\bj,\bk}
\]
is bounded. We want to show
\begin{align}\label{iso_relation}
S\circ\widetilde{\Lambda} = \operatorname{Id}_{S^{r}_{p,q}B(\R^d)} \quad,\quad \widetilde{\Lambda}\circ S = \operatorname{Id}_{s^{r}_{p,q}b(\Gamma^d)}  \,,
\end{align}
which implies that both $S$ and $\widetilde{\Lambda}$ are isomorphisms. In the proof, let us use the abbreviations
$\mathscr{F}:=S^{r}_{p,q}B(\R^d)$,
$\tilde{\mathscr{F}}:=S^{r'}_{p',q'}B(\R^d)$, $\tilde{\tilde{\mathscr{F}}}:=S^{r^{\prime\prime}}_{p^{\prime\prime},q^{\prime\prime}}B(\R^d)$ and recall~\eqref{eq:conj_bidual}, i.e.\ $\mathscr{F}\hookrightarrow \tilde{\tilde{\mathscr{F}}}$.  

We first note that relation~\eqref{iso_relation} is certainly true on $L_2(\R^d)$ resp.\ $\ell_2(\Gamma^d)$,
since the primal and dual Chui-Wang systems form a pair of biorthogonal Riesz bases in $L_2(\R^d)$ (in the sense of~\eqref{rel_biortho_multi}).
Taking especially $\varphi\in\mathcal{S}(\R^d)\hookrightarrow L_2(\R^d)$, one obtains the $L_2$ reproducing formula  
\[
\varphi 
= \sum_{(\bj,\bk)\in\Gamma^d} 2^{\bj_+} \langle \varphi, \psi_{\bj,\bk}\rangle_{\R^d} \tilde{\psi}_{\bj,\bk}
\qquad\text{for all }\varphi\in\mathcal{S}(\R^d) \,. 
\]
Noting $\mathcal{S}(\R^d)\hookrightarrow \tilde{\mathscr{F}}$, the convergence above is also weakly in $\tilde{\mathscr{F}}$ with respect to $\tilde{\tilde{\mathscr{F}}}$. 
This follows from Proposition~\ref{prop:analysis_CW} and Corollary~\ref{cor:synthesis_CW}. 
As a consequence, for any $f\in\mathscr{F}\hookrightarrow \tilde{\tilde{\mathscr{F}}}$,
\begin{align}\label{first_consequence}
 \langle f,\varphi \rangle_{\mathscr{F}\times\tilde{\mathscr{F}}}
=\Big\langle f,  \sum_{(\bj,\bk)\in\Gamma^d} 2^{|\bj|_1} \langle \varphi,\psi_{\bj,\bk}\rangle_{\R^d} \cdot\tilde{\psi}_{\bj,\bk} \Big\rangle_{\mathscr{F}\times\tilde{\mathscr{F}}} =
\sum_{(\bj,\bk)\in\Gamma^d} 2^{|\bj|_1} \langle \varphi, \psi_{\bj,\bk} \rangle_{\R^d} \cdot \langle f,\tilde{\psi}_{\bj,\bk} \rangle_{\mathscr{F}\times\tilde{\mathscr{F}}} \,.
\end{align}

Further, with $f\in\mathscr{F}$ as before, $\widetilde{\Lambda}(f) \in s^{r}_{p,q}b(\Gamma^d)$ by Corollary~\ref{cor:synthesis_CW} and, according to Proposition~\ref{prop:analysis_CW}, the synthesis converges weakly to some $g\in \mathscr{F}$ with test space $\tilde{\mathscr{F}}$. 
To verify the equality $f=g$ it suffices to test with functions $\varphi\in\mathcal{S}(\R^d)$. In view of~\eqref{first_consequence}, a calculation indeed yields
\[
\langle g,\varphi\rangle_{\mathscr{F}\times\tilde{\mathscr{F}}} = \Big\langle  \sum_{(\bj,\bk)\in\Gamma^d} 2^{|\bj|_1} \dual{f,\tilde{\psi}_{\bj,\bk}} \cdot\psi_{\bj,\bk} , \varphi \Big\rangle_{\mathscr{F}\times\tilde{\mathscr{F}}} = \sum_{(\bj,\bk)\in\Gamma^d} 2^{|\bj|_1}  \dual{f,\tilde{\psi}_{\bj,\bk}} \cdot\langle  \psi_{\bj,\bk} , \varphi\rangle_{\mathscr{F}\times\tilde{\mathscr{F}}}
=   \langle f,\varphi \rangle_{\mathscr{F}\times\tilde{\mathscr{F}}} \,.
\]
We have thus proven the primal expansion $f=S(\widetilde{\Lambda}(f))$, which implies that the operator $\widetilde{\Lambda}$ must be bounded from below. By Corollary~\ref{cor:synthesis_CW}, it is also bounded from above. Consequently,
\begin{align*}
    \|f\|_{S^{r}_{p,q}B(\R^{d})} \asymp \big\| \widetilde{\Lambda}(f)\big\|_{s^{r}_{p,q}b(\Gamma^d)} 
    = \big\| \big\{ 2^{|\bj|_1} \dual{f,\tilde{\psi}_{\bj,\bk}} \big\}_{\bj,\bk} \big\|_{s^{r}_{p,q}b(\Gamma^d)}  \,.
\end{align*}
To prove the relation $\widetilde{\Lambda}\circ S = \operatorname{Id}_{s^{r}_{p,q}b(\Gamma^d)}$, let $\{\mu_{\bj,\bk}\}_{\bj,\bk}$ be a sequence from $s^{r}_{p,q}b(\Gamma^d)$. 
Using the weak convergence of $\sum_{(\bj,\bk)\in\Gamma^d}  \mu_{\bj,\bk} \psi_{\bj,\bk}$ in $\mathscr{F}$ with respect to  $\tilde{\mathscr{F}}$, we obtain
\begin{align*}
\Big\langle \sum_{(\bj,\bk)\in\Gamma^d}  \mu_{\bj,\bk} \psi_{\bj,\bk},\tilde{\psi}_{\bj^*,\bk^*} \Big\rangle_{\mathscr{F}\times\tilde{\mathscr{F}}}
= \sum_{(\bj,\bk)\in\Gamma^d}  \mu_{\bj,\bk}    \langle    \psi_{\bj,\bk} \:,\,  \tilde{\psi}_{\bj^\ast,\bk^\ast}  \rangle_{\mathscr{F}\times\tilde{\mathscr{F}}}
= \mu_{\bj^*,\bk^*} \,.
\end{align*}
This finishes the proof of~\eqref{iso_relation}.
The dual statements for the synthesis operator $\widetilde{S}$ and the analysis operator $\Lambda$, where the primal and dual roles of the wavelets are switched, follow analogously.
\end{proof}

\section*{Acknowledgement}
The authors would like to thank Dirk Nuyens for inspiring discussions which helped
to shape the concept of this paper. They would also like to thank Kai Lüttgen
for his valuable contributions to the presented work.


\newpage

\appendix

\section{Some background on Besov spaces}

A more general characterization of the space $S^{r}_{p,q}B(\R^{d})$ as in Definition~\ref{def:domix_Besov_Rd} is given in Theorem~\ref{thm:char_localmeans}.

\begin{definition}\label{def:localmeansfam}
We call a family $\{\Psi_{j}\}_{j\in\N_{0}}$ of functions in $\mathcal{S}(\R)$ a \emph{local means family on $\R$} with $L\in\N$ vanishing moments if the following criteria are met:
	\begin{itemize}
            \item[(i)] $\Psi_j(x) = 2^{j-1}\Psi_1(2^{j-1}x)$ for $j\ge 2$.
		\item[(ii)] There exists $\eps > 0$ such that $|\widehat{\Psi}_{0}(\xi)| > 0$ if $|\xi| < \eps$ and $|\widehat{\Psi}_{1}(\xi)|  > 0$ if $\eps/2 < |\xi| < 2\eps$.
		\item[(iii)] $D^{\alpha}\widehat{\Psi}_1(0)=0$ for $0\le\alpha<L$.
	\end{itemize}
Further, suppose $\{\Psi_{j}^{(i)}\}_{j\in\N_{0}}$ are local means families on $\R$ for every $i\in[d]$, each with $L$ vanishing moments. Then the family $\{\Psi_{\bj}\}_{\bj\in\N_{0}^{d}}$ with $\Psi_{\bj} := \Psi_{j_{1}}^{(1)} \otimes \ldots \otimes \Psi_{j_{d}}^{(d)}$ is called a \emph{local means family on $\R^{d}$} with $L$ vanishing moments.
\end{definition}

\noindent
With the local means families defined above the following characterization of $S^{r}_{p,q}B(\R^{d})$ is possible.

\begin{theorem}[{\cite[Thm~1.23]{Vy06}}]\label{thm:char_localmeans}
Let $0<p,q\le\infty$, $r\in\R$, and $\{\Psi_{\bj}\}_{\bj\in \N^d_{0}}$ be a local means family on $\R^d$
with $L\in\N$ vanishing moments. In case $L>r$, we have
\begin{align*}
\|f\|_{S^{r}_{p,q}B(\R^{d})} \asymp \begin{cases}
			\displaystyle \Big( \sum_{\bj \in \N_{0}^{d}} 2^{r q |\bj|_{1}} \norm{\Psi_{\bj} \ast f}_{L^{p}(\R^{d})}^{q} \Big)^{1/q} \; , \quad &q < \infty \; , \\[4ex]
			\displaystyle \sup_{\bj \in \N_{0}^{d}} 2^{r |\bj|_{1}} \norm{\Psi_{\bj} \ast f}_{L^{p}(\R^{d})}  \; , \quad &q = \infty \; .
   \end{cases}
\end{align*}
\end{theorem}

\noindent
A short history of local mean characterizations is given in~\cite[Rem.~4.6]{UU2015}. They first appeared in \cite{Tr06} for the isotropic spaces $B^{r}_{p,q}(\R^d)$.
The dominating mixed setting was then first handled in~\cite{Vy06}.

\subsection{Peetre maximal inequality} 
\label{ssec:Appendix_Peetre}

\noindent 
Let $f \in L_1(\R^{d})$ be a function such that $\widehat{f}$ is compactly
supported. For $\bar{b}=(b_1,\ldots,b_d)>0$ and $a>0$, the Peetre maximal function $P_{\bar{b},a}f$ is defined by (cf.~\cite[eq.~(1.17)]{Vy06})
\begin{equation*}
  P_{\bar{b},a}f(x) := \sup\limits_{z\in \R^{d}} \frac{|f(x-z)|}{(1+|b_1z_1|)^a\cdots (1+|b_dz_d|)^a}
  \,.  
\end{equation*}
This construction is used for instance in~\cite{Schmeisser1987} and \cite{Vy06}.
Originally, maximal functions of this kind are due to Peetre~\cite{Peetre75} and Fefferman and Stein~\cite{FeffStein72}.
An important relation between $f$ and $P_{\bar{b},a}f$ is the Peetre maximal inequality, stated in the following theorem.

\begin{theorem}[{\cite[1.6.4]{Schmeisser1987}}]\label{peetremax} 
Let $0<p\le\infty$, $\bar{b}>0$, and $a>\tfrac{1}{p}$. For a bandlimited function $f\in L_1(\R^d)$ with
            $$
                      \supp(\widehat{f}) \subset [-b_1,b_1]\times\cdots \times [-b_d,b_d]
            $$
            the following inequality holds
            $$
         \|P_{\bar{b},a}f \|_{L_p(\R^d)} \leq C \|f\|_{L_p(\R^d)} \,,
            $$
           where $C>0$ is some constant independent of $f$ and $\bar{b}$.
 \end{theorem}

\subsection{Embeddings}

Some fundamental embeddings for 
$S^r_{p,q}B(\R^d)$ with $d\in\N$ have been collected in the following proposition.

\begin{prop}[{cf.~\cite[Sec.~2.2.3 \&~2.4.1]{Schmeisser1987},~\cite[Lem.~3.2.1 \&~3.2.2]{TDiff06}}]
Let $0<p,q,u,v\le\infty$ and $r,w\in \R$.
\begin{description}
\item(i) If $r>1/p$ then
 $$
    S^r_{p,q}B(\R^d) \hookrightarrow C(\R^d)\,.
 $$
 \item(ii) 
 If $v\ge q$ and $w\le r$ then
 $$
     S^r_{p,q}B(\R^d) \hookrightarrow S^w_{p,v}B(\R^d)\,.
 $$
   \item(iii) If $\varepsilon>0$ and $p\le u\le [(1/p-\varepsilon)_+]^{-1}$ (with $0^{-1}=\infty$) then
 $$
   S^{r+\varepsilon}_{p,q}B(\R^d) \hookrightarrow S^r_{p,v}B(\R^d) \quad\text{and}\quad S^{r+\varepsilon}_{p,q}B(\R^d) \hookrightarrow S^r_{u,q}B(\R^d)\,.
 $$
\end{description}
\end{prop}

\noindent
The above embeddings carry over to the spaces $S^r_{p,q}B(D)$ on arbitrary domains $D\subset\R^d$ and also hold true in case $D=\T^d$ (see e.g.~\cite[Rem.~4.2.1]{TDiff06}).
If $D$ is bounded, an additional elementary embedding is given below. 

\begin{prop}\label{prop:emb} Let $0<p,q\le\infty$, $r\in \R$, and $D\subset\R^d$ be bounded or $D=\T^d$. If $0<u\le p$ then
 $$
   S^{r}_{p,q}B(D) \hookrightarrow S^r_{u,q}B(D) \,.
 $$
\end{prop}

\noindent
We finally note that $S^{r}_{p,q}B(\R^{d})$ is comprised entirely of regular distributions if the smoothness parameter $r\in\R$ is larger than $\sigma_p$ as defined in~\eqref{eqdef:sigma_p}. For a proof we refer to~\cite[Prop.~3(2.2.3)]{Schmeisser1987}.

\begin{prop}[{\cite[Lem.~3.2.2]{TDiff06}}]\label{prop:embed_regular}
Let $0<p,q\le\infty$ and $\tilde{p}:=\max\{p,1\}$. If $r > \sigma_p:=(\tfrac{1}{p} - 1)_+$ then
$$
S^r_{p,q}B(\R^d)  \hookrightarrow L_p(\R^d) \cap L_{\tilde{p}}(\R^d)  \,.
$$
\end{prop}

\subsection{Calder\'on reproducing formula}
\label{ssecA:Calderon}

Let us assume that $\{\Lambda_{l}\}_{l\in\N_{0}}$ is a local means family on $\R$ with $L\in\N$ vanishing moments. The functions $\Lambda_0$ and $\Lambda_1$ thus fulfill the Tauberian condition (condition~$(ii)$) in Definition~\ref{def:localmeansfam} for some $\varepsilon>0$. 
According to the construction in~\cite[Sec.~3.3]{UU2015}, there then exists a family of associated kernels $\{\Phi_{l}\}_{l\in\N_{0}}\subset\mathcal{S}(\R)$ such that  $\supp(\widehat{\Phi}_{0})\subset \{\xi \in \R : \abs{\xi} \leq \varepsilon \}$,  $\supp (\widehat{\Phi}_{1}) \subset \{ \xi \in \R : \varepsilon/2 \leq \abs{\xi} \leq 2\varepsilon \}$, for all $\xi\in\R$
\begin{align*}
  \sum_{l=0}^\infty \widehat{\Lambda}_l(\xi) \widehat{\Phi}_l(\xi) = 1 \,,
\end{align*}
and, additionally, $\Phi_l(x) = 2^{l-1}\Phi_1(2^{l-1}x)$ for $x\in\R$ and $l\ge 2$.

Defining 
$\Lambda_{\bl} := \Lambda_{l_{1}} \otimes \ldots \otimes \Lambda_{l_{d}}$
and $\Phi_{\bl} := \Phi_{l_{1}} \otimes \ldots \otimes \Phi_{l_{d}}$ for $\bl=(l_1,\ldots,l_d)\in\N_{0}^d$
one obtains a local means family $\{\Lambda_{\bl}\}_{\bl\in\N_{0}^{d}}$ on $\R^d$
with $L$ vanishing moments and an associated family  $\{\Phi_{\bl}\}_{\bl\in\N_{0}^{d}}$.
Let us put $\Delta_{\bl}:=\Lambda_{\bl}\ast\Phi_{\bl}$. 
The family  $\{\delta_{\bl}\}_{\bl\in\N_{0}^{d}}$ with
$\delta_{\bl}:=\widehat{\Delta}_{\bl}=\widehat{\Lambda}_{\bl}\cdot\widehat{\Phi}_{\bl}$
is then a hyperbolic decomposition of unity on $\R^d$. For a distribution $f\in\mathcal{S}^\prime(\R^{d})$, the associated Littlewood-Paley decomposition can hence be expanded into
\begin{align}\label{def:Calderon}
    f=\sum_{\bl\in\N_{0}^{d}} \Delta_{\bl} \ast f = \sum_{\bl\in\N_{0}^{d}}  \Lambda_{\bl} \ast\Phi_{\bl}\ast f  \,.
\end{align}
This expansion is known as discrete Calder\'on reproducing formula (cf.~\cite{Calderon77}) and at least weak* convergent in $\mathcal{S}^\prime(\R^{d})$.
For Schwartz functions $f$ it converges strongly in $S(\R^{d})$ (the isotropic case is proved e.g.\ in~\cite[Thm.~1.24 \&~Thm.~1.35]{Sawano18}). If $f\in L_2(\R^d)$ one has $L_2$ convergence.
The weak* convergence of~\eqref{def:Calderon} implies, with $\Lambda_{\bl}^{\boldsymbol{-}}:=\Lambda_{\bl}(-\cdot)$,
\begin{align}\label{Calderon_dual_prod}
\langle  f,   g \rangle_{\mathcal{S}^\prime\times\mathcal{S}}  = 
\sum_{\bl\in\N_{0}^{d}} \langle \Delta_{\bl} \ast f,
	 g \rangle_{S^\prime\times S}
=
\sum_{\bl\in\N_{0}^{d}} \langle  \Phi_{\bl}\ast f,
	\Lambda_{\bl}^{\boldsymbol{-}}\ast g \rangle_{\R^d} \,,
\end{align}
where $\langle f,g\rangle_{\R^d}$ is the short-hand notation for the integral product of $f$ and $g$.

\subsection{Duality product}
\label{ssec:Appendix_duality}

\noindent
An important concept in the theory of Besov spaces is conjugation, which 
is a special kind of duality.

\begin{definition}\label{def:conj_space}
Let $0<p,q\le\infty$, $r\in\R$. The \emph{conjugate parameters} $1\le p^\prime,q^\prime\le \infty$, $r^\prime\in\R$, are defined by  
\begin{align}\label{eqdef:general_conj_exponents}
p^\prime:=\begin{cases} \frac{p}{p-1}  &,\, 1< p\le\infty \\ \infty &,0<p\le1 \end{cases} \,,\quad q^\prime:=\begin{cases} \frac{q}{q-1}  &,\, 1< q\le\infty \\ \infty &,0<q\le1 \end{cases}   \,,\quad r^\prime:=-r+\sigma_p \,.
\end{align}
The space $S^{r'}_{p',q'}B(\R^{d})$ is called the \emph{conjugate space} or \emph{conjugate dual} of $S^{r}_{p,q}B(\R^{d})$.
\end{definition}

\noindent
For $p$ and $p'$, and accordingly for $q$ and $q'$, we have the relations
\[
\tfrac{1}{p'}=\big(1-\tfrac{1}{p}\big)_+ \quad\text{and}\quad \tfrac{1}{p'}+\tfrac{1}{p} = 1 + \sigma_p  \,.
\]
Note further that
always $1\le p^\prime,q^\prime\le\infty$ and
$p^{\prime\prime}=\max\{1,p\}$, $q^{\prime\prime}=\max\{1,q\}$
for the bi-dual parameters $p^{\prime\prime}=(p')'$ and $q^{\prime\prime}=(q')'$. In general, the bi-dual space
$S^{r^{\prime\prime}}_{p^{\prime\prime},q^{\prime\prime}}B(\R^d)$ is hence not identical to $S^{r}_{p,q}B(\R^d)$. Calculating $r^{\prime\prime}=(r')'=-r'=r-\sigma_p$ and in view of Proposition~\ref{prop:emb}, one obtains the canonical embedding
\begin{align}\label{eq:conj_bidual}
S^{r}_{p,q}B(\R^d) \hookrightarrow S^{r^{\prime\prime}}_{p^{\prime\prime},q^{\prime\prime}}B(\R^d)
= S^{r-\sigma_p}_{\max\{p,1\},\max\{q,1\}}B(\R^d) \,.
\end{align}

\noindent
Founded on formula~\eqref{Calderon_dual_prod}, a duality product can
be defined on pairs of conjugate spaces (cf.~\cite[Prop.~3.20]{LiYaYuSaUl13}).

\begin{definition}\label{def:duality_prod}
Let $\mathscr{F}=S^{r}_{p,q}B(\R^{d})$ with $0< p,q\le\infty$, $r\in\R$, and $\tilde{\mathscr{F}}$ its conjugate space. One defines 
\begin{align}\label{eqdef_app:gen_dual_prod}
\langle f,g \rangle_{\mathscr{F}\times \tilde{\mathscr{F}}} := \sum_{\bl\in\N_{0}^{d}} \langle  \Phi_{\bl}\ast f,
	\Lambda_{\bl}^{\boldsymbol{-}}\ast g \rangle_{\R^d} \qquad\text{for $f\in \mathscr{F}$, $g\in\tilde{\mathscr{F}}$}\,.
\end{align}
\end{definition}

\noindent
Due to \eqref{Calderon_dual_prod}, since $\mathscr{F}\hookrightarrow \mathcal{S}^\prime(\R^d)$ and $\mathcal{S}(\R^d)\hookrightarrow\tilde{\mathscr{F}}$, it is clear from the definition that 
\begin{align*}
\langle f,  g \rangle_{\mathscr{F}\times\tilde{\mathscr{F}}}  
=\langle  f,   g \rangle_{\mathcal{S}^\prime\times\mathcal{S}} \qquad\text{for $f\in \mathscr{F}$, $g\in\mathcal{S}(\R^d)$} \,.
\end{align*}
The product $\langle  \cdot,  \cdot \rangle_{\mathscr{F}\times\tilde{\mathscr{F}}}$ is hence compatible
with $\langle  \cdot,  \cdot \rangle_{\mathcal{S}^\prime\times\mathcal{S}}$.
Moreover, since $\mathcal{S}(\R^d)$ is dense in $S^{r}_{p,q}B(\R^{d})$ if $\max\{p,q\}<\infty$, the conjugate dual corresponds to the usual topological dual of $S^{r}_{p,q}B(\R^{d})$ in this case.

\begin{lemma}
    The product~\eqref{eqdef_app:gen_dual_prod} is a well-defined bounded bilinear map $\langle \cdot,\cdot \rangle_{\mathscr{F}\times \tilde{\mathscr{F}}}:\mathscr{F}\times \tilde{\mathscr{F}} \to \C $.
    \end{lemma}
    \begin{proof}
    Since the sum in~\eqref{eqdef_app:gen_dual_prod} is absolutely convergent, the product is well-defined. In fact, utilizing the Hölder inequality and the embedding~\eqref{eq:conj_bidual}, one can estimate
    \begin{align*}
    |\langle f,g \rangle_{\mathscr{F}\times \tilde{\mathscr{F}}}|
   &\le  \Big\| \Big\{2^{(r-\sigma_p)|\bl|_1} \| \Phi_{\bl}\ast f \|_{L_{\max\{p,1\}}(\R^d)} \Big\}_{\bl} \Big\|_{\ell_{\max\{q,1\}}(\N_{0}^d)} \cdot  \Big\|\Big\{
           2^{(\sigma_p-r)|\bl|_1} \|\Lambda_{\bl}^{\boldsymbol{-}}\ast g \|_{L_{p^{\prime}}(\R^d)} \Big\}_{\bl} \Big\|_{\ell_{q^{\prime}}(\N_{0}^d)} \nonumber \\
   &\lesssim  \Big\| \Big\{2^{r |\bl|_1} \| \Phi_{\bl}\ast f \|_{L_p(\R^d)} \Big\}_{\bl} \Big\|_{\ell_{q}(\N_{0}^d)} \cdot  \Big\|\Big\{
           2^{r^{\prime}|\bl|_1} \|\Lambda_{\bl}^{\boldsymbol{-}}\ast g \|_{L_{p^{\prime}}(\R^d)} \Big\}_{\bl} \Big\|_{\ell_{q^{\prime}}(\N_{0}^d)} \asymp \|f\|_{\mathscr{F}} \cdot \|g\|_{\tilde{\mathscr{F}}}\,, 
    \end{align*}
where the observation $p^\prime=\max\{p,1\}^\prime$, $q^\prime=\max\{q,1\}^\prime$, and $r^\prime=\sigma_p-r$ is important. 
The boundedness of~\eqref{eqdef_app:gen_dual_prod} is a direct consequence of this estimate. 
The bilinearity of~\eqref{eqdef_app:gen_dual_prod}, at last, 
is obvious.
\end{proof}

\end{document}